\documentclass[a4paper, intlimits, reqno]{amsart}

\usepackage[english]{babel}
\usepackage[T1]{fontenc}
\usepackage[utf8]{inputenc}

\usepackage{amsmath}
\usepackage{amssymb}
\usepackage{MnSymbol}
\usepackage{amsthm}
\allowdisplaybreaks
\usepackage{amsfonts}
\usepackage{mathrsfs} 
\usepackage{mathtools}
\usepackage{nicefrac}
\usepackage{enumitem}
\usepackage{multicol}
\usepackage{url}
\usepackage{dsfont}
\usepackage[numbers]{natbib}
\usepackage{prettyref}
\usepackage{doi}

\newrefformat{def}{Definition \ref{#1}}
\newrefformat{rem}{Remark \ref{#1}}
\newrefformat{sect}{Section \ref{#1}}
\newrefformat{prop}{Proposition \ref{#1}}
\newrefformat{thm}{Theorem \ref{#1}}
\newrefformat{cor}{Corollary \ref{#1}}
\newrefformat{ex}{Example \ref{#1}}
\newrefformat{que}{Question \ref{#1}}
\newrefformat{cond}{Condition \ref{#1}}

\swapnumbers
\newtheoremstyle{dotless}{}{}{\itshape}{}{\bfseries}{}{}{}
\theoremstyle{dotless}
\theoremstyle{plain}
\newtheorem{thm}{Theorem}[section]

\newtheorem{prop}[thm]{Proposition}
\newtheorem{cor}[thm]{Corollary}
\theoremstyle{definition}
\newtheorem{defn}[thm]{Definition}
\newtheorem{rem}[thm]{Remark}
\newtheorem{exa}[thm]{Example}

\newtheorem{que}[thm]{Question}
\newtheorem{cond}[thm]{Condition}

\newcommand{\N} {\mathbb{N}}
\newcommand{\Z} {\mathbb{Z}}
\newcommand{\Q} {\mathbb{Q}}
\newcommand{\R} {\mathbb{R}}
\newcommand{\C} {\mathbb{C}}
\newcommand{\K} {\mathbb{K}}
\newcommand{\D} {\mathbb{D}}
\newcommand{\F} {\mathcal{F}(\Omega)}
\newcommand{\FE} {\mathcal{F}(\Omega,E)}
\newcommand{\FV} {\mathcal{FV}(\Omega)}
\newcommand{\FVE} {\mathcal{FV}(\Omega,E)}
\newcommand{\oacx} {\overline{\operatorname{acx}}}

\DeclareMathOperator{\id}{id}
\DeclareMathOperator{\re}{Re}
\DeclareMathOperator{\dom}{dom}
\providecommand{\differential}{\mathrm{d}}
\renewcommand{\d}{\differential}

\begin{document}

\title[Extension]{Extension of vector-valued functions and sequence space representation}
\author[K.~Kruse]{Karsten Kruse}
\address{Hamburg University of Technology\\ 
Institute of Mathematics \\
Am Schwarzenberg-Campus~3 \\
21073 Hamburg \\
Germany}
\email{karsten.kruse@tuhh.de}



\date{\today}
\begin{abstract}
 We give a unified approach to handle the problem of extending  
 functions with values in a locally convex Hausdorff space $E$ over a field $\K$, 
 which have weak extensions in a space $\mathcal{F}(\Omega,\K)$ of scalar-valued functions on a set $\Omega$, 
 to functions in a vector-valued counterpart $\FE$ of $\mathcal{F}(\Omega,\K)$. 
 The results obtained base upon a representation of vector-valued functions 
 as linear continuous operators and extend results of Bonet, Frerick, Gramsch and Jord\'{a}. 
 In particular, we apply them to obtain a sequence space representation of $\FE$ from a known representation 
 of $\mathcal{F}(\Omega,\K)$.
\end{abstract}
\maketitle
\section{Introduction}
We study the problem of extending vector-valued functions via the existence of weak extensions. 
The precise description of this problem reads as follows.
Let $E$ be a locally convex Hausdorff space over the field $\K$ of real or complex numbers and 
$\F:=\mathcal{F}(\Omega,\K)$ a locally convex Hausdorff space of 
$\K$-valued functions on a set $\Omega$. 
Suppose that the point evaluations $\delta_{x}$ belong to the dual $\F'$ for every $x\in\Omega$ and that 
there is a locally convex Hausdorff space $\FE$ of $E$-valued functions on $\Omega$ such 
that the map 
\begin{equation}\label{eq:intro}
S\colon \F\varepsilon E \to \FE,\;u\longmapsto [x\mapsto u(\delta_{x})],
\end{equation}
is a linear topological isomorphism into, i.e.\ to its range, where the space of continuous linear operators 
$\F\varepsilon E:=L_{e}(\F_{\kappa}',E)$ is Schwartz' $\varepsilon$-product.
The space $\F\varepsilon E$ can be considered as a linearisation of (a subspace of) $\FE$. 
Linearisations basing on the Dixmier-Ng theorem were used by 
Bonet, Doma\'nski and Lindstr\"{o}m in \cite[Lemma 10, p.\ 243]{BonDomLind2001} resp.\ 
Laitila and Tylli in \cite[Lemma 5.2, p.\ 14]{Laitila2006} to describe the space of weakly holomorphic resp.\ harmonic
functions on the unit disc $\Omega=\D\subset\C$ with values in a (complex) Banach space $E$.

\begin{que}\label{que:weak_strong}
Let $\Lambda$ be a subset of $\Omega$ and $G$ a linear subspace of $E'$.
Let $f\colon\Lambda\to E$ be such that for every $e'\in G$,
the function $e'\circ f\colon\Lambda\to \K$ has an extension in $\F$.
When is there an extension $F\in\FE$ of $f$, i.e.\ $F_{\mid \Lambda}=f$\ ?
\end{que}

An affirmative answer for $\Lambda=\Omega$ and $G=E'$ is called a weak-strong principle. 
For weighted continuous functions on a completely regular Hausdorff space $\Omega$ with values in a semi-Montel or Schwartz space $E$ 
a weak-strong principle is given by Bierstedt in \cite[2.10 Lemma, p.\ 140]{B2}. 
Weak-strong principles for holomorphic functions on open subsets $\Omega\subset\C$ were shown 
by Dunford in \cite[Theorem 76, p.\ 354]{Dunford1938} for Banach spaces $E$ and by
Grothendieck in \cite[Th\'{e}or\`{e}me 1, p.\ 37-38]{Grothendieck1953} for quasi-complete $E$. 
For a wider class of function spaces weak-strong principles are due to Grothendieck, mainly, in the case 
that $\F$ is nuclear and $E$ complete (see \cite[Chap.\ II, \S3, n$^\circ$3, Th\'{e}or\`{e}me 13, p.\ 80]{Gro}), 
which covers the case that $\F$ is the space $\mathcal{C}^{\infty}(\Omega)$ of smooth functions on an open set $\Omega\subset\R^{d}$ 
(with its usual topology).
 
Gramsch \cite{Gramsch1977} analized the weak-strong principles of Grothendieck and realized that they can be used to extend functions 
if $\Lambda$ is a set of uniqueness, i.e.\ from $f\in\F$ and $f(x)=0$ for all $x\in\Lambda$ follows that $f=0$, and 
$\F$ a semi-Montel space, $E$ complete and $G=E'$ (see \cite[0.1, p.\ 217]{Gramsch1977}).
An extension result for holomorphic functions where $G=E'$ and $E$ is sequentially complete 
was shown by Bogdanowicz in \cite[Corollary 3, p.\ 665]{Bogdanowicz1969}.

Grosse-Erdmann proved in \cite[5.2 Theorem, p.\ 35]{grosse-erdmann1992} for holomorphic functions on $\Lambda=\Omega$ 
that it is sufficient to test locally bounded functions $f$ with values 
in a locally complete space $E$ with functionals from a weak$^{\star}$-dense subspace $G$ of $E'$.  
Arendt and Nikolski \cite{Arendt2000}, \cite{Arendt2006} shortened his proof in the case that $E$ is a Fr\'{e}chet space (see
\cite[Theorem 3.1, p.\ 787]{Arendt2000} and \cite[Remark 3.3, p.\ 787]{Arendt2000}).
Arendt gave an affirmative answer in \cite[Theorem 5.4, p.\ 74]{Arendt2016} for harmonic functions 
on an open subset $\Lambda=\Omega\subset\R^{d}$ where the 
range space $E$ is a Banach space and $G$ a weak$^{\star}$-dense subspace of $E'$. 

In \cite{Gramsch1977} Gramsch also derived extension results for a large class of Fr\'{e}chet-Montel spaces $\F$ 
in the case that $\Lambda$ is a special set of uniqueness, 
$E$ sequentially complete and $G$ strongly dense in $E'$ (see \cite[3.3 Satz, p.\ 228-229]{Gramsch1977}). He applied it 
to the space of holomorphic functions and Grosse-Erdmann \cite{grosse-erdmann2004} expanded this result by the case of 
$E$ being $B_{r}$-complete and $G$ only a weak$^{\star}$-dense subspace of $E'$ 
(see \cite[Theorem 2, p.\ 401]{grosse-erdmann2004} and \cite[Remark 2 (a), p.\ 406]{grosse-erdmann2004}). 
In a series of papers \cite{B/F/J}, \cite{F/J}, \cite{F/J/W}, \cite{jorda2005}, \cite{jorda2013} 
these results were generalised and improved 
by Bonet, Frerick, Jord\'{a} and Wengenroth who used \eqref{eq:intro} to obtain extensions for vector-valued functions 
via extensions of linear operators. 
In \cite{jorda2005}, \cite{jorda2013} by Jord\'{a} for holomorphic functions on a domain (i.e.\ open and connected) 
$\Omega\subset\C$ and weighted holomorphic functions on a domain $\Omega$ in a Banach space.
In \cite{B/F/J} by Bonet, Frerick and Jord\'{a} for closed subsheaves $\F$ of the sheaf of 
smooth functions $\mathcal{C}^{\infty}(\Omega)$ on a domain $\Omega\subset\R^{d}$. 
Their results implied some consequences on the work of Bierstedt and Holtmanns \cite{Bierstedt2003} as well.
Further, in \cite{F/J} by Frerick and Jord\'{a} for closed subsheaves $\F$ of smooth functions on a domain $\Omega\subset\R^{d}$ 
which are closed in the sheaf $\mathcal{C}(\Omega)$ of continuous functions
and in \cite{F/J/W} by the first two authors and Wengenroth in the case that $\F$ is the space of bounded functions 
in the kernel of a hypoelliptic linear  partial differential operator, 
in particular, the spaces of bounded holomorphic or harmonic functions. 
The results of \cite{F/J/W} are not used in the present paper but will be treated separately and extended in \cite{kruse2019_3}.

In this paper we present a unified approach to the extension problem for a large class of function spaces. 
The spaces we treat are usally of the kind 
that $\F$ belongs to the class of semi-Montel or Fr\'echet-Schwartz spaces. 
Even quite general weighted spaces $\F$ are treated, at least, if $E$ is a semi-Montel space. 
The case of Banach spaces is handled in \cite{F/J/W} and \cite{kruse2019_3}.
Our approach is based on the representation of (a subspace) of $\FE$ as a space of continuous linear operators 
via the map $S$ from \eqref{eq:intro}. 
All our examples of such spaces are actually of the form of a general weighted space $\FVE$ introduced in \cite{kruse2017} which is 
generated by linear operators $T^{E}$ on a domain in $E^{\Omega}$ and equipped with a kind of graph topology 
(see \prettyref{def:standard_space}). 
Spaces of this form cover many examples of function spaces like the ones we already mentioned and
standard examples of such spaces are weighted spaces of continuously partially differentiable functions 
which are generated by the partial derivative operators. 
The key to generalise \prettyref{que:weak_strong} and to obtain that $S$ is a topological isomorphism (into) lies in a condition on the 
interplay of $S$ and the pair of operators $(T^{E},T^{\K})$ which we call consistency 
(see \prettyref{def:cons_strong} and  \prettyref{thm:S_iso_into}).
This condition is used to extend the mentioned results and we always have to balance the sets $\Lambda$ 
from which we extend our functions and the subspaces $G\subset E'$ 
with which we test. The case of \glq thin\grq{} sets $\Lambda$ and \glq thick\grq{} subspaces $G$ is handled in Section 3 and 5, 
the converse case of \glq thick\grq{} sets $\Lambda$ and \glq thin\grq{} subspaces $G$ in Section 4.
In our last section an application of our results is given to represent the $E$-valued space of $2\pi$-periodic smooth functions 
and the multiplier space of the Schwartz space 
by sequence spaces with explicit isomorphisms describing this representation 
(see \prettyref{cor:Schauder_coeff_space_2pi_per}, \prettyref{cor:Schauder_coeff_space_multiplier}).
\section{Notation and Preliminaries}
The notation and preliminaries are essentially the same as in \cite[Section 2, 3]{kruse2017,kruse2018_1}.
We equip the spaces $\R^{d}$, $d\in\N$, and $\C$ with the usual Euclidean norm $|\cdot|$.
Furthermore, for a subset $M$ of a topological space $X$ we denote by $\overline{M}$ the closure of $M$ in $X$.
For a subset $M$ of a topological vector space $X$, we write $\oacx(M)$ 
for the closure of the absolutely convex hull $\operatorname{acx}(M)$ of $M$ in $X$.

By $E$ we always denote a non-trivial locally convex Hausdorff space (lcHs) over the field 
$\K=\R$ or $\C$ equipped with a directed fundamental system of 
seminorms $(p_{\alpha})_{\alpha\in \mathfrak{A}}$. 
If $E=\K$, then we set $(p_{\alpha})_{\alpha\in \mathfrak{A}}:=\{|\cdot|\}$. 
For more details on the theory of locally convex spaces see \cite{F/W/Buch}, \cite{Jarchow} or \cite{meisevogt1997}.

By $X^{\Omega}$ we denote the set of maps from a non-empty set $\Omega$ to a non-empty set $X$
and by $L(F,E)$ the space of continuous linear operators from $F$ to $E$ 
where $F$ and $E$ are locally convex Hausdorff spaces. 
If $E=\K$, we just write $F':=L(F,\K)$ for the dual space and $G^{\circ}$ for the polar set of $G\subset F$. 
If $F$ and $E$ are (linearly topologically) isomorphic, we write $F\cong E$.
We denote by $L_{t}(F,E)$ the space $L(F,E)$ equipped with the locally convex topology $t$ of uniform convergence 
on the finite subsets of $F$ if $t=\sigma$, on the absolutely convex, compact subsets of $F$ if $t=\kappa$, 
on the absolutely convex, $\sigma(F,F')$-compact subsets of $F$ if $t=\mu$, 
on the precompact (totally bounded) subsets of $F$ if $t=\tau_{pc}$ and on the bounded subsets of $F$ if $t=b$. 
We use the symbol $t(F',F)$ for the corresponding topology on $F'$. 
A linear subspace $G$ of $F'$ is called separating if $f'(x)=0$ for every $f'\in G$ implies $x=0$. 
This is equivalent to $G$ being $\sigma(F',F)$-dense (and $\kappa(F',F)$-dense) in $F'$ by the bipolar theorem. 
The so-called $\varepsilon$-product of Schwartz is defined by $F\varepsilon E:=L_{e}(F_{\kappa}',E)$
where $L(F_{\kappa}',E)$ is equipped with the topology of uniform convergence on the equicontinuous subsets of $F'$. 
This definition of the $\varepsilon$-product coincides with the original 
one by Schwartz \cite[Chap.\ I, \S1, D\'{e}finition, p.\ 18]{Sch1}. 
It is symmetric which means that $F\varepsilon E\cong E\varepsilon F$.
Besides the $\varepsilon$-product of spaces there is an $\varepsilon$-product of continuous linear operators as well. 
For locally convex Hausdorff spaces $F_{i}$, $E_{i}$ and $T_{i}\in L(F_{i},E_{i})$, $i=1,2$, 
we define the $\varepsilon$-product $T_{1}\varepsilon T_{2}\in L(F_{1}\varepsilon F_{2},E_{1}\varepsilon E_{2})$ of the operators $T_{1}$ and $T_{2}$ by 
\[
 (T_{1}\varepsilon T_{2})(u):=T_{2}\circ u\circ T_{1}^{t},\quad u\in F_{1}\varepsilon F_{2},
\]
where $T_{1}^{t}\colon E_{1}'\to F_{1}'$, $e'\mapsto e'\circ T_{1}$, is the dual map of $T_{1}$. 
If $T_{1}$ is an isomorphism and $F_{2}=E_{2}$, then $T_{1}\varepsilon \id_{E_{2}}$ 
is also an isomorphism with inverse $T_{1}^{-1}\varepsilon\id_{E_{2}}$ by \cite[Chap.\ I, \S1, Proposition 1, p.\ 20]{Sch1}.
For more information on the theory of $\varepsilon$-products see \cite{Jarchow}, \cite{Kaballo} and \cite{Sch1}.

Further, for a disk $D\subset F$, i.e.\ a bounded, absolutely convex set, 
the vector space $F_{D}:=\bigcup_{n\in\N}nD$ becomes a normed space if it is equipped with 
gauge functional of $D$ as a norm (see \cite[p.\ 151]{Jarchow}). The space $F$ is called locally 
complete if $F_{D}$ is a Banach space for every closed disk $D\subset F$ (see \cite[10.2.1 Proposition, p.\ 197]{Jarchow}).

In the introduction we already mentioned that linearisations of spaces of vector-valued functions by means of 
$\varepsilon$-products are essential for our approach. Here, one of the important questions 
is which spaces of vector-valued functions can be represented by $\varepsilon$-products. 
Let us recall some basic definitions and results from \cite[Section 3]{kruse2017,kruse2018_1}.
Let $\Omega$ be a non-empty set and $E$ an lcHs. If $\F\subset\K^{\Omega}$
is an lcHs such that $\delta_{x}\in\F'$ for all $x\in\Omega$, then the map 
\[
S\colon \F\varepsilon E\to E^{\Omega},\;u\longmapsto [x\mapsto u(\delta_{x})],
\]
is well-defined and linear. 

\begin{defn}[{$\varepsilon$-into-compatible}]
Let $\Omega$ be a non-empty set and $E$ an lcHs. Let $\F\subset\K^{\Omega}$ and $\FE\subset E^{\Omega}$ 
be lcHs such that $\delta_{x}\in\F'$ for all $x\in\Omega$.
We call the spaces $\F$ and $\FE$ \emph{$\varepsilon$-into-compatible} if the map
\[
S\colon \F\varepsilon E\to \FE,\;u\longmapsto [x\mapsto u(\delta_{x})],
\]
is a well-defined isomorphism into. We call $\F$ and $\FE$ \emph{$\varepsilon$-compatible} if $S$ is an isomorphism. 
We write $S_{\F}$ if we want to emphasise the dependency on $\F$.
\end{defn}

The notion of $\varepsilon$-compatibility was introduced in \cite[Definition 3.4, p.\ 360]{kruse2018_1}. Next, we 
introduce a concept of pairs of operators $T^{\K}$ and $T^{E}$ acting on $\F$ and $\FE$, respectively, whose 
interplay with the map $S$ is the key to answer the question of linearisation of $\FE$ via $\varepsilon$-products 
and to generalise \prettyref{que:weak_strong}. 

\begin{defn}[{strong, consistent}]\label{def:cons_strong}
Let $\Omega$ be a non-empty set and $E$ an lcHs. Let $\F\subset\K^{\Omega}$ and $\FE\subset E^{\Omega}$ 
be lcHs such that $\delta_{x}\in\F'$ for all $x\in\Omega$. Let $(\omega_m)_{m\in M}$ be a family of non-empty sets, 
$T^{\K}_{m}\colon\dom T^{\K}_{m}\to \K^{\omega_{m}}$ 
and $T^{E}_{m}\colon\dom T^{E}_{m}\to E^{\omega_{m}}$ be linear with 
$\F\subset\dom T^{\K}_{m}\subset\K^{\Omega}$ and $\FE\subset\dom T^{E}_{m}\subset E^{\Omega}$ for all $m\in M$.
\begin{enumerate}
\item[a)] We call $(T^{E}_{m},T^{\K}_{m})_{m\in M}$ a \emph{consistent} family for $(\F,E)$, in short $(\mathcal{F},E)$, if 
for every $u\in\F\varepsilon E$, $m\in M$ and $x\in\omega_{m}$ holds
\begin{enumerate}
\item[(i)] $S(u)\in\FE$ and $T^{\K}_{m,x}:=\delta_{x}\circ T^{\K}_{m}\in\F'$,
\item[(ii)] $T^{E}_{m}S(u)(x)=u(T^{\K}_{m,x})$.
\end{enumerate}
\item[b)] We call $(T^{E}_{m},T^{\K}_{m})_{m\in M}$ a \emph{strong} family for $(\F,E)$, in short $(\mathcal{F},E)$, if 
for every $e'\in E'$, $f\in\FE$, $m\in M$ and $x\in\omega_{m}$ holds
\begin{enumerate}
\item[(i)] $e'\circ f\in\F$,
\item[(ii)] $T^{\K}_{m}(e'\circ f)(x)=e'\circ T^{E}_{m}(f)(x)$.
\end{enumerate}
\end{enumerate}
\end{defn}

As a convention we omit the index $m$ of the set $\omega_{m}$, the operators $T^{E}_{m}$ and $T^{\K}_{m}$ if $M$ is a singleton. 
If the family $(T^{E}_{m},T^{\K}_{m})_{m\in M}$ is incorporated in the topology of and $\F$ and $\FE$ in the sense of 
a weighted graph topology, then consistency implies $\varepsilon$-into-compatibility which we are about to explain. 
In this case the spaces $\F$ and $\FE$ are weighted spaces whose topology is induced by a family of weights $\mathcal{V}$ 
and operators $(T^{\K}_{m})_{m\in M}$ and $(T^{E}_{m})_{m\in M}$, respectively.

\begin{defn}[{weight function, \cite[Definition 2, p.\ 1515]{kruse2017}}]\label{def:weight} 
Let $J$ be a non-empty set and $(\omega_{m})_{m\in M}$ a family of non-empty sets. 
We call $\mathcal{V}:=(\nu_{j,m})_{j\in J,m\in M}$  
a family of \emph{weight functions} on $(\omega_{m})_{m\in M}$ if $\nu_{j,m}\colon \omega_{m}\to [0,\infty)$ 
for all $j\in J$, $m\in M$ and 
\begin{equation}\label{eq:loc3}
  \forall \; m\in M,\, x\in\omega_{m}\;\exists\;j\in J:\;0< \nu_{j,m}(x).
\end{equation}
\end{defn}

\begin{defn}[{\cite[Definition 3, p.\ 1515]{kruse2017}}]\label{def:standard_space}
Let $\Omega$ be a non-empty set, a family of weight functions $\mathcal{V}:=(\nu_{j,m})_{j\in J,m\in M}$ 
given on $(\omega_{m})_{m\in M}$ and
$T^{E}_{m}\colon E^{\Omega}\supset\dom T^{E}_{m} \to E^{\omega_{m}}$ a linear map for every $m\in M$. 
Let $\operatorname{AP}(\Omega,E)$ be a linear subspace of $E^{\Omega}$ and define the space of intersections 
\[
F(\Omega,E):=\operatorname{AP}(\Omega,E)\cap\bigl(\bigcap_{m\in M}\dom T^{E}_{m}\bigr)
\]
as well as
\[
\FVE:=\bigl\{f\in F(\Omega,E)\;|\; 
 \forall\;j\in J,\, m\in M,\,\alpha\in \mathfrak{A}:\; |f|_{j,m,\alpha}<\infty\bigr\}
\]
where
\[
|f|_{j,m,\alpha}:=\sup_{x \in \omega_{m}}
p_{\alpha}\bigl(T^{E}_{m}(f)(x)\bigr)\nu_{j,m}(x).
\]
Further, we write $\FV:=\mathcal{FV}(\Omega,\K)$. If we want to emphasise dependencies, we write 
$M(\mathcal{FV})$ or $M(E)$ instead of $M$. We omit the index $\alpha$ if $E$ is a normed space.
\end{defn}

In $\operatorname{AP}(\Omega,E)$ additional properties of the functions are gathered which are not incorporated into the topology. 
It is easy to check that $\FVE$ is locally convex but need not be Hausdorff. Furthermore, we need the point evaluations to be elements 
of the dual $\FV'$ for the map $S$ to be defined. 

\begin{defn}[{$\dom$-space, \cite[Definition 4, p.\ 1515]{kruse2017}}]
We call $\FVE$ a \emph{$\dom$-space} if it is an lcHs, 
the system of seminorms $(|\cdot|_{j,m,\alpha})_{j\in J, m\in M, \alpha\in\mathfrak{A}}$ 
is directed and, in addition, $\delta_{x}\in\FV'$ for every $x\in\Omega$ if $E=\K$.
\end{defn}

\begin{defn}[{generator}]
Consider the $\dom$-spaces $\FV$ and $\FVE$ with $M:=M(\K)=M(E)$. 
\begin{enumerate}
\item[a)] We call $(T^{E}_{m},T^{\K}_{m})_{m\in M}$ from \prettyref{def:standard_space} a \emph{generator} for $(\FV,E)$, in short $(\mathcal{FV},E)$.
\item[b)] We call a generator $(T^{E}_{m},T^{\K}_{m})_{m\in M}$ consistent if it is consistent in the sense of \prettyref{def:cons_strong} a). 
\item[c)] We call a generator $(T^{E}_{m},T^{\K}_{m})_{m\in M}$ strong if it is strong in the sense of \prettyref{def:cons_strong} b). 
\end{enumerate}
\end{defn}

The following remark shows that the preceding definition of a consistent resp.\ strong generator coincides with the 
one given in \cite[Definition 6, p.\ 1516]{kruse2017}.

\begin{rem}
We note that the condition $T^{\K}_{m,x}\in\FV'$ for all $m\in M$ and $x\in\omega_{m}$ in a)(i) of \prettyref{def:cons_strong} 
is always satisfied for generators by \cite[Remark 5 b), p.\ 1516]{kruse2017} and \eqref{eq:loc3}. Moreover, 
if $S(u)\in\operatorname{AP}(\Omega,E)\cap\dom T^{E}_{m}$ for $u\in\FV\varepsilon E$ and all $m\in M$
and a)(ii) of \prettyref{def:cons_strong} is fulfilled, then $S(u)\in\FVE$ by \cite[Lemma 7, p.\ 1517]{kruse2017}, 
implying that a)(i) is satisfied. 
Further, if $f\in\FVE$ and $e'\circ f\in \operatorname{AP}(\Omega)\cap\dom T^{\K}_{m}$ for all $e'\in E'$ and $m\in M$
and b)(ii) of \prettyref{def:cons_strong} is fulfilled, then $e'\circ f\in\FV$ by \cite[Lemma 12, p.\ 1522-1523]{kruse2017}, 
implying that b)(i) is satisfied. 
\end{rem}

\begin{thm}[{\cite[3.9 Theorem, p.\ 9]{kruse2017a}}]\label{thm:S_iso_into}
Let $(T^{E}_{m},T^{\K}_{m})_{m\in M}$ be a consistent generator for $(\mathcal{FV},E)$. 
Then $S\colon \FV\varepsilon E\to \FVE$ is an isomorphism into, i.e.\ $\FV$ and $\FVE$ are $\varepsilon$-into-compatible.
\end{thm}

Sufficient conditions for $\varepsilon$-compatibility involving the strength of the generator as well can be found in 
\cite[Theorem 14, p.\ 1524]{kruse2017}. Let us a give a standard example, namely, weighted spaces of continuously partially differentiable 
functions. More examples can be found in \cite{kruse2017,kruse2018_1}. 
We recall the definition of continuous partial differentiability of a vector-valued function. 
A function $f\colon\Omega\to E$ on an open set $\Omega\subset\R^{d}$ to an lcHs $E$ is called 
continuously partially differentiable ($f$ is $\mathcal{C}^{1}$) 
if for the $n$-th unit vector $e_{n}\in\R^{d}$ the limit
\[
(\partial^{e_{n}})^{E}f(x)
:=\lim_{\substack{h\to 0\\ h\in\R, h\neq 0}}\frac{f(x+he_{n})-f(x)}{h}
\]
exists in $E$ for every $x\in\Omega$ and $(\partial^{e_{n}})^{E}f$ 
is continuous on $\Omega$ ($(\partial^{e_{n}})^{E}f$ is $\mathcal{C}^{0}$) for every $1\leq n\leq d$.
For $k\in\N$ a function $f$ is said to be $k$-times continuously partially differentiable 
($f$ is $\mathcal{C}^{k}$) if $f$ is $\mathcal{C}^{1}$ and all its first partial derivatives are $\mathcal{C}^{k-1}$.
A function $f$ is called infinitely continuously partially differentiable ($f$ is $\mathcal{C}^{\infty}$) 
if $f$ is $\mathcal{C}^{k}$ for every $k\in\N$.
For $k\in\N_{\infty}:=\N\cup\{\infty\}$ the linear space of all functions $f\colon\Omega\to E$ which are $\mathcal{C}^{k}$ 
is denoted by $\mathcal{C}^{k}(\Omega,E)$. 
Let $f\in\mathcal{C}^{k}(\Omega,E)$. For $\beta\in\N_{0}^{d}$ with 
$|\beta|:=\sum_{n=1}^{d}\beta_{n}\leq k$ we set 
$(\partial^{\beta_{n}})^{E}f:=f$ if $\beta_{n}=0$, and
\[
(\partial^{\beta_{n}})^{E}f
:=\underbrace{(\partial^{e_{n}})^{E}\cdots(\partial^{e_{n}})^{E}}_{\beta_{n}\text{-times}}f
\]
if $\beta_{n}\neq 0$ as well as 
\[
(\partial^{\beta})^{E}f
:=(\partial^{\beta_{1}})^{E}\cdots(\partial^{\beta_{d}})^{E}f.
\]
If $E=\K$, we usually write $\partial^{\beta}f:=(\partial^{\beta})^{\K}f$.

\begin{exa}[{\cite[3.6, 3.15 Example, p.\ 6, 11, 28]{kruse2017a}}]\label{ex:standard_example}
Let $k\in\N_{\infty}$ and $\Omega\subset\R^{d}$ be open. 
We consider the cases
\begin{enumerate}
\item [(i)] $\omega_{m}:=M_{m}\times\Omega$ with $M_{m}:=\{\beta\in\N_{0}^{d}\;|\;|\beta|\leq \min(m,k)\}$ for all $m\in\N_{0}$, or
\item [(ii)] $\omega_{m}:=\N_{0}^{d}\times\Omega$ for all $m\in\N_{0}$ and $k=\infty$,
\end{enumerate}
and let $\mathcal{V}^{k}:=(\nu_{j,m})_{j\in J, m\in\N_{0}}$ be a directed family 
of weights on $(\omega_{m})_{m\in\N_{0}}$ where directed means that for every $j_{1},j_{2}\in J$ and $m_{1},m_{2}\in\N_{0}$ there 
are $j_{3}\in J$, $m_{3}\in\N_{0}$, $m_{3}\geq m_{1},m_{2}$, and $C>0$ such that $\nu_{j_{1},m_{1}},\nu_{j_{2},m_{2}}\leq C\nu_{j_{3},m_{3}}$. 
We define the weighted space of $k$-times continuously partially differentiable functions with values in an lcHs $E$ as
\[
 \mathcal{CV}^{k}(\Omega,E):=\{f\in\mathcal{C}^{k}(\Omega,E)\;|\;\forall\;j\in J,\,m\in\N_{0},\,
 \alpha\in\mathfrak{A}:\;|f|_{j,m,\alpha}<\infty\} 
\]
where 
\[
 |f|_{j,m,\alpha}:=\sup_{(\beta,x)\in\omega_{m}}
 p_{\alpha}\bigl((\partial^{\beta})^{E}f(x)\bigr)\nu_{j,m}(\beta,x).
\]
Setting $\dom T^{E}_{m}:=\mathcal{C}^{k}(\Omega,E)$ and 
\begin{equation}\label{eq:standard_ex_generator}
 T^{E}_{m}\colon\mathcal{C}^{k}(\Omega,E)\to E^{\omega_{m}},\; f\longmapsto [(\beta,x)\mapsto (\partial^{\beta})^{E}f(x)], 
\end{equation}
as well as $\operatorname{AP}(\Omega,E):=E^{\Omega}$, we observe that $\mathcal{CV}^{k}(\Omega,E)$ is a $\dom$-space and 
\[
 |f|_{j,m,\alpha}=\sup_{x\in\omega_{m}}p_{\alpha}\bigl(T^{E}_{m}f(x)\bigr)\nu_{j,m}(x).
\]

b) The space $\mathcal{C}^{k}(\Omega,E)$ with its usual topology of uniform convergence of all partial derivatives up to order $k$
on compact subsets of $\Omega$ is a special case of a)(i) with $J:=\{K\subset\Omega\;|\;K\;\text{compact}\}$, 
$\nu_{K,m}(\beta,x):=\chi_{K}(x)$, $(\beta,x)\in\omega_{m}$, for all $m\in\N_{0}$ and $K\in J$ 
where $\chi_{K}$ is the characteristic function of $K$.  
In this case we write $\mathcal{W}^{k}:=\mathcal{V}^{k}$ for the family of weight functions.

c) The Schwartz space is defined by
\[
\mathcal{S}(\R^{d},E)
:=\{f\in \mathcal{C}^{\infty}(\R^{d},E)\;|\;
\forall\;m\in\N_{0},\,\alpha\in \mathfrak{A}:\;|f|_{m,\alpha}<\infty\}
\]
where 
\[
 |f|_{m,\alpha}:=\sup_{\substack{x\in\R^{d}\\ \beta\in\N_{0}^{d},|\beta|\leq m}}
 p_{\alpha}\bigl((\partial^{\beta})^{E}f(x)\bigr)(1+|x|^{2})^{m/2}.
\]
This is a special case of a)(i) with $k=\infty$, $\Omega=\R^{d}$, $J=\{1\}$ and $\nu_{1,m}(\beta,x):=(1+|x|^{2})^{m/2}$, 
$(\beta,x)\in\omega_{m}$, for all $m\in\N_{0}$.

d) Let $\mathfrak{K}:=\{K\subset\Omega\;|\;K\;\text{compact}\}$ 
and $(M_{p})_{p\in\N_{0}}$ be a sequence of positive real numbers. 
The space $\mathcal{E}^{(M_{p})}(\Omega,E)$ of ultradifferentiable functions of class $(M_{p})$ of Beurling-type is defined as 
\[
\mathcal{E}^{(M_{p})}(\Omega,E)
:=\{f\in \mathcal{C}^{\infty}(\Omega,E)\;|\;\forall\;K\in\mathfrak{K},\,h>0,\,\alpha\in \mathfrak{A}:\;|f|_{(K,h),\alpha}<\infty\}
\]
where 
\[
 |f|_{(K,h),\alpha}:=\sup_{\substack{x\in K \\ \beta\in\N_{0}^{d}}}
 p_{\alpha}\bigl((\partial^{\beta})^{E}f(x)\bigr)\frac{1}{h^{|\beta|}M_{|\beta|}}.
\]
This is a special case of a)(ii) with $J:=\mathfrak{K}\times\R_{>0}$ and 
$\nu_{(K,h),m}(\beta,x):=\chi_{K}(x)\frac{1}{h^{|\beta|}M_{|\beta|}}$, $(\beta,x)\in\omega_{m}$, 
for all $(K,h)\in J$ and $m\in\N_{0}$.

e) Let $\mathfrak{K}$ and $(M_{p})_{p\in\N_{0}}$ be as in d). 
The space $\mathcal{E}^{\{M_{p}\}}(\Omega,E)$ of ultradifferentiable functions of class $\{M_{p}\}$ of Roumieu-type is defined as
\[
\mathcal{E}^{\{M_{p}\}}(\Omega,E)
:=\{f\in \mathcal{C}^{\infty}(\Omega,E)\;|\;\forall\;(K,H)\in J,\,\alpha\in \mathfrak{A}:\;|f|_{(K,H),\alpha}<\infty\}
\]
where 
\[
J:=\mathfrak{K}\times\{H=(H_{n})_{n\in\N}\;|\;\exists\;(h_{k})_{k\in\N},\,h_{k}>0,\,h_{k}\nearrow\infty
\;\forall\;n\in\N:\;
H_{n}=h_{1}\cdot\ldots\cdot h_{n}\}
\]
and
\[
 |f|_{(K,H),\alpha}:=\sup_{\substack{x\in K\\ \beta\in\N_{0}^{d}}}
 p_{\alpha}\bigl((\partial^{\beta})^{E}f(x)\bigr)\frac{1}{H_{|\beta|}M_{|\beta|}}
\]
(see \cite[Proposition 3.5, p.\ 675]{Kom9}). Again, this is a special case of a)(ii) with 
$\nu_{(K,H),m}(\beta,x):=\chi_{K}(x)\frac{1}{H_{|\beta|}M_{|\beta|}}$, $(\beta,x)\in\omega_{m}$, 
for all $(K,H)\in J$ and $m\in\N_{0}$.

f) Let $n\in\N$, $\beta_{i}\in\N_{0}^{d}$ with $|\beta_{i}|\leq k$ and 
$a_{i}\colon\Omega\to\K$ for $1\leq i\leq n$. We set 
\[
 P(\partial)^{E}\colon \mathcal{C}^{k}(\Omega,E)\to E^{\Omega},\; P(\partial)^{E}(f)(x):=\sum_{i=1}^{n}a_{i}(x)(\partial^{\beta_{i}})^{E}(f)(x).
\]
and obtain the (topological) subspace of $\mathcal{CV}^{k}(\Omega,E)$ given by 
\[
 \mathcal{CV}_{P(\partial)}^{k}(\Omega,E):=\{f\in\mathcal{CV}^{k}(\Omega,E)\;|\;f\in\ker P(\partial)^{E}\}.
\]

g) In the case (i), i.e.\ $\omega_{m}=M_{m}\times\Omega$ with $M_{m}=\{\beta\in\N_{0}^{d}\;|\;|\beta|\leq \min(m,k)\}$ for all $m\in\N_{0}$,
we define the topological subspace of $\mathcal{CV}^{k}(\Omega,E)$ from a)
consisting of the functions that vanish with all their derivatives when weighted at infinity by 
 \begin{align*}
  \mathcal{CV}^{k}_{0}(\Omega,E):=\{f\in\mathcal{CV}^{k}(\Omega,E)\;|\;&\forall\;j\in J,\,
  m\in\N_{0},\,\alpha\in\mathfrak{A},\,\varepsilon>0\\
  &\exists\;K\subset \Omega\;\text{compact}:\;|f|_{\Omega\setminus K,j,m,\alpha}<\varepsilon\} 
 \end{align*}
where 
 \[
  |f|_{\Omega\setminus K,j,m,\alpha}:=\sup_{\substack{x\in\Omega\setminus K\\ \beta\in M_{m}}}
  p_{\alpha}\bigl((\partial^{\beta})^{E}f(x)\bigr)\nu_{j,m}(\beta,x).
 \] 
Further, we define its subspace $\mathcal{CV}^{k}_{P(\partial),0}(\Omega,E):=\mathcal{CV}_{0}^{k}(\Omega,E)\cap\mathcal{CV}_{P(\partial)}^{k}(\Omega,E)$ 
with the linear partial differential operator $P(\partial)^{E}$ from f).
\end{exa}

If $\mathcal{V}^{k}$, $k\in\N_{\infty}$, is locally bounded away from zero on an open set $\Omega\subset\R^{d}$, i.e.\ 
for every compact set $K\subset\Omega$ and $m\in\N_{0}$ there is $j\in J$ such that 
\[
\inf_{\substack{x\in K, \beta\in\N_{0}^{d}\\ |\beta|\leq \min(m,k)}}\nu_{j,m}(\beta,x)>0,
\]
then the inclusion $\mathcal{CV}^{k}(\Omega)\to\mathcal{CW}^{k}(\Omega)$, $f\mapsto f$, is continuous
and we have the following result concerning consistency, strength and $\varepsilon$-into-compatibility 
by virtue of the Banach-Steinhaus theorem.

\begin{prop}\label{prop:S_iso_into_standard_example}
Let $E$ be an lcHs, $k\in\N_{\infty}$, $\mathcal{V}^{k}$ a directed family of weights which is locally bounded away from zero 
on an open set $\Omega\subset\R^{d}$ and $\mathcal{F}(\Omega)$ barrelled where $\mathcal{F}$ stands for 
$\mathcal{CV}^{k}$, $\mathcal{CV}^{k}_{0}$, $\mathcal{CV}^{k}_{P(\partial)}$ or $\mathcal{CV}^{k}_{P(\partial),0}$. 
Then the following holds.
\begin{enumerate}
\item[a)] If $u\in\F\varepsilon E$, then $S(u)\in\mathcal{C}^{k}(\Omega,E)$ and 
\[
(\partial^{\beta})^{E}S(u)(x)=u(\delta_{x}\circ\partial^{\beta}),\quad\beta\in\N_{0}^{d},\;|\beta|\leq k,\;x\in\Omega.
\]

\item[b)] If $e'\in E'$ and $f\in\FE$, then $e'\circ f\in\mathcal{C}^{k}(\Omega,E)$ and
\[
\partial^{\beta}(e'\circ f)(x)=e'\bigl((\partial^{\beta})^{E}f(x)\bigr),\quad \beta\in\N_{0}^{d},\;|\beta|\leq k,\;x\in\Omega.
\]
\item[c)] $((\partial^{\beta})^{E},\partial^{\beta})_{\beta\in\N_{0}^{d},|\beta|\leq m}$ with $m\leq k$ is a strong, 
consistent family for $(\mathcal{F},E)$.
\item[d)] $(T^{E}_{m},T^{\K}_{m})_{m\in\N_{0}}$ from \eqref{eq:standard_ex_generator} is a strong, consistent generator for $(\mathcal{F},E)$.
\item[e)] $\F$ and $\FE$ are $\varepsilon$-into-compatible.
\end{enumerate}
\end{prop}
\begin{proof}
Part a) is shown in \cite[Proposition 10, p.\ 1520]{kruse2017} and b) in the proof of 
\cite[Proposition 9, p.\ 1519]{kruse2017}. 
Part c) is included in part d) by the definition of the generator. 
The consistency and strength for $(\mathcal{F},E)$ in part d) is a direct consequence of a) and b) if $\mathcal{F}=\mathcal{CV}^{\infty}$. 
The additional properties of vanishing at infinity or being in the kernel of $P(\partial)$ needed for $S(u)$ in (i) of 
\prettyref{def:cons_strong} a) for $u\in\F\varepsilon E$ and for $e'\circ f$ in (i) of \prettyref{def:cons_strong} b) 
for $e'\in E'$ and $f\in\FE$ if $\mathcal{F}=\mathcal{CV}^{k}_{0}$, $\mathcal{CV}^{k}_{P(\partial)}$ or $\mathcal{CV}^{k}_{P(\partial),0}$ 
are proved in \cite[Proposition 3.15 a), p.\ 243]{kruse2018_2} 
for $\mathcal{F}=\mathcal{CV}^{\infty}_{0}$, in \cite[Proposition 9, p.\ 1519]{kruse2017} for $\mathcal{F}=\mathcal{CV}^{\infty}_{P(\partial)}$ and in the proof of \cite[5.10 Example b), p.\ 28]{kruse2017a} for 
$\mathcal{F}=\mathcal{CV}^{\infty}_{P(\partial),0}$. 
Part e) follows from d) by \prettyref{thm:S_iso_into}.
\end{proof}
\section{Extension of vector-valued functions}
Using the functionals $T^{\K}_{m,x}$, we extend the definition of a set of uniqueness 
and a space of restrictions given in \cite[Definition 4, 5, p.\ 230]{B/F/J}. This prepares the ground for a generalisation 
of \prettyref{que:weak_strong} using a strong, consistent family $(T^{E}_{m},T^{\K}_{m})_{m\in M}$.

\begin{defn}[{set of uniqueness}]
Let $\Omega$ be a non-empty set, $\F\subset\K^{\Omega}$ an lcHs, $(\omega_m)_{m\in M}$ be a family of non-empty sets 
and $T^{\K}_{m}\colon \F\to \K^{\omega_{m}}$ be linear for all $m\in M$.
$U\subset\bigcup_{m\in M}\{m\}\times\omega_{m}$ is called 
a \emph{set of uniqeness} for $(T^{\K}_{m},\mathcal{F})_{m\in M}$ if
\begin{enumerate}
\item [(i)] $\forall\; (m,x)\in U:\; T^{\K}_{m,x}\in \F' $,
\item [(ii)] $\forall\; f\in\F: \;T^{\K}_{m}(f)(x)=0 \;\;\forall\,(m,x)\in U\;\;\Rightarrow\;\; f=0$.
\end{enumerate}
We omit the index $m$ in $\omega_{m}$ and $T^{\K}_{m}$ if $M$ is a singleton and consider $U$ as a subset of $\Omega$.
\end{defn}

If $U$ is a set of uniqueness for $(T^{\K}_{m},\mathcal{F})_{m\in M}$, then $\operatorname{span}\{T^{\K}_{m,x}\;|\;(m,x)\in U\}$ is
dense in $\F_{\sigma}'$ (and $\F_{\kappa}'$) by the bipolar theorem. 

\begin{rem}\label{rem:Schauder_coeff_set_uni} 
Let $\Omega$ be a non-empty set and $\F\subset\K^{\Omega}$ an lcHs.
\begin{enumerate}
 \item[a)] A simple set of uniqueness for $(\id_{\K^{\Omega}},\mathcal{F})$ is given by $U:=\Omega$ if 
 $\delta_{x}\in\F'$ for all $x\in\Omega$.
 \item[b)] If $\F$ has a Schauder basis $(f_{n})_{n\in\N}$  with associated sequence of coefficient functionals 
 $T^{\K}:=(T^{\K}_{n})_{n\in\N}$. Then $U:=\N$ is a set of uniqueness for $(T^{\K},\mathcal{F})$.
\end{enumerate}
\end{rem}

An example for b) is the space of holomorphic functions on an open disc $\D_{r}(z_{0})\subset\C$ 
with radius $0<r\leq\infty$ and center $z_0\in\C$. 
If we equip this space with compact-open topology, then it has the shifted monomials $((\cdot-z_{0})^{n})_{n\in\N_{0}}$ 
as a Schauder basis 
with the point evaluations $(\delta_{z_0}\circ\partial_{\C}^{n})_{n\in\N_{0}}$ 
given by $(\delta_{z_0}\circ\partial_{\C}^{n})(f):=f^{(n)}(z_{0})$ 
as associated sequence of coefficient functionals where $f^{(n)}(z_{0})$ 
denotes the $n$-th complex derivative at $z_{0}$ of a holomorphic function $f$ 
on $\D_{r}(z_{0})$. We will explore further sets of uniqueness for concrete function spaces in the upcoming examples 
and come back to b) in our last section. 

\begin{defn}[{restriction space}]
Let $G\subset E'$ be a separating subspace and $U$ a set of uniqueness for $(T^{\K}_{m},\mathcal{F})_{m\in M}$.
Let $\mathcal{F}_{G}(U,E)$ be the space of functions $f\colon U\to E$ such that for every $e'\in G$ there is 
$f_{e'}\in\F$ with $T^{\K}_{m}(f_{e'})(x)=e'\circ f(m,x)$ for all $(m,x)\in U$.
\end{defn}

\begin{rem}\label{rem:R_f}
Since $U$ is a set of uniqueness, the functions $f_{e'}$ are unique and the map 
$\mathscr{R}_{f}\colon E'\to \F,$ $\mathscr{R}_{f}(e'):=f_{e'}$, is well-defined and linear.
\end{rem}

\begin{rem}\label{rem:R_well-defined}
Let $\F$ and $\FE$ be $\varepsilon$-into-compatible.
Consider a set of uniqueness $U$ for $(T^{\K}_{m},\mathcal{F})_{m\in M}$, a separating subspace $G\subset E'$ and 
a strong, consistent family $(T^{E}_{m},T^{\K}_{m})_{m\in M}$ for $(\mathcal{F},E)$.
For $u\in \F\varepsilon E$ set $f:=S(u)$. Then $f\in\FE$ by the $\varepsilon$-into-compatibility 
and we set $\widetilde{f}\colon U\to E$, $\widetilde{f}(m,x):=T^{E}_{m}(f)(x)$. It follows that
\[
e'\circ \widetilde{f}(m,x)= (e'\circ T^{E}_{m}(f))(x)=T^{\K}_{m}(e'\circ f)(x)
\]
for all $(m,x)\in U$ and $f_{e'}:=e'\circ f\in\F$ for all $e'\in E'$ by the strength of the family.
We conclude that $\widetilde{f}\in\mathcal{F}_{G}(U,E)$.
\end{rem} 

Under the assumptions of the preceding remark the map
\[
R_{U,G}\colon S(\F\varepsilon E)\to \mathcal{F}_{G}(U,E),\;f\mapsto (T^{E}_{m}(f)(x))_{(m,x)\in U},
\]
is well-defined. The map $R_{U,G}$ is also linear since $T^{E}_{m}$ is linear for all $m\in M$. 
Further, the strength of the defining family guarantees that $R_{U,G}$ is injective.

\begin{prop}\label{prop:injectivity}
Let $\F$ and $\FE$ be $\varepsilon$-into-compatible, 
$G\subset E'$ a separating subspace and $U$ a set of uniqueness for $(T^{\K}_{m},\mathcal{F})_{m\in M}$.
If $(T^{E}_{m},T^{\K}_{m})_{m\in M}$ is a strong family for $(\mathcal{F},E)$, 
then the map 
\[
T^{E}\colon \FE\to E^{U}, \;f\mapsto (T^{E}_{m}(f)(x))_{(m,x)\in U},
\]
is injective, in particular, $R_{U,G}$ is injective.
\end{prop}
\begin{proof}
Let $f\in \FE$ with $T^{E}(f)=0$. Then
\[ 
0=e'\circ T^{E}(f)(m,x)=e'\circ T^{E}_{m}(f)(x)=T^{\K}_{m}(e'\circ f)(x),\quad (m,x)\in U,
\]
and $e'\circ f \in\F$ for all $e'\in E'$ by the strength of the family. 
Since $U$ is a set of uniqueness, we get that $e'\circ f=0$ for all $e'\in E'$, which implies $f=0$.
\end{proof}

\begin{que}\label{que:surj_restr_set_unique}
Let $\F$ and $\FE$ be $\varepsilon$-into-compatible,
$G\subset E'$ a separating subspace, 
$(T^{E}_{m},T^{\K}_{m})_{m\in M}$ a strong family for $(\mathcal{F},E)$
and $U$ a set of uniqueness for $(T^{\K}_{m},\mathcal{F})_{m\in M}$.
When is the injective restriction map 
\[
R_{U,G}\colon S(\F\varepsilon E)\to \mathcal{F}_{G}(U,E),\;f\mapsto (T^{E}_{m}(f)(x))_{(m,x)\in U},
\]
surjective?
\end{que}

The \prettyref{que:weak_strong} is a special case of this question if there is a 
set of uniqueness $U$ for $(T^{\K}_{m},\mathcal{F})_{m\in M}$ with 
$\{T^{\K}_{m,x}\;|\; (m,x)\in U\}=\{\delta_{x}\;|\;x\in\Lambda\}$, $\Lambda\subset\Omega$. 
We observe that a positive answer to the surjectivity of $R_{\Omega,G}$ results in the following weak-strong principle.

\begin{prop}\label{prop:weak_strong_principle}
Let $\F$ and $\FE$ be $\varepsilon$-into-compatible, $G\subset E'$ a separating subspace such that 
$e'\circ f\in \F$ for all $e'\in G$ and $f\in\FE$. 
If 
\[
R_{\Omega,G}\colon S(\F\varepsilon E)\to\mathcal{F}_{G}(\Omega,E),\;f\mapsto f,
\]
with the set of uniqueness $\Omega$ for $(\id_{\K^{\Omega}},\mathcal{F})$ is surjective, then 
\[
 \F\varepsilon E\cong \FE \quad\text{via}\;S\quad\text{and}\quad\FE=\{f\colon\Omega\to E\;|\;\forall\;e'\in G:\;e'\circ f\in\F \}.
\]
\end{prop}
\begin{proof}
From the $\varepsilon$-into-compatibility and the surjectivity of $R_{\Omega,G}$ we obtain
\[
 \{f\colon\Omega\to E\;|\;\forall\;e'\in G:\;e'\circ f\in\F\}=\mathcal{F}_{G}(\Omega,E)=S(\F\varepsilon E)\subset \FE. 
\]
Further, the assumption that $e'\circ f\in \F$ for all $e'\in G$ and $f\in\FE$, implies that $\FE$ is a subspace of the space 
on the left-hand side, which proves our statement, in particular, the surjectivity of $S$.
\end{proof}

To answer \prettyref{que:surj_restr_set_unique} for general sets of uniqueness we have to restrict to a certain class
of separating subspaces of $E'$.

\begin{defn}[{determine  boundedness \cite[p.\ 230]{B/F/J}}]
A linear subspace $G\subset E'$ \emph{determines boundedness} if every $\sigma(E,G)$-bounded set $B\subset E$ is 
already bounded in $E$.
\end{defn}

In \cite[p.\ 139]{fernandez1989} such a space $G$ is called uniform boundedness deciding by Fern\'andez et al.\ and 
in \cite[p.\ 63]{nygaard2006} $w^{\ast}$-thick by Nygaard if $E$ is a Banach space.

\begin{rem} 
\begin{enumerate}
\item[a)] Let $E$ be an lcHs. Then $G:=E'$ determines boundedness by \cite[Mackey's theorem 23.15, p.\ 268]{meisevogt1997}.
\item[b)] Let $X$ be a barrelled lcHs, $Y$ an lcHs and $E:=L_{b}(X,Y)$. For $x\in X$ and $y'\in Y'$ we set 
$\delta_{x,y'}\colon L(X,Y)\to \K$, $T\to y'(T(x))$, and $G:=\{\delta_{x,y'}\;|\;x\in X,\,y'\in Y'\}\subset E'$. 
Then the span of $G$ determines boundedness (in $E$) by Mackey's theorem and the uniform boundedness principle. 
For Banach spaces $X$,$Y$ this is already observed in \cite[Remark 11, p.\ 233]{B/F/J} and, if in addition $Y=\K$, 
in \cite[Remark 1.4 b), p.\ 781]{Arendt2000}.
\item[c)] Further examples and a characterisation of subspaces $G\subset E'$ that determine boundedness can be found 
in \cite[Remark 1.4, p.\ 781-782]{Arendt2000}, \cite[Theorem 1.5, p.\ 63-64]{nygaard2006}
and \cite[Theorem 2.3, 2.4, p.\ 67-68]{nygaard2006} in the case that $E$ is a Banach space. 
\end{enumerate}
\end{rem}

\section*{\texorpdfstring{$\F$}{F(Omega)} a semi-Montel space and \texorpdfstring{$E$}{E} (sequentially) complete}

Our next results are in need of spaces $\F$ such that closed graph theorems hold with Banach spaces as domain spaces and 
$\F$ as the range space. Let us formally define this class of spaces. 

\begin{defn}[{BC-space \cite[p.\ 395]{powell1955}}]
We call an lcHs $F$ a \emph{BC-space} if for every Banach space $X$ and every linear map $f\colon X\to F$ 
with closed graph in $X\times F$, one has that $f$ is continuous. 
\end{defn}

A characterisation of BC-spaces is given by Powell in \cite[6.1 Corollary, p.\ 400-401]{powell1955}. 
Since every Banach space is ultrabornological and barrelled, the 
\cite[Closed graph theorem 24.31, p.\ 289]{meisevogt1997} of de Wilde 
and the Pt\'{a}k-K\={o}mura-Adasch-Valdivia closed graph theorem \cite[\S34, 9.(7), p.\ 46]{Koethe} 
imply that webbed spaces and $B_{r}$-complete spaces (infra-Pt\'{a}k spaces) are BC-spaces, 
for instance $B$-complete spaces, Fr\'{e}chet spaces, LF-spaces and strong duals of LF-spaces.
The following proposition is a modification of \cite[Satz 10.6, p.\ 237]{Kaballo} and 
uses the map $\mathscr{R}_{f}\colon e'\mapsto f_{e'}$ from \prettyref{rem:R_f}. 

\begin{prop}\label{prop:ext_F_semi_M}
Let $U$ be a set of uniqueness for $(T^{\K}_{m},\mathcal{F})_{m\in M}$ and $\F$ a BC-space.
Then $\mathscr{R}_{f}(B_{\alpha}^{\circ})$ is bounded in $\F$ for every $f\in\mathcal{F}_{E'}(U,E)$ 
and $\alpha\in\mathfrak{A}$ where $B_{\alpha}:=\{x\in E\;|\; p_{\alpha}(x)<1\}$. 
In addition, if $\F$ is semi-Montel, then $\mathscr{R}_{f}(B_{\alpha}^{\circ})$ is relatively compact in $\F$.
\end{prop}
\begin{proof}
Let $f\in\mathcal{F}_{E'}(U,E)$ and $\alpha\in\mathfrak{A}$. The polar $B_{\alpha}^{\circ}$ 
is compact in $E_{\sigma}'$ and thus $E_{B_{\alpha}^{\circ}}'$ is a Banach space 
by \cite[Corollary 23.14, p.\ 268]{meisevogt1997}.
We claim that the restriction of $\mathscr{R}_{f}$ to $E_{B_{\alpha}^{\circ}}'$ has closed graph. Indeed, let 
$(e_{\iota}')$ be a net in $E_{B_{\alpha}^{\circ}}'$ converging to $e'$ in $E_{B_{\alpha}^{\circ}}'$ and  
$\mathscr{R}_{f}(e_{\iota}')$ converging to $g$ in $\F$. For $(m,x)\in U$ we note that
\begin{align*}
  T^{\K}_{m,x}\bigl(\mathscr{R}_{f}(e_{\iota}')\bigr)
&=T^{\K}_{m}(f_{e_{\iota}'})(x)
 =(e_{\iota}'\circ f)(m,x)
 \to (e'\circ f)(m,x)=T^{\K}_{m}(f_{e'})(x)\\
&=T^{\K}_{m}\bigl(\mathscr{R}_{f}(e')\bigr)(x).
\end{align*}
The left-hand side converges to $T^{\K}_{m,x}(g)$ since $T^{\K}_{m,x}\in\F'$ for all 
$(m,x)\in U$. Hence we have $T^{\K}_{m}(g)(x)=T^{\K}_{m}\bigl(\mathscr{R}_{f}(e')\bigr)(x)$ 
for all $(m,x)\in U$. From $U$ being a set of uniqueness follows that $g=\mathscr{R}_{f}(e')$. 
Thus the restriction of $\mathscr{R}_{f}$ to the Banach space $E_{B_{\alpha}^{\circ}}'$ has closed graph 
and is continuous since $\F$ is a BC-space. This yields that $\mathscr{R}_{f}(B_{\alpha}^{\circ})$ 
is bounded as $B_{\alpha}^{\circ}$ is bounded in $E_{B_{\alpha}^{\circ}}'$.
If $\F$ is also a semi-Montel space, then $\mathscr{R}_{f}(B_{\alpha}^{\circ})$ is even relatively compact.
\end{proof}

Now, we are ready to prove our first extension theorem. Its proof of surjectivity of $R_{U,E'}$ 
is just an adaption of the proof of surjectivity of $S$ given in \cite[Theorem 14, p.\ 1524]{kruse2017}. 
Let $U$ be a set of uniqueness for $(T^{\K}_{m},\mathcal{F})_{m\in M}$. 
For $f\in\mathcal{F}_{E'}(U,E)$ we consider the dual map 
\[
\mathscr{R}_{f}^{t}\colon \F' \to E'^{\star},\;\mathscr{R}_{f}^{t}(f')(e'):=f'(f_{e'}),
\]
where $E'^{\star}$ is the algebraic dual of $E'$. 
We identify $E$ with a linear subspace of $E'^{\star}$ 
by the canonical injection $ x\longmapsto [e'\mapsto e'(x)]=:\langle x, e'\rangle$.

\begin{thm}\label{thm:ext_F_semi_M}
Let $\F$ and $\FE$ be $\varepsilon$-into-compatible, 
$(T^{E}_{m},T^{\K}_{m})_{m\in M}$ a strong, consistent family for $(\mathcal{F},E)$, 
$\F$ a semi-Montel BC-space and $U$ a set of uniqueness for $(T^{\K}_{m},\mathcal{F})_{m\in M}$. 
If
\begin{enumerate}
\item [(i)] $E$ is complete, or if
\item [(ii)] $E$ is sequentially complete and for every $f\in\mathcal{F}_{E'}(U,E)$ and $f'\in\F'$ there is 
a sequence $(f_{n}')_{n\in\N}$ in $\F'$ converging to $f'$ in $\F_{\kappa}'$ such that 
$\mathscr{R}_{f}^{t}(f_{n}')\in E\subset E'^{\star}$ for every $n\in\N$, 
\end{enumerate}
then the restriction map $R_{U,E'}\colon S(\F\varepsilon E)\to \mathcal{F}_{E'}(U,E)$ is surjective.
\end{thm}
\begin{proof} 
Let $f\in\mathcal{F}_{E'}(U,E)$. For $\alpha\in \mathfrak{A}$ we set $B_{\alpha}:=\{x\in E\;|\; p_{\alpha}(x)<1\}$ 
and 
\begin{equation}\label{eq1:ext_F_semi_M}
p_{B^{\circ}_{\alpha}}(y):=\sup_{e'\in B^{\circ}_{\alpha}}|y(e')|\leq\infty,\quad y\in E'^{\star}.
\end{equation}
We remark that $p_{\alpha}(x)=p_{B^{\circ}_{\alpha}}(\langle x, \cdot\rangle)$ for all $x\in E$.
We claim that $\mathscr{R}_{f}^{t}\in L(\F_{\kappa}',E)$. 
Indeed, we have  
\begin{equation}\label{eq2:ext_F_semi_M}
p_{B^{\circ}_{\alpha}}\bigl(\mathscr{R}_{f}^{t}(f')\bigr)
=\sup_{e'\in B^{\circ}_{\alpha}}|f'(f_{e'})|
=\sup_{x\in \mathscr{R}_{f}(B^{\circ}_{\alpha})}|f'(x)|
\leq\sup_{x\in K_{\alpha}}|f'(x)|,\quad f'\in \F',
\end{equation}
where $K_{\alpha}:=\overline{\mathscr{R}_{f}(B^{\circ}_{\alpha})}$. 
Due to \prettyref{prop:ext_F_semi_M} the set $\mathscr{R}_{f}(B^{\circ}_{\alpha})$ is absolutely convex 
and relatively compact, implying that $K_{\alpha}$ is absolutely convex and compact in 
$\F$ by \cite[6.2.1 Proposition, p.\ 103]{Jarchow}. 
Further, we have for all $e'\in E'$ and $(m,x)\in U$
\begin{equation}\label{eq3:ext_F_semi_M}
\mathscr{R}_{f}^{t}(T^{\K}_{m,x})(e')=T^{\K}_{m,x}(f_{e'})=(e'\circ f)(m,x)=\langle f(m,x),e'\rangle 
\end{equation}
and thus $\mathscr{R}_{f}^{t}(T^{\K}_{m,x})\in E$. 

First, let condition (i) be satisfied, i.e.\ $E$ be complete, and $f'\in \F'$. 
The span of $\{T^{\K}_{m,x}\;|\; (m,x)\in U\}$ is dense in
$\F_{\kappa}'$ since $U$ is a set of uniqueness for $\F$. Thus there is a net $(f_{\iota}')$ converging to 
$f'$ in $\F_{\kappa}'$ with $\mathscr{R}_{f}^{t}(f_{\iota}')\in E$ by \eqref{eq3:ext_F_semi_M}.
As
\begin{equation}\label{eq4:ext_F_semi_M}
p_{B^{\circ}_{\alpha}}\bigl(\mathscr{R}_{f}^{t}(f_{\iota}')-\mathscr{R}_{f}^{t}(f')\bigr)
 \underset{\eqref{eq2:ext_F_semi_M}}{\leq} \sup_{x\in K_{\alpha}}|(f_{\iota}'-f')(x)|\to 0,
\end{equation}
for all $\alpha\in \mathfrak{A}$, we gain that $(\mathscr{R}_{f}^{t}(f_{\iota}'))$ is a Cauchy net 
in the complete space $E$.
Hence it has a limit $g\in E$ which coincides with $\mathscr{R}_{f}^{t}(f')$ since
\[
p_{B^{\circ}_{\alpha}}\bigl(g-\mathscr{R}_{f}^{t}(f')\bigr)
\underset{\eqref{eq4:ext_F_semi_M}}{\leq} p_{B^{\circ}_{\alpha}}
   \bigl(g-\mathscr{R}_{f}^{t}(f_{\iota}')\bigr)
 + \sup_{x\in K_{\alpha}}\bigl|(f_{\iota}'-f')(x)\bigr|\to 0.
\]
We conclude that $\mathscr{R}_{f}^{t}(f')\in E$ for every $f'\in \F'$. 

Second, let condition (ii) be satisfied and $f'\in \F'$. Then there is 
a sequence $(f_{n}')$ in $\F'$ converging to $f'$ in $\F_{\kappa}'$ such that 
$\mathscr{R}_{f}^{t}(f_{n}')\in E$ for every $n\in\N$. From \eqref{eq2:ext_F_semi_M} we derive 
that $(\mathscr{R}_{f}^{t}(f_{n}'))$ is a Cauchy sequence in the sequentially complete 
space $E$ converging to $\mathscr{R}_{f}^{t}(f')\in E$.

Therefore we obtain in both cases that $\mathscr{R}_{f}^{t}\in L(\F_{\kappa}',E)$ from \eqref{eq2:ext_F_semi_M}.
This implies $\mathscr{R}_{f}^{t}\in L(\F_{\kappa}', E)=\F\varepsilon E$ (as linear spaces).
We set $F:=S(\mathscr{R}_{f}^{t})$ and obtain from consistency that
\[
T^{E}_{m}(F)(x)=T^{E}_{m}S(\mathscr{R}_{f}^{t})(x)
=\mathscr{R}_{f}^{t}(T^{\K}_{m,x})
\underset{\eqref{eq3:ext_F_semi_M}}{=}f(m,x)
\]
for every $(m,x)\in U$, which means $R_{U,E'}(F)=f$.
\end{proof}

If $E$ is complete and $U$ a set of uniqueness for $(T^{\K}_{m},\mathcal{F})_{m\in M}$ with 
$\{T^{\K}_{m,x}\;|\; (m,x)\in U\}=\{\delta_{x}\;|\;x\in\Lambda\}$, $\Lambda\subset\Omega$, then 
we get \cite[0.1, p.\ 217]{Gramsch1977} as a special case. 

Let us consider a concrete example. 
For an open set $\Omega\subset\R^{d}$, an lcHs $E$ and a linear partial differential operator 
$P(\partial)^{E}\colon\mathcal{C}^{\infty}(\Omega,E)\to\mathcal{C}^{\infty}(\Omega,E)$ 
which is hypoelliptic if $E=\K$ we define the spaces of zero solutions 
\[
\mathcal{C}^{\infty}_{P(\partial)}(\Omega,E):=\{f\in\mathcal{C}^{\infty}(\Omega,E)\;|\;f\in\ker P(\partial)^{E}\}
\]
and the space of bounded zero solutions 
\[
\mathcal{C}^{\infty}_{P(\partial),b}(\Omega,E):=\{f\in\mathcal{C}^{\infty}_{P(\partial)}(\Omega,E)\;|\;
\forall\;\alpha\in\mathfrak{A}:\;\|f\|_{\infty,\alpha}:=\sup_{x\in\Omega}p_{\alpha}(f(x))<\infty\}.
\]
Apart from the topology given by $(\|\cdot\|_{\infty,\alpha})_{\alpha\in\mathfrak{A}}$ 
there is another weighted locally convex topology 
on $\mathcal{C}^{\infty}_{P(\partial),b}(\Omega,E)$ which is of interest, namely, the one induced by the seminorms 
\[
|f|_{\nu,\alpha}:=\sup_{x\in\Omega}p_{\alpha}(f(x))|\nu(x)|,\quad f\in\mathcal{C}^{\infty}_{P(\partial),b}(\Omega,E),
\]
for $\nu\in\mathcal{C}_{0}(\Omega)$ and $\alpha\in\mathfrak{A}$ where $\mathcal{C}_{0}(\Omega)$ 
is the space of $\K$-valued continuous functions on $\Omega$ that vanish at infinity.
We denote by $(\mathcal{C}^{\infty}_{P(\partial),b}(\Omega,E),\beta)$ the space 
$\mathcal{C}^{\infty}_{P(\partial),b}(\Omega,E)$ equipped with the topology $\beta$ induced 
by the seminorms $(|\cdot|_{\nu,\alpha})_{\nu\in\mathcal{C}_{0}(\Omega),\alpha\in\mathfrak{A}}$. 
The topology $\beta$ is called the \emph{strict topology}.

\begin{prop}\label{prop:strict_top_B_r}
Let $\Omega\subset\R^{d}$ be open and $P(\partial)^{\K}$ a hypoelliptic linear partial differential operator. 
Then $(\mathcal{C}^{\infty}_{P(\partial),b}(\Omega),\beta)$ is a $B$-complete semi-Montel space. 
If $E$ is a quasi-complete lcHs, then the family 
$((\partial^{\beta})^{E},\partial^{\beta})_{\beta\in\N_{0}^{d},|\beta|\leq m}$ 
is strong and consistent for $((\mathcal{C}^{\infty}_{P(\partial),b}(\Omega),\beta),E)$ for every $m\in\N_{0}$
and $(\mathcal{C}^{\infty}_{P(\partial),b}(\Omega),\beta)\varepsilon E
\cong (\mathcal{C}^{\infty}_{P(\partial),b}(\Omega,E),\beta)$ via $S$.
\end{prop}
\begin{proof}
It is easy to check that $(\mathcal{C}^{\infty}_{P(\partial),b}(\Omega),\|\cdot\|_{\infty})$ 
is a Banach space and that the closed $\|\cdot\|_{\infty}$-unit ball $B_{\|\cdot\|_{\infty}}$ 
is $\tau_{co}$-compact in $\mathcal{C}^{\infty}_{P(\partial),b}(\Omega)$ 
where $\tau_{co}$ denotes the compact-open topology on $\mathcal{C}^{\infty}_{P(\partial),b}(\Omega)$, i.e.\ the 
topology of uniform convergence on compact subsets of $\Omega$.  
Due to \cite[Proposition 3, p.\ 590]{cooper1971}, saying that the topology $\beta$ coincides 
with the mixed topology $\gamma(\tau_{co},\|\cdot\|_{\infty})$ on 
the space $\mathcal{C}_{b}(\Omega)$ of bounded continuous functions on $\Omega$, 
and \cite[Section I.4, 4.6 Proposition, p.\ 44]{cooper1978}, 
saying that this is inherited by subspaces if $B_{\|\cdot\|_{\infty}}$ is $\tau_{co}$-compact, 
we obtain that $\beta=\gamma(\tau_{co},\|\cdot\|_{\infty})$ 
on $\mathcal{C}^{\infty}_{P(\partial),b}(\Omega)$. 
Thus \cite[Section I.1, 1.13 Proposition, p.\ 11]{cooper1978} yields 
that $(\mathcal{C}^{\infty}_{P(\partial),b}(\Omega),\beta)$ is a semi-Montel space. 
From \cite[2.9 Theorem, p.\ 185]{ruess1977} it follows that the space is $B$-complete.

If $E$ is quasi-complete, then $S$ is a topological isomorphism by \cite[3.1 Bemerkung, p.\ 141]{B2}. 
Clearly the family $((\partial^{\beta})^{E},\partial^{\beta})_{\beta\in\N_{0}^{d},|\beta|\leq m}$ is 
strong which means that $e'\circ f\in\mathcal{C}^{\infty}_{P(\partial),b}(\Omega)$ and 
\[
\partial^{\beta}(e'\circ f)=e'\circ (\partial^{\beta})^{E}f, \quad\beta\in\N_{0}^{d},
\]
for all $e'\in E'$ and $f\in\mathcal{C}^{\infty}_{P(\partial),b}(\Omega,E)$. Let us turn to consistency. 
We already know that $S(u)\in \mathcal{C}^{\infty}_{P(\partial),b}(\Omega,E)$ for every 
$u\in(\mathcal{C}^{\infty}_{P(\partial),b}(\Omega),\beta)\varepsilon E$.  
So we only need to prove that 
\[
(\partial^{\beta})^{E}S(u)(x)=u(\delta_{x}\circ\partial^{\beta}),\quad\beta\in\N_{0}^{d},\;x\in\Omega.
\]
This is a consequence of \cite[Proposition 10, p.\ 1520]{kruse2017} once we have shown that
$u\in(\mathcal{C}^{\infty}_{P(\partial)}(\Omega),\tau_{\mathcal{C}^{\infty}})\varepsilon E$ 
where $\tau_{\mathcal{C}^{\infty}}$ is the usual topology of uniform convergence of partial derivatives 
on compact subsets of $\Omega$.
For $\alpha\in\mathfrak{A}$ there are an absolutely convex, compact 
$K\subset(\mathcal{C}^{\infty}_{P(\partial),b}(\Omega),\beta)$ 
and $C>0$ such that for all $f'\in(\mathcal{C}^{\infty}_{P(\partial),b}(\Omega),\beta)'$ it holds that
\begin{equation}\label{eq:strict_top}
p_{\alpha}(u(f'))\leq C\sup_{f\in K}|f'(f)|.
\end{equation}
From the compactness of $K$ in $(\mathcal{C}^{\infty}_{P(\partial),b}(\Omega),\beta)$ follows that 
$K$ is $\|\cdot\|_{\infty}$-bounded and $\tau_{co}$-compact by 
\cite[Proposition 1 (viii), p.\ 586]{cooper1971} since $(\mathcal{C}^{\infty}_{P(\partial),b}(\Omega),\beta)$ 
carries the induced topology of $(\mathcal{C}_{b}(\Omega),\beta)$ and the strict topology $\beta$ 
is the mixed topology $\gamma(\tau_{co},\|\cdot\|_{\infty})$. 
Let $f'\in(\mathcal{C}^{\infty}_{P(\partial)}(\Omega),\tau_{co})'$. Then there are $M\subset\Omega$ compact 
and $C_{0}>0$ such that 
\[
|f'(f)|\leq C_{0}\sup_{x\in M}|f(x)|
\]
for all $f\in \mathcal{C}^{\infty}_{P(\partial)}(\Omega)$.
Choosing a compactly supported cut-off function $\nu\in\mathcal{C}^{\infty}_{c}(\Omega)$ with $\nu=1$ near $M$, 
we obtain 
\[
|f'(f)|\leq C_{0}\sup_{x\in \Omega}|f(x)||\nu(x)|=C_{0}|f|_{\nu}
\] 
for all $f\in \mathcal{C}^{\infty}_{P(\partial)}(\Omega)$. 
Therefore $f'\in(\mathcal{C}^{\infty}_{P(\partial)}(\Omega),\beta)'$. 
In combination with the $\tau_{co}$-compactness of $K$ it follows from \eqref{eq:strict_top} that 
$u\in(\mathcal{C}^{\infty}_{P(\partial)}(\Omega),\tau_{co})\varepsilon E$. Using that 
$\tau_{co}=\tau_{\mathcal{C}^{\infty}}$ on $\mathcal{C}^{\infty}_{P(\partial)}(\Omega)$
by the hypoellipticity of $P(\partial)^{\K}$ (see e.g.\ \cite[p.\ 690]{F/J/W}), we obtain that 
$u\in(\mathcal{C}^{\infty}_{P(\partial)}(\Omega),\tau_{\mathcal{C}^{\infty}})\varepsilon E$. 
\end{proof}

\begin{rem}\label{rem:strict_top_B_r}
Let $\Omega\subset\R^{d}$ be open and $P(\partial)^{\K}$ a hypoelliptic linear partial differential operator. 
Then $(\mathcal{C}^{\infty}_{P(\partial),b}(\Omega),\beta)$ is non-barrelled if $\tau_{co}$ does not coincide with 
the $\|\cdot\|_{\infty}$-topology by \cite[Section I.1, 1.15 Proposition, p.\ 12]{cooper1978}, e.g.\ 
$(\mathcal{C}^{\infty}_{\overline{\partial},b}(\D),\beta)$ is non-barrelled.
\end{rem}

\begin{cor}
Let $\Omega\subset\R^{d}$ be open, $E$ a complete lcHs and $P(\partial)^{\K}$ 
a hypoelliptic linear  partial differential operator. 
Let $(T^{E}_{m},T^{\K}_{m})_{m\in M}$ be a strong, consistent family for 
$((\mathcal{C}^{\infty}_{P(\partial),b}(\Omega),\beta),E)$
and $U$ a set of uniqueness for $(T^{\K}_{m},(\mathcal{C}^{\infty}_{P(\partial),b}(\Omega),\beta))_{m\in M}$. 
If $f\colon U\to E$ is a function such that there is $f_{e'}\in\mathcal{C}^{\infty}_{P(\partial),b}(\Omega)$ 
for each $e'\in E'$ with $T^{\K}_{m}(f_{e'})(x)=(e'\circ f)(m,x)$ for all $(m,x)\in U$, 
then there is a unique $F\in\mathcal{C}^{\infty}_{P(\partial),b}(\Omega,E)$ 
with $T^{E}_{m}(F)(x)=f(m,x)$ for all $(m,x)\in U$.
\end{cor}
\begin{proof}
Due to \prettyref{prop:strict_top_B_r} $(\mathcal{C}^{\infty}_{P(\partial),b}(\Omega),\beta)$ 
is a $B$-complete semi-Montel space and thus a BC-space. 
Moreover, $(\mathcal{C}^{\infty}_{P(\partial),b}(\Omega),\beta)$ and 
$(\mathcal{C}^{\infty}_{P(\partial),b}(\Omega,E),\beta)$ 
are $\varepsilon$-compatible by \prettyref{prop:strict_top_B_r}, 
yielding our statement by \prettyref{thm:ext_F_semi_M} (i) and \prettyref{prop:injectivity}. 
\end{proof}

In particular, for every $m\in\N_{0}$ the family 
$((\partial^{\beta})^{E},\partial^{\beta})_{\beta\in\N_{0}^{d},|\beta|\leq m}$ 
is strong and consistent for $((\mathcal{C}^{\infty}_{P(\partial),b}(\Omega),\beta),E)$ 
by \prettyref{prop:strict_top_B_r}.

Similarly, we may apply \prettyref{thm:ext_F_semi_M} to the space 
$\mathcal{E}^{\{M_{p}\}}(\Omega,E)$ of ultradifferentiable functions of class $\{M_{p}\}$ of Roumieu-type.
$\mathcal{E}^{\{M_{p}\}}(\Omega)$ is a projective limit of a countable sequence of DFS-spaces by 
\cite[Theorem 2.6, p.\ 44]{Kom7} and thus webbed because being webbed is stable under the formation of 
projective and inductive limits of countable sequences by \cite[5.3.3 Corollary, p.\ 92]{Jarchow}. 
Further, if the sequence $(M_{p})_{p\in\N_{0}}$ satisfies Komatsu's conditions (M.1) and (M.3)' 
(see \cite[p.\ 26]{Kom7}), then $\mathcal{E}^{\{M_{p}\}}(\Omega)$ is a Montel space by 
\cite[Theorem 5.12, p.\ 65-66]{Kom7}. 
The spaces $\mathcal{E}^{\{M_{p}\}}(\Omega)$ and $\mathcal{E}^{\{M_{p}\}}(\Omega,E)$ are $\varepsilon$-compatible 
if (M.1) and (M.3)' hold and $E$ is complete by \cite[Example 16 c), p.\ 1526]{kruse2017}. 
Hence \prettyref{thm:ext_F_semi_M} (i) is applicable.

\begin{rem}
We note that \prettyref{rem:R_well-defined} and \prettyref{thm:ext_F_semi_M} still hold 
if the map $S\colon\F\varepsilon E\to\FE$ is only a linear isomorphism into, 
i.e.\ an isomorphism into of linear spaces, since the topological nature of $\varepsilon$-into-compatibility 
is not used in the proof. 
In particular, this means that it can be applied to the space $\mathcal{M}(\Omega,E)$ of meromorphic functions 
on an open, connected set $\Omega\subset\C$ with values in an lcHs $E$ over $\C$ (see \cite[p.\ 356]{Bonet2002}). 
The space $\mathcal{M}(\Omega)$ is a Montel LF-space, thus webbed, by the proof of 
\cite[Theorem 3 (a), p.\ 294-295]{grosse-erdmann1995} if it is equipped with the locally convex topology 
$\tau_{ML}$ given in \cite[p.\ 292]{grosse-erdmann1995}. 
By \cite[Proposition 6, p.\ 357]{Bonet2002} the map $S\colon\mathcal{M}(\Omega)\varepsilon E\to\mathcal{M}(\Omega,E)$ 
is an isomorphism of linear spaces if $E$ is locally complete and does not contain the space $\C^{\N}$. 
Therefore we can apply \prettyref{thm:ext_F_semi_M} if $E$ is complete and does not contain $\C^{\N}$. 
This augments \cite[Theorem 12, p.\ 12]{jorda2005} where $E$ is assumed to be locally complete 
with suprabarrelled strong dual and $(T^{E},T^{\C})=(\id_{E^{\Omega}},\id_{\C^{\Omega}})$.
\end{rem}

\section*{\texorpdfstring{$\F$}{F(Omega)} a Fr\'{e}chet-Schwartz space and \texorpdfstring{$E$}{E} locally complete}

We recall the following abstract extension result. 

\begin{prop}[{\cite[Proposition 7, p.\ 231]{B/F/J}}]\label{prop:ext_FS_set_uni}
Let $E$ be a locally complete lcHs, $Y$ a Fr\'{e}chet-Schwartz space, $X\subset Y_{b}'(=Y_{\kappa}')$ dense 
and $\mathsf{A}\colon X\to E$ linear.
Then the following assertions are equivalent:
\begin{enumerate}
\item [a)] There is a (unique) extension $\widehat{\mathsf{A}}\in Y\varepsilon E$ of $\mathsf{A}$.
\item [b)] $(\mathsf{A}^{t})^{-1}(Y)$ $(=\{e'\in E'\;|\; e'\circ \mathsf{A}\in Y\})$ determines boundedness in $E$.
\end{enumerate}
\end{prop}

Next, we generalise \cite[Theorem 9, p.\ 232]{B/F/J} using the preceding proposition. The proof 
of the generalisation is simply obtained by replacing the set of uniqueness in the proof of \cite[Theorem 9, p.\ 232]{B/F/J} 
by our more general set of uniqueness. 

\begin{thm}\label{thm:ext_FS_set_uni}
Let $E$ be a locally complete lcHs, $G\subset E'$ determine boundedness and $\F$ and $\FE$ be $\varepsilon$-into-compatible. 
Let $(T^{E}_{m},T^{\K}_{m})_{m\in M}$ be a strong, consistent family for $(\mathcal{F},E)$,
$\F$ a Fr\'{e}chet-Schwartz space and $U$ a set of uniqueness for $(T^{\K}_{m},\mathcal{F})_{m\in M}$. 
Then the restriction map $R_{U,G}\colon S(\F\varepsilon E) \to \mathcal{F}_{G}(U,E)$ is surjective.
\end{thm}
\begin{proof}
Let $f\in \mathcal{F}_{G}(U,E)$. 
We choose $X:=\operatorname{span}\{T^{\K}_{m,x}\;|\;(m,x)\in U\}$ and $Y:=\F$.
Let $\mathsf{A}\colon X\to E$ be the linear map generated by $\mathsf{A}(T^{\K}_{m,x}):=f(m,x)$. 
The map $\mathsf{A}$ is well-defined since $G$ is $\sigma(E',E)$-dense. 
Let $e'\in G$ and $f_{e'}$ be the unique element in $\F$ such that 
$T^{\K}_{m}(f_{e'})(x)=e'\circ \mathsf{A}(T^{\K}_{m,x})$ for all $(m,x)\in U$. This equation 
allows us to consider $f_{e'}$ as a linear form on $X$ 
(by setting $f_{e'}(T^{\K}_{m,x}):=e'\circ \mathsf{A}(T^{\K}_{m,x})$), 
which yields $e'\circ \mathsf{A}\in\F$ for all $e'\in G$. It follows that
$G\subset(\mathsf{A}^{t})^{-1}(Y)$, implying that $(\mathsf{A}^{t})^{-1}(Y)$ determines boundedness. 
Applying \prettyref{prop:ext_FS_set_uni}, there is an extension $\widehat{\mathsf{A}}\in\F\varepsilon E$ of 
$\mathsf{A}$ and we set $F:=S(\widehat{\mathsf{A}})$. We note that 
\[
T^{E}_{m}(F)(x)=T^{E}_{m}S(\widehat{\mathsf{A}})(x)=\widehat{\mathsf{A}}(T^{\K}_{m,x})
=\mathsf{A}(T^{\K}_{m,x})=f(m,x)
\]
for all $(m,x)\in U$ by consistency, yielding $R_{U,G}(F)=f$.
\end{proof}

Let us apply the preceding theorem to our weighted spaces of continuously partially differentiable functions and its subspaces from 
\prettyref{ex:standard_example}.

\begin{cor}\label{cor:weak_strong_CV}
Let $E$ be a locally complete lcHs, $G\subset E'$ determine boundedness,
$\mathcal{V}^{\infty}$ a directed family of weights which is locally bounded away from zero on an open set $\Omega\subset\R^{d}$, 
let $\mathcal{F}(\Omega)$ be a Fr\'{e}chet-Schwartz space and $U\subset\N_{0}^{d}\times\Omega$ a set of uniqueness 
for $(\partial^{\beta},\mathcal{F})_{\beta\in\N_{0}^{d}}$ 
where $\mathcal{F}$ stands for $\mathcal{CV}^{\infty}$, $\mathcal{CV}^{\infty}_{0}$, $\mathcal{CV}^{\infty}_{P(\partial)}$ 
or $\mathcal{CV}^{\infty}_{P(\partial),0}$. 
Then the following holds. 
\begin{enumerate}
\item[a)] If $f\colon U\to E$ is a function such that there is $f_{e'}\in\mathcal{F}(\Omega)$ for each $e'\in G$ 
with $\partial^{\beta}f_{e'}(x)=(e'\circ f)(\beta,x)$ for all $(\beta,x)\in U$, 
then there is a unique $F\in\mathcal{F}(\Omega,E)$ with $(\partial^{\beta})^{E}F(x)=f(\beta,x)$ for all $(\beta,x)\in U$.
\item[b)] If $U\subset\Omega$ and $f\colon U\to E$ is a function such that 
$e'\circ f$ admits an extension $f_{e'}\in\mathcal{F}(\Omega)$ for every $e'\in G$, 
then there is a unique extension $F\in\mathcal{F}(\Omega,E)$ of $f$. 
\item[c)] $\F\varepsilon E\cong \FE$ via $S$ and $\FE=\{f\colon\Omega\to E\;|\;\forall\;e'\in G:\;e'\circ f\in\F \}$.
\end{enumerate}
\end{cor}
\begin{proof}
In all cases $\mathcal{V}^{\infty}$ is locally bounded away from zero and the Fr\'echet space $\mathcal{F}(\Omega)$ is barrelled. 
This implies the consistency of $((\partial^{\beta})^{E},\partial^{\beta})_{\beta\in\N_{0}^{d}}$ 
for $(\mathcal{F},E)$ and the $\varepsilon$-into-compatibility of 
$\F$ and $\FE$ by \prettyref{prop:S_iso_into_standard_example} c) and e).

$\F$ is a Fr\'echet-Schwartz space and $((\partial^{\beta})^{E},\partial^{\beta})_{\beta\in\N_{0}^{d}}$ obviously strong as well, 
which implies that part a) and its special case part b) hold by \prettyref{thm:ext_FS_set_uni} and \prettyref{prop:injectivity}. 
Part c) follows from part b) and \prettyref{prop:weak_strong_principle} since 
$U:=\Omega$ is a set of uniqueness for $(\id_{\K^{\Omega}},\mathcal{F})$. 
\end{proof}

Closed subspaces of Fr\'echet-Schwartz spaces are also Fr\'echet-Schwartz spaces by \cite[Proposition 24.18, p.\ 284]{meisevogt1997}. 
The spaces $\mathcal{CV}^{\infty}_{0}(\Omega)$ and $\mathcal{CV}^{\infty}_{P(\partial),0}(\Omega)$ are closed subspaces of $\mathcal{CV}^{\infty}(\Omega)$ 
and $\mathcal{CV}^{\infty}_{P(\partial)}(\Omega)$, respectively. 
The space $\mathcal{CV}^{\infty}_{P(\partial)}(\Omega)$ is closed in $\mathcal{CV}^{\infty}(\Omega)$ if there is an lcHs $Y$ 
such that $P(\partial)_{\mid\mathcal{CV}^{\infty}(\Omega)}\colon\mathcal{CV}^{\infty}(\Omega)\to Y$ is continuous. For example, this 
is fulfilled if the coefficients of $P(\partial)$ belong to $\mathcal{C}(\Omega)$, 
in particular if $P(\partial):=\Delta$ or $\overline{\partial}$, with $Y:=(\mathcal{C}(\Omega),\tau_{co})$ 
due to $\mathcal{V}^{\infty}$ being locally bounded away from zero.
If $\omega_{m}=M_{m}\times\Omega$ with $M_{m}=\{\beta\in\N_{0}^{d}\;|\;|\beta|\leq m\}$ and $J$ is countable, 
then $\mathcal{CV}^{\infty}(\Omega)$ is a Fr\'echet space by 
\cite[Proposition 3.7, p.\ 240]{kruse2018_2}. Conditions on the weights $\mathcal{V}^{\infty}$ 
which make $\mathcal{CV}^{\infty}(\Omega)$ and its closed subspaces nuclear Fr\'echet spaces, in particular, Fr\'echet-Schwartz spaces 
can be found in \cite[Theorem 3.1, p.\ 188]{kruse2018_4}. For the case $\omega_{m}=\N_{0}^{d}\times\Omega$ 
see the references given in \cite[p.\ 174]{kruse2018_4}.

The preceding corollary can be applied to the Schwartz space $\mathcal{CV}^{\infty}(\R^{d}):=\mathcal{S}(\R^{d})$ and 
improves the $\varepsilon$-compatibility given in 
\cite[Proposition 9, p.\ 108, Th\'{e}or\`{e}me 1, p.\ 111]{Schwartz1955} ($E$ quasi-complete) 
and \cite[Theorem 4.9 a), p.\ 371]{kruse2018_1} ($E$ sequentially complete). 
An application to the Fr\'{e}chet-Schwartz space $\mathcal{CV}^{\infty}(\Omega):=\mathcal{E}^{(M_{p})}(\Omega)$ 
of ultradifferentiable functions of class $(M_{p})$ of Beurling-type (see \cite[Theorem 2.6, p.\ 44]{Kom7}) 
also improves \cite[Theorem 3.10, p.\ 678]{Kom9} since Komatsu's conditions (M.0), (M.1), (M.2)' and (M.3)' (see 
\cite[p.\ 26]{Kom7} and \cite[p.\ 653]{Kom9}) are not needed and the condition that $E$ is sequentially complete is weakened 
to local completeness. 

\begin{rem}\label{rem:set_unique_CV}
Let $\mathcal{V}^{\infty}$ be a directed family of weights which is locally bounded away from zero on an open set $\Omega\subset\R^{d}$. 
\begin{enumerate}
\item[a)] Then any dense set $U\subset\Omega$ is a set of uniqueness for $(\id_{\K^\Omega},\mathcal{F})$ with 
$\mathcal{F}=\mathcal{CV}^{\infty}$, $\mathcal{CV}^{\infty}_{0}$, $\mathcal{CV}^{\infty}_{P(\partial)}$ 
or $\mathcal{CV}^{\infty}_{P(\partial),0}$ due to continuity.
\item[b)] Let $\Omega$ be connected and $x_{0}\in\Omega$. Then $U:=\{(e_{n},x)\;|\;1\leq n\leq d,\,x\in\Omega\}\cup\{(0,x_{0})\}$ 
is a set of uniqueness for $(\partial^{\beta},\mathcal{F})_{\beta\in\N_{0}}$ by the mean value theorem with $\mathcal{F}$ from a).
\item[c)] Let $\K:=\R$, $d:=1$, $\Omega:=(a,b)\subset\R$, $g\colon (a,b)\to \N$ and $x_{0}\in(a,b)$. 
Then $U:=\{(g(x),x)\;|\;x\in(a,b)\}\cup\{(n,x_{0})\;|\;n\in\N_{0}\}$ is a set of uniqueness 
for $(\partial^{\beta},\mathcal{F})_{\beta\in\N_{0}}$ with $\mathcal{F}$ from a). 
Indeed, if $f\in\mathcal{F}(\Omega)$ and $0=\partial^{g(x)}f(x)$ for all $x\in(a,b)$, then $f$ is a polynomial 
by \cite[Chap.\ 11, Theorem, p.\ 53]{donoghue1969}. If, in addition, $0=\partial^{n}f(x_{0})$ for all $n\in\N_{0}$, 
then the polynomial $f$ must vanish on the whole interval $\Omega$.
\item[d)] Let $\Omega\subset\C$ be connected. Then any set $U\subset\Omega$ with an accumulation point in $\Omega$ is a set 
of uniqueness for $(\id_{\C^\Omega},\mathcal{CV}^{\infty}_{\overline{\partial}})$ by the identity theorem for holomorphic functions.
\item[e)] Let $\Omega\subset\C$ be connected and $z_{0}\in\Omega$. Then $U:=\{(n,z_{0})\;|\;n\in\N_{0}\}$ is a set of uniqueness 
for $(\partial^{n}_{\C},\mathcal{CV}^{\infty}_{\overline{\partial}})_{n\in\N_{0}}$ by local power series expansion 
and the identity theorem where $\partial^{n}_{\C}$ denotes the $n$-th complex differential operator, 
which is related to the real partial differential operators by 
\begin{equation}\label{eq:complex_deriv}
 \partial^{\beta}f(z)=i^{\beta_{2}}\partial^{|\beta|}_{\C}f(z),\quad\beta:=(\beta_{1},\beta_{2})\in\N_{0}^{2},\;z\in\Omega
\end{equation}
for all $f\in\mathcal{C}^{\infty}_{\overline{\partial}}(\Omega)$ (see e.g.\ \cite[3.4 Lemma, p.\ 17]{ich}).
\item[f)] Let $\Omega\subset\R^{d}$ be connected. Then any non-empty open set $U\subset\Omega$ is a set of uniqueness for 
$(\id_{\K^\Omega},\mathcal{CV}^{\infty}_{\Delta})$ by the identity theorem for harmonic functions (see e.g.\ \cite[Theorem 5, p.\ 218]{gustin1948}).
\item[g)] Further examples of sets of uniqueness for $(\id_{\K^\Omega},\mathcal{CV}^{\infty}_{\Delta})$ are given in \cite{kokurin2015}.
\end{enumerate}
\end{rem}

In part e) a special case of \prettyref{rem:Schauder_coeff_set_uni} b) is used, namely, 
that $\mathcal{CW}^{\infty}_{\overline{\partial}}(\D_{r}(z_{0}))$ 
has a Schauder basis with associated coefficient functionals $(\delta_{z_{0}}\circ\partial^{n}_{\C})_{n\in\N_{0}}$ 
where $0<r\leq\infty$ is such that $\D_{r}(z_{0})\subset\Omega$.
In order to obtain some sets of uniqueness which are more sensible w.r.t.\ the family of weights $\mathcal{V}^{\infty}$, 
we turn to entire and harmonic functions fulfilling some growth conditions. 
For a family $\mathcal{V}^{\infty}:=(\nu_{j})_{j\in\N}$ of continuous weights on $\R^{d}$, 
a hypoelliptic linear partial differential operator
$P(\partial)$ and an lcHs $E$ we define the weighted space of zero solutions 
\[
\mathcal{AV}^{\infty}_{P(\partial)}(\R^{d},E):=\{f\in\mathcal{C}^{\infty}_{P(\partial)}(\R^{d},E)\;|\;\forall\;j\in\N,\,\alpha\in\mathfrak{A}:\;
|f|_{j,\alpha}^{\sim}<\infty\}
\]
where 
\[
|f|_{j,\alpha}^{\sim}:=\sup_{x\in\R^{d}}p_{\alpha}(f(x))\nu_{j}(x).
\]
If $P(\partial)=\overline{\partial}$, $d=2$ and $\K=\C$, or $P(\partial)=\Delta$ and there is $0\leq\tau<\infty$ such that $\nu_{j}(x)=\exp(-(\tau+\tfrac{1}{j})|x|)$, $x\in\R^{d}$, for all $j\in\N$, 
then $A_{\overline{\partial}}^{\tau}(\C,E):=\mathcal{AV}^{\infty}_{\overline{\partial}}(\C,E)$ is the space of entire and $A_{\Delta}^{\tau}(\R^{d},E):=\mathcal{AV}^{\infty}_{\Delta}(\R^{d},E)$ 
the space of harmonic functions of exponential type $\tau$. If $\tau=0$, then the elements of these spaces are also called functions of infra-exponential type.

\begin{cond}\label{cond:weights}
Let $\mathcal{V}^{\infty}:=(\nu_{j})_{j\in\N}$ be an increasing family of continuous weights on $\R^{d}$. 
Let there be $r\colon\R^{d}\to (0,1]$ and for any $j\in\N$ let there be $\psi_{j}\in L^{1}(\R^{d})$, $\psi_{j}>0$, 
and $I_{m}(j)\geq j$ and $A_{m}(j)>0$ such that for any $x\in\R^{d}$:
\begin{enumerate}
  \item [$(\alpha.1)$] $\sup_{\zeta\in\R^{d},\,\|\zeta\|_{\infty}\leq r(x)}\nu_{j}(x+\zeta)
  \leq A_{1}(j)\inf_{\zeta\in\R^{d},\,\|\zeta\|_{\infty}\leq r(x)}\nu_{I_{1}(j)}(x+\zeta)$
  \item [$(\alpha.2)$] $\nu_{j}(x)\leq A_{2}(j)\psi_{j}(x)\nu_{I_{2}(j)}(x)$
  \item [$(\alpha.3)$] $\nu_{j}(x)\leq A_{3}(j)r(x)\nu_{I_{3}(j)}(x)$
\end{enumerate}
\end{cond}

The preceding condition is a special case of \cite[Condition 2.1, p.\ 176]{kruse2018_4} with $\Omega:=\Omega_{n}:=\R^{d}$ for all $n\in\N$.
If $\mathcal{V}^{\infty}$ fulfils \prettyref{cond:weights} and we set $\mathcal{V}^{\infty,\ast}:=(\nu_{j,m})_{j\in\N,m\in\N_{0}}$ where 
$\nu_{j,m}\colon\{\beta\in\N_{0}^{d}\;|\;|\beta|\leq m\}\times\Omega$, $\nu_{j,m}(\beta,x):=\nu_{j}(x)$, 
then $\mathcal{CV}^{\infty,\ast}(\R^{d})$ and its closed subspace $\mathcal{CV}^{\infty,\ast}_{P(\partial)}(\R^{d})$ 
for $P(\partial)$ with continuous coefficients
are nuclear by \cite[Theorem 3.1, p.\ 188]{kruse2018_4} in combination with 
\cite[Remark 2.7, p.\ 178-179]{kruse2018_4} and Fr\'echet spaces 
by \cite[Proposition 3.7, p.\ 240]{kruse2018_2}. 

\begin{rem}\label{rem:ex_NF_cond_weights}
Let $0\leq \tau<\infty$. Then $\mathcal{V}^{\infty}:=(\nu_{j})_{j\in\N}$ given by 
$\nu_{j}(x):=\exp(-(\tau+\tfrac{1}{j})|x|)$, $x\in\R^{d}$, 
fulfils \prettyref{cond:weights} by \cite[Example 2.8 (iii), p.\ 179]{kruse2018_4}. 
Further examples of families of weights fulfilling \prettyref{cond:weights} 
can be found in \cite[Example 2.8, p.\ 179]{kruse2018_4} and \cite[1.5 Examples, p.\ 205]{meise1987}.
\end{rem}

Now, we can use \prettyref{cor:weak_strong_CV} and these conditions to show that $\mathcal{AV}^{\infty}_{P(\partial)}(\R^{d},E)$ 
coincides as a locally convex space with $\mathcal{CV}^{\infty,\ast}_{P(\partial)}(\R^{d},E)$ if $P(\partial)=\overline{\partial}$ 
or $\Delta$ and $E$ is locally complete, which is used in the next section as well.

\begin{prop}\label{prop:AV_CV_coincide_hol_har}
Let $E$ be a locally complete lcHs. If $\mathcal{V}^{\infty}$ fulfils \prettyref{cond:weights}, then $\mathcal{AV}^{\infty}_{\overline{\partial}}(\C)$ 
and $\mathcal{AV}^{\infty}_{\Delta}(\R^{d})$ are nuclear Fr\'echet spaces and 
$\mathcal{AV}^{\infty}_{\overline{\partial}}(\C,E)=\mathcal{CV}^{\infty,\ast}_{\overline{\partial}}(\C,E)$ and 
$\mathcal{AV}^{\infty}_{\Delta}(\R^{d},E)=\mathcal{CV}^{\infty,\ast}_{\Delta}(\R^{d},E)$ as locally convex spaces.
\end{prop}
\begin{proof}
Let $P(\partial):=\overline{\partial}$ ($d:=2$ and $\K:=\C$) or $P(\partial):=\Delta$. 
First, we show that $\mathcal{AV}^{\infty}_{P(\partial)}(\R^{d})=\mathcal{CV}^{\infty,\ast}_{P(\partial)}(\R^{d})$ 
as locally convex spaces, which implies that 
$\mathcal{AV}^{\infty}_{P(\partial)}(\R^{d})$ is a nuclear Fr\'echet space as $\mathcal{CV}^{\infty,\ast}_{P(\partial)}(\R^{d})$ 
is such a space. 
Let $f\in\mathcal{AV}^{\infty}_{\overline{\partial}}(\C)$, $j\in\N$, $m\in\N_{0}$, $z\in\C$ 
and $\beta:=(\beta_{1},\beta_{2})\in\N_{0}^{2}$. 
Then it follows from $\|\cdot\|_{\infty}\leq |\cdot|$ and Cauchy's inequality that
\begin{align*}
 |\partial^{\beta}f(z)|\nu_{j}(z)&\underset{\mathclap{\eqref{eq:complex_deriv}}}{=}\;|i^{\beta_{2}}\partial^{|\beta|}_{\C}f(z)|\nu_{j}(z)
  \leq \frac{|\beta|!}{r(z)^{|\beta|}}\sup_{|w-z|=r(z)}|f(w)|\nu_{j}(z)\\
 &\underset{\mathclap{(\alpha.3)}}{\leq}\;|\beta|!C(j,|\beta|)\sup_{|w-z|=r(z)}|f(w)|\nu_{B_{3}(j)}(z)\\
 &\underset{\mathclap{(\alpha.1)}}{\leq}\;|\beta|!C(j,|\beta|)A_{1}(B_{3}(j))\sup_{|w-z|=r(z)}|f(w)|\nu_{I_{1}B_{3}(j)}(w)\\
 &\leq |\beta|!C(j,|\beta|)A_{1}(B_{3}(j))|f|_{I_{1}B_{3}(j)}^{\sim}
\end{align*}
where $C(j,|\beta|):=A_{3}(j)A_{3}(I_{3}(j))\cdots A_{3}((B_{3}-1)(j))$ and 
$B_{3}-1$ is the $(|\beta|-1)$-fold composition of $I_{3}$. Choosing $k:=\max_{|\beta|\leq m}I_{1}B_{3}(j)$, 
it follows that 
\[
 |f|_{j,m}\leq\sup_{|\beta|\leq m}|\beta|!C(j,|\beta|)A_{1}(B_{3}(j))|f|_{k}^{\sim}<\infty
\]
and thus $f\in\mathcal{CV}^{\infty,\ast}_{\overline{\partial}}(\C)$ and 
$\mathcal{AV}^{\infty}_{\overline{\partial}}(\C)=\mathcal{CV}^{\infty,\ast}_{\overline{\partial}}(\C)$ as locally convex spaces.
In the case $P(\partial)=\Delta$ an analogous proof works due to Cauchy's inequality for harmonic functions, i.e.\ 
for all $f\in\mathcal{AV}^{\infty}_{\Delta}(\R^{d})$, $j\in\N$, $x\in\R^{d}$ and $\beta\in\N_{0}^{d}$ it holds that
\[
 |\partial^{\beta}f(x)|\nu_{j}(x)\leq \Bigl(\frac{d|\beta|}{r(x)}\Bigr)^{|\beta|}\sup_{|w-x|<r(x)}|f(w)|\nu_{j}(x)
\]
(see e.g.\ \cite[Theorem 2.10, p.\ 23]{gilbarg_trudinger2001}).

The nuclear Fr\'echet space $\mathcal{AV}^{\infty}_{P(\partial)}(\R^{d})$ is a Fr\'echet-Schwartz space and the set
$U:=\Omega$ is a set of uniqueness for $(\id_{\K^{\R^{d}}},\mathcal{AV}^{\infty}_{P(\partial)})$. The pair
$(\id_{E^{\R^{d}}},\id_{\K^{\R^{d}}})$ is a strong, consistent generator for $(\mathcal{AV}^{\infty}_{P(\partial)},E)$.
Indeed, we only need to check condition (i) of \prettyref{def:cons_strong} a) and b), respectively, 
which is satisfied by \prettyref{prop:S_iso_into_standard_example} d) since 
$\mathcal{AV}^{\infty}_{P(\partial)}(\R^{d})=\mathcal{CV}^{\infty,\ast}_{P(\partial)}(\R^{d})$ is barrelled
and $\mathcal{V}^{\infty,\ast}$ locally bounded away from zero. This yields the $\varepsilon$-into-compatibility of 
$\mathcal{AV}^{\infty}_{P(\partial)}(\R^{d})$ and $\mathcal{AV}^{\infty}_{P(\partial)}(\R^{d},E)$ by \prettyref{thm:S_iso_into} as well. 
It follows from \prettyref{thm:ext_FS_set_uni} and \prettyref{prop:weak_strong_principle} that 
$\mathcal{AV}^{\infty}_{P(\partial)}(\R^{d})\varepsilon E\cong\mathcal{AV}^{\infty}_{P(\partial)}(\R^{d},E)$ via $S$. Hence we have
\[
\mathcal{AV}^{\infty}_{P(\partial)}(\R^{d},E)\cong \mathcal{AV}^{\infty}_{P(\partial)}(\R^{d})\varepsilon E
\cong\mathcal{CV}^{\infty,\ast}_{P(\partial)}(\R^{d})\varepsilon E
\cong\mathcal{CV}^{\infty,\ast}_{P(\partial)}(\R^{d},E)
\]
by \prettyref{cor:weak_strong_CV} b) since $\mathcal{AV}^{\infty}_{P(\partial)}(\R^{d})=\mathcal{CV}^{\infty,\ast}_{P(\partial)}(\R^{d})$ 
as locally convex spaces. 
Clearly, the isomorphism $\mathcal{AV}^{\infty}_{P(\partial)}(\R^{d},E)\cong\mathcal{CV}^{\infty,\ast}_{P(\partial)}(\R^{d},E)$ is the identity. 
\end{proof}

Hence we may complement our list in \prettyref{rem:set_unique_CV} by some more examples for spaces of functions 
of exponential type $0\leq\tau<\infty$.

\begin{rem}\label{rem:ex_set_uni_AP}
The following sets $U\subset\C$ are sets of uniqueness for $(\id_{\C^{\C}},A_{\overline{\partial}}^{\tau})$.
\begin{enumerate}
\item[a)] If $\tau<\pi$, then $U:=\N_{0}$ is a set of uniqueness by \cite[9.2.1 Carlson's theorem, p.\ 153]{boas1954}. 
\item[b)] Let $\delta>0$ and $(\lambda_{n})_{n\in\N}\subset(0,\infty)$ such that $\lambda_{n+1}-\lambda_{n}>\delta$ for all $n\in\N$. 
Then $U:=(\lambda_{n})_{n\in\N}$ is a set of uniqueness if $\limsup_{r\to\infty}r^{-2\tau/\pi}\psi(r)=\infty$ where 
$\psi(r):=\exp(\sum_{\lambda_{n}<r}\lambda_{n}^{-1})$, $r>0$, by \cite[9.5.1 Fuchs's theorem, p.\ 157-158]{boas1954}.
\end{enumerate}
The following sets $U$ are sets of uniqueness for $(\partial^{n}_{\C},A_{\overline{\partial}}^{\tau})_{n\in\N_{0}}$.
\begin{enumerate}
\item[c)] Let $(\lambda_{n})_{n\in\N_{0}}\subset\C$ with $|\lambda_{n}|<1$ for all $n\in\N_{0}$. If $\tau<\ln(2)$, 
then $U:=\{(n,\lambda_{n})\;|\;n\in\N_{0}\}$ is a set of uniqueness by \cite[9.11.1 Theorem, p.\ 172]{boas1954}. 
If $\tau<\ln(2+\sqrt{3})$, then $U:=\{(2n+1,0)\;|\;n\in\N_{0}\}\cup\{(2n,\lambda_{n})\;|\;n\in\N_{0}\}$ is 
a set of uniqueness by \cite[9.11.3 Theorem, p.\ 173]{boas1954}.
\item[d)] Let $(\lambda_{n})_{n\in\N_{0}}\subset\C$ with $\limsup_{n\to\infty}n^{-1}\sum_{k=1}^{n}|\lambda_{k}|\leq 1$. 
If $\tau<e^{-1}$, then $U:=\{(n,\lambda_{n})\;|\;n\in\N_{0}\}$ is a set of uniqueness by \cite[9.11.4 Theorem, p.\ 173]{boas1954}.
\end{enumerate}
The following sets $U\subset\R^{d}$ are sets of uniqueness for $(\id_{\R^{\R^{d}}},A_{\Delta}^{\tau})$.
\begin{enumerate}
\item[e)] Let $d:=2$. If there is $k\in\N$ with $\tau<\pi/k$, then $U:=\Z\cup(\Z+ik)$ is a set of uniqueness by \cite[Theorem 1, p.\ 425]{boas1972}.
\item[f)] Let $d:=2$. If $\tau<\pi$ and $\theta\notin\pi\Q$, then $U:=\Z\cup(e^{i\theta}\Z)$ is a set of uniqueness 
by \cite[Theorem 2, p.\ 426]{boas1972}.
\item[g)] If $\tau<\pi$, then $U:=\{0,1\}\times\Z^{d-1}$ is a set of uniqueness by \cite[Corollary 1.8, p.\ 312]{rao1974}. 
\item[h)] If $\tau<\pi$ and $a\in\R$ with $|a|\leq\sqrt{1/(d-1)}$, then $U:=\Z^{d-1}\times\{0,a\}$ is a set of uniqueness by 
\cite[Theorem A, p.\ 335]{zeilberger1976}.
\item[i)] Further examples of sets of uniqueness can be found in \cite{armitage1979}.
\end{enumerate}
The following sets $U$ are sets of uniqueness for $((\partial^{\beta})^{\R},A_{\Delta}^{\tau})_{\beta\in\N_{0}^{d}}$.
\begin{enumerate}
\item[j)] If $\tau<\pi$, then $U:=\{(\beta,(x,0))\;|\;\beta\in\{0,e_{d}\},\,x\in\Z^{d-1}\}$ is 
a set of uniqueness by \cite[Theorem B, p.\ 335]{zeilberger1976}. Further examples can be found in \cite{armitage1979}.
\end{enumerate}
\end{rem}

We close this section by an examination of the space
\begin{align*}
\mathcal{E}_{0}(E):=\{f\in\mathcal{C}^{\infty}((0,1),E)\;|\;\forall\;k\in\N_{0}:\;&(\partial^{k})^{E}f\;
\text{cont.\ extendable on}\;[0,1]\\
&\text{and}\;(\partial^{k})^{E}f(1)=0\}
\end{align*}
where $(\partial^{k})^{E}f(1):=\lim_{x\nearrow 1}(\partial^{k})^{E}f(x)$ and which we equip with the system of seminorms given by
\[
|f|_{m,\alpha}:=\sup_{\substack{x\in(0,1)\\k\in\N_{0},k\leq m}}p_{\alpha}\bigl((\partial^{k})^{E}f(x)\bigr),\quad f\in \mathcal{E}_{0}(E),
\]
for $m\in\N_{0}$ and $\alpha\in\mathfrak{A}$. We need the following weak-strong principle in our last section.

\begin{cor}\label{cor:weak_strong_E_0}
Let $E$ be a locally complete lcHs and $G\subset E'$ determine boundedness. Then
$\mathcal{E}_{0}\varepsilon E\cong\mathcal{E}_{0}(E)$ via $S$ and 
$\mathcal{E}_{0}(E)=\{f\colon(0,1)\to E\;|\;\forall\;e'\in G:\;e'\circ f\in\mathcal{E}_{0} \}$.
\end{cor}
\begin{proof}
Analogously to the proof of \cite[Example 20, p.\ 1529]{kruse2017} we may deduce that 
$((\partial^{k})^{E},\partial^{k})_{k\in\N_{0}}$ is a strong, consistent generator for $(\mathcal{E}_{0},E)$ 
since $\mathcal{E}_{0}$ is a Fr\'echet-Schwartz space by \cite[Example 28.9 (5), p.\ 350]{meisevogt1997}, in particular, barrelled. 
Therefore $\mathcal{E}_{0}$ and $\mathcal{E}_{0}(E)$ are $\varepsilon$-into-compatible by \prettyref{thm:S_iso_into} 
and we derive our statement from \prettyref{thm:ext_FS_set_uni} and \prettyref{prop:weak_strong_principle} with $U:=(0,1)$.
\end{proof}
\section{Extension of locally bounded functions}
In order to obtain an affirmative answer to \prettyref{que:surj_restr_set_unique} 
for general separating subspaces of $E'$ we have to restrict to the spaces $\FV$ from \prettyref{def:standard_space} 
and a certain class of sets of uniqueness.

\begin{defn}[{fix the topology}]\label{def:fix_top_1}
Let $\FV$ be a $\dom$-space. $U\subset\bigcup_{m\in M}\{m\}\times\omega_{m}$ \emph{fixes the topology} in $\FV$ if 
for every $j\in J$ and $m\in M $ there are $i\in J$, $k\in M$ and $C>0$ such that 
\[
|f|_{j,m}\leq C \sup_{\substack{x\in\omega_{k}\\(k,x)\in U}}|T^{\K}_{k}(f)(x)|\nu_{i,k}(x),\quad
f\in \FV .
\]
\end{defn}

In particular, $U$ is a set of uniqueness if it fixes the topology. The present definition of fixing 
the topology is a generalisation of \cite[Definition 13, p.\ 234]{B/F/J}. 
Sets that fix the topology appear under several different notions. 
Rubel and Shields call them \emph{dominating} in \cite[4.10 Definition, p.\ 254]{rubelshields1966} 
in the context of bounded holomorphic functions. 
In the context of the space of holomorphic functions with the compact-open topology studied by Grosse-Erdmann 
\cite[p.\ 401]{grosse-erdmann2004} 
they are said to \emph{determine locally uniform convergence}. 
Ehrenpreis \cite[p.\ 3,4,13]{ehrenpreis1970} (cf.\ \cite[Definition 3.2, p.\ 166]{schneider1974}) 
refers to them as \emph{sufficient sets}
when he considers inductive limits of weighted spaces of entire resp.\ holomorphic functions, including the case of Banach spaces. 
In the case of Banach spaces sufficient sets coincide with \emph{weakly sufficient sets} defined by Schneider 
\cite[Definition 2.1, p.\ 163]{schneider1974} 
(see e.g.\ \cite[\S7, 1), p.\ 547]{korobenik1987}) and these notions are extended beyond spaces of holomorphic functions 
by Korobe{\u{\i}}nik \cite[p.\ 531]{korobenik1987}.
Seip \cite[p.\ 93]{seip1992b} uses the term \emph{sampling sets} in the context of weighted Banach spaces 
of holomorphic functions whereas Beurling uses the 
term \emph{balayage} in \cite[p.\ 341]{beurling1989} and \cite[Definition, p.\ 343]{beurling1989}.
Leibowitz \cite[Exercise 4.1.4, p.\ 53]{leibowitz1970}, Stout \cite[7.1 Definition, p.\ 36]{stout1971} 
and Globevnik \cite[p.\ 291-292]{globevnik1979} 
call them \emph{boundaries} in the context of subalgebras of the algebra $\mathcal{C}(\Omega,\C)$ of complex-valued continuous functions 
on a compact Hausdorff space $\Omega$ with sup-norm. 
Fixing the topology is also connected to the notion of \emph{frames} used by Bonet et al.\ in \cite{bonet2017}. Let us set
\[
\ell\mathcal{V}(U,E):=\{f\colon U\to E\;|\;\forall\;j\in J,m\in M,\alpha\in\mathfrak{A}:\;\|f\|_{j,m,\alpha}<\infty\}
\]
with
\[
\|f\|_{j,m,\alpha}:=\sup_{\substack{x\in\omega_{m}\\(m,x)\in U}}p_{\alpha}(f(m,x))\nu_{j,m}(x)
\]
for an lcHs $E$ and a set $U$ which fixes the topology in $\FV$. If $U$ is countable, 
the inclusion $\ell\mathcal{V}(U)\hookrightarrow\K^{U}$ continuous where 
$\K^{U}$ is equipped with the topology of pointwise convergence and $\ell\mathcal{V}(U)$ contains the space of sequences (on $U$)
with compact support as a linear subspace, then $(T^{\K}_{k,x})_{(k,x)\in U}$ is an $\ell\mathcal{V}(U)$-frame in the sense of 
\cite[Definition 2.1, p.\ 3]{bonet2017}.

\begin{defn}[{$lb$-restriction space}]
Let $\FV$ be a $\dom$-space, $U$ fix the topology in $\FV$ and $G\subset E'$ a separating subspace. We set 
\[
N_{U,i,k}(f):=\{f(k,x)\nu_{i,k}(x)\;|\;x\in\omega_{k},\,(k,x)\in U\}
\]
for $i\in J$, $k\in M$ and $f\in\mathcal{FV}_{G}(U,E)$ and 
\begin{align*}
\mathcal{FV}_{G}(U,E)_{lb}:=&\{f\in\mathcal{FV}_{G}(U,E)\;|\;\forall\;i\in J,\,k\in M:\;N_{U,i,k}(f)\;\text{bounded in}\; E\}\\
=&\mathcal{FV}_{G}(U,E)\cap\ell\mathcal{V}(U,E).
\end{align*}
\end{defn}

Consider a set $U$ which fixes the topology in $\FV$, a separating subspace $G\subset E'$ and 
a strong, consistent family $(T^{E}_{m},T^{\K}_{m})_{m\in M}$ for $(\mathcal{FV},E)$.
For $u\in \FV\varepsilon E$ set $f:=S(u)\in\FVE$ by \prettyref{thm:S_iso_into}. Then 
we have $R_{U,G}(f)\in \mathcal{FV}_{G}(U,E)$ with $f:=S(u)$ 
by \prettyref{rem:R_well-defined} and for $i\in J$ and $k\in M$
\[
\sup_{y\in N_{U,i,k}(R_{U,G}(f))}p_{\alpha}(y)
=\sup_{\substack{x\in\omega_{k}\\(k,x)\in U}}p_{\alpha}(T^{E}_{k}(f)(x))\nu_{i,k}(x)
\leq |f|_{i,k,\alpha}<\infty
\]
for all $\alpha\in\mathfrak{A}$, implying the boundedness of $N_{U,i,k}(R_{U,G}(f))$ in $E$. 
Thus $R_{U,G}(f)\in\mathcal{FV}_{G}(U,E)_{lb}$ and the injective linear map
\[
R_{U,G}\colon S(\FV\varepsilon E)\to \mathcal{FV}_{G}(U,E)_{lb},\;f\mapsto (T^{E}_{m}(f)(x))_{(m,x)\in U},
\]
is well-defined. 

\begin{que}
Let $G\subset E'$ be a separating subspace, 
$(T^{E}_{m},T^{\K}_{m})_{m\in M}$ a strong, consistent generator for $(\mathcal{FV},E)$
and $U$ fix the topology in $\FV$.
When is the injective restriction map 
\[
R_{U,G}\colon S(\FV\varepsilon E)\to \mathcal{FV}_{G}(U,E)_{lb},\;f\mapsto (T^{E}_{m}(f)(x))_{(m,x)\in U},
\]
surjective?
\end{que}

If $G\subset E'$ determines boundedness and $U$ fixes the topology in $\FV$, then the preceding question and \prettyref{que:surj_restr_set_unique} 
coincide.

\begin{rem}\label{rem:rest_spaces_coincide}
Let $G\subset E'$ determine boundedness, 
$(T^{E}_{m},T^{\K}_{m})_{m\in M}$ a strong, consistent generator for $(\mathcal{FV},E)$
and $U$ fix the topology in $\FV$. Then 
\[
\mathcal{FV}_{G}(U,E)_{lb}=\mathcal{FV}_{G}(U,E).
\]
\end{rem}
\begin{proof}
We only need to show that the inclusion $\supset$ holds. Let $f\in\mathcal{FV}_{G}(U,E)$. Then there is $f_{e'}\in\FV$ with 
$T^{\K}_{m}(f_{e'})(x)=(e'\circ f)(m,x)$ for all $(m,x)\in U$ and 
\[
\sup_{y\in N_{U,i,k}(f)}|e'(y)|=\sup_{\substack{x\in\omega_{k}\\(k,x)\in U}}|(e'\circ f)(k,x)|\nu_{i,k}(x)\leq |f_{e'}|_{i,k}<\infty
\]
for each $e'\in G$, $i\in J$ and $k\in M$. Since $G\subset E'$ determines boundedness, this means that $N_{U,i,k}(f)$ is bounded in $E$ 
and hence $f\in\mathcal{FV}_{G}(U,E)_{lb}$.
\end{proof}

\section*{\texorpdfstring{$\FV$}{FV(Omega)} arbitrary and \texorpdfstring{$E$}{E} a semi-Montel space}

\begin{defn}[{generalised Schwartz space}]
We call an lcHs $E$ a generalised Schwartz space if every bounded set in $E$ is already precompact.
\end{defn}

In particular, semi-Montel spaces and Schwartz spaces are generalised Schwartz spaces by \cite[10.4.3 Corollary, p.\ 202]{Jarchow}. 
Conversely, a generalised Schwartz space is a Schwartz space if it is quasi-normable 
by \cite[10.7.3 Corollary, p.\ 215]{Jarchow}.

\begin{prop}\label{prop:ext_E_semi_M}
Let $E$ be an lcHs,
$\FV$ a $\dom$-space and $U$ fix the topology in $\FV$.
Then $\mathscr{R}_{f}\in L(E_{b}',\FV)$ 
and $\mathscr{R}_{f}(B_{\alpha}^{\circ})$ is bounded in $\FV$  
for every $f\in\mathcal{FV}_{E'}(U,E)_{lb}$ and $\alpha\in\mathfrak{A}$ 
where $B_{\alpha}:=\{x\in E\;|\;p_{\alpha}(x)<1\}$ and $\mathscr{R}_{f}$ 
is the map from \prettyref{rem:R_f}. In addition, if $E$ is a generalised Schwartz space, 
then $\mathscr{R}_{f}\in L(E_{\tau_{pc}}',\FV)$ 
and $\mathscr{R}_{f}(B_{\alpha}^{\circ})$ is relatively compact in $\FV$.
\end{prop}
\begin{proof}
Let $f\in \mathcal{FV}_{E'}(U,E)_{lb}$, $j\in J$ and $m\in M$. Then there are $i\in J$, $k\in M$ and $C>0$ such that 
for every $e'\in E'$
\begin{align*}
|\mathscr{R}_{f}(e')|_{j,m}&=|f_{e'}|_{j,m}
\leq C\sup_{\substack{x\in\omega_{k}\\(k,x)\in U}}|T^{\K}_{k}(f_{e'})(x)|\nu_{i,k}(x)\\
&=C \sup_{\substack{x\in\omega_{k}\\(k,x)\in U}}|(e'\circ f)(k,x)|\nu_{i,k}(x)
=C\sup_{y\in N_{U,i,k}(f)}|e'(y)|,
\end{align*}
which proves the first part because $N_{U,i,k}(f)$ is bounded in $E$. Let us consider the second part.
The bounded set $N_{U,i,k}(f)$ is already precompact in $E$ because $E$ is a generalised Schwartz space. 
Therefore we have $\mathscr{R}_{f}\in L(E_{\tau_{pc}}',\FV)$. 
The polar $B_{\alpha}^{\circ}$ is relatively compact 
in $E_{\tau_{pc}}'$ for every $\alpha\in\mathfrak{A}$ by the Alao\u{g}lu-Bourbaki theorem 
and thus $\mathscr{R}_{f}(B_{\alpha}^{\circ})$ in $\FV$ as well.
\end{proof}

\begin{thm}\label{thm:fix_topo_E_semi_M}
Let $E$ be a semi-Montel space, 
$(T^{E}_{m},T^{\K}_{m})_{m\in M}$ a strong, consistent generator for $(\mathcal{FV},E)$
and $U$ fix the topology in $\FV$. 
Then the restriction map $R_{U,E'}\colon S(\FV\varepsilon E)\to \mathcal{FV}_{E'}(U,E)_{lb}$ is surjective.
\end{thm}
\begin{proof}
Let $f\in \mathcal{FV}_{E'}(\Omega,E)_{lb}$ and $e'\in E'$. 
For every $f'\in \FV'$ there are $j\in J$, $m\in M$ and $C_{0}>0$ with
\[
|\mathscr{R}_{f}^{t}(f')(e')|=|f'(f_{e'})|\leq C_{0} |f_{e'}|_{j,m}.
\]
By the proof of \prettyref{prop:ext_E_semi_M} there are $i\in J$, $k\in M$ and $C>0$ such that
\[
|\mathscr{R}_{f}^{t}(f')(e')|\leq C_{0}C\sup_{y\in N_{U,i,k}(f)}|e'(y)|\leq C_{0}C\sup_{y\in \oacx(N_{U,i,k}(f))}|e'(y)|.
\]
The set $\oacx(N_{U,i,k}(f))$ is absolutely convex and compact by \cite[6.2.1 Proposition, p.\ 103]{Jarchow} 
and \cite[6.7.1 Proposition, p.\ 112]{Jarchow} because $E$ is semi-Montel. 
Therefore $\mathscr{R}_{f}^{t}(f')\in (E'_{\kappa})'=\mathcal{J}(E)$ by the Mackey-Arens theorem.
Like in \prettyref{thm:ext_F_semi_M} we obtain $\mathcal{J}^{-1}\circ\mathscr{R}_{f}^{t}\in \FV\varepsilon E$ 
by \eqref{eq1:ext_F_semi_M}, \eqref{eq2:ext_F_semi_M} and \prettyref{prop:ext_E_semi_M}. 
Setting $F:=S(\mathcal{J}^{-1}\circ\mathscr{R}_{f}^{t})$ 
we conclude $T^{E}_{m}(F)(x)=f(m,x)$ for all $(m,x)\in U$ by \eqref{eq3:ext_F_semi_M} and 
so $R_{U,E'}(F)=f$.
\end{proof}

We denote by $\mathcal{C}_{bu}(\Omega,E)$ the space
of bounded uniformly continuous functions from a metric space $\Omega$ to an lcHs $E$ which we endow
with the system of seminorms given by
\[
|f|_{\alpha}:=\sup_{x\in\Omega}p_{\alpha}(f(x)),\quad f\in\mathcal{C}_{bu}(\Omega,E),
\]
for $\alpha\in\mathfrak{A}$.

\begin{cor}\label{cor:uniformly_cont}
Let $\Omega$ be a metric space, $U\subset\Omega$ a dense subset and $E$ a semi-Montel space.
If $f\colon U\to E$ is a function such that $e'\circ f$ admits an extension $f_{e'}\in\mathcal{C}_{bu}(\Omega)$ for each $e'\in E'$, 
then there is a unique extension $F\in\mathcal{C}_{bu}(\Omega,E)$ of $f$. In particular, 
\[
\mathcal{C}_{bu}(\Omega,E)=\{f\colon\Omega\to E\;|\;\forall\;e'\in E':\;e'\circ f\in\mathcal{C}_{bu}(\Omega)\}.
\]
\end{cor}
\begin{proof}
$(\id_{E^{\Omega}},\id_{\K^{\Omega}})$ is a strong, consistent generator for $(\mathcal{C}_{bu},E)$ and 
$\mathcal{C}_{bu}(\Omega)\varepsilon E\cong \mathcal{C}_{bu}(\Omega,E)$ via $S$ by \cite[5.8 Example, p.\ 27]{kruse2017a}. 
Due to \prettyref{thm:fix_topo_E_semi_M}, \prettyref{prop:injectivity} and \prettyref{rem:rest_spaces_coincide} with $G=E'$
the extension $F$ exists and is unique because the dense set $U\subset\Omega$ fixes the topology in $\mathcal{C}_{bu}(\Omega)$. 
The rest follows from \prettyref{prop:weak_strong_principle}. 
\end{proof}

Let $\Omega\subset\C$ be open and bounded and $E$ an lcHs over $\C$. 
We denote by $\mathcal{A}(\overline{\Omega},E)$ the space of continuous functions 
from $\overline{\Omega}$ to an lcHs $E$ which are holomorphic on $\Omega$
and equip $\mathcal{A}(\overline{\Omega},E)$ with the system of seminorms given by
\[
|f|_{\alpha}:=\sup_{x\in\overline{\Omega}}p_{\alpha}(f(x)),\quad f\in\mathcal{A}(\overline{\Omega},E),
\]
for $\alpha\in\mathfrak{A}$.

\begin{cor}\label{cor:disc_algebra}
Let $\Omega\subset\C$ be open and bounded, $U\subset\overline{\Omega}$ fix the topology in $\mathcal{A}(\overline{\Omega})$ 
and $E$ a semi-Montel space over $\C$. If $f\colon U\to E$ is a function such that $e'\circ f$ admits an extension 
$f_{e'}\in\mathcal{A}(\overline{\Omega})$ for each $e'\in E'$, 
then there is a unique extension $F\in\mathcal{A}(\overline{\Omega},E)$ of $f$. In particular, 
\[
\mathcal{A}(\overline{\Omega},E)=\{f\colon\overline{\Omega}\to E\;|\;\forall\;e'\in E':\;e'\circ f\in\mathcal{A}(\overline{\Omega})\}.
\]
\end{cor}
\begin{proof}
$(\id_{E^{\overline{\Omega}}},\id_{\C^{\overline{\Omega}}})$ is a strong, consistent generator for $(\mathcal{A},E)$ and 
$\mathcal{A}(\overline{\Omega})\varepsilon E\cong \mathcal{A}(\overline{\Omega},E)$ via $S$ by \cite[3.1 Bemerkung, p.\ 141]{B2}. 
Due to \prettyref{thm:fix_topo_E_semi_M}, \prettyref{prop:injectivity} and \prettyref{rem:rest_spaces_coincide} with $G=E'$
the extension $F$ exists and is unique. The remaining part follows from \prettyref{prop:weak_strong_principle}. 
\end{proof}

If $\Omega\subset\C$ is connected, then the boundary $\partial\Omega$ of $\Omega$ fixes the topology in $\mathcal{A}(\overline{\Omega})$ 
by the maximum principle. If $\Omega=\D$, then $\partial\D$ is the intersection of all sets 
that fix the topology in $\mathcal{A}(\overline{\D})$ by \cite[7.7 Example, p.\ 39]{stout1971}. 

If $E$ is a generalised Schwartz space which is not a semi-Montel space, 
we do not know whether the extension results in \prettyref{cor:uniformly_cont} and \prettyref{cor:disc_algebra}
hold but we still have a weak-strong principle due to the following observation which is based on 
\cite[Chap.\ 3, \S9, Proposition 2, p.\ 231]{horvath} with $\sigma(E,E')$ replaced by $\sigma(E,G)$. 

\begin{prop}\label{prop:general_Schwartz}
If
\begin{enumerate}
\item[(i)] $E$ is a semi-Montel space and $G\subset E'$ a separating subspace, or if 
\item[(ii)] $E$ is a generalised Schwartz space and $G\subset \widehat{E}'$ a separating subspace, 
i.e.\ separates the points of the completion $\widehat{E}$,
\end{enumerate}
then the initial topology of $E$ and the topology $\sigma(E,G)$ coincide on the bounded sets of $E$.
\end{prop}
\begin{proof}
(i) Let $B\subset E$ be a bounded set. If $E$ is a semi-Montel space, then the closure $\overline{B}$ is compact 
in $E$. The topology induced by $\sigma(E,G)$ on $\overline{B}$ is Hausdorff and weaker than 
the initial topology induced by $E$. Thus the two topologies coincide on $\overline{B}$ and so on $B$ by 
the remarks above \cite[Chap.\ 3, \S9, Proposition 2, p.\ 231]{horvath}.

(ii) Let $B\subset E$ be a bounded set. If $E$ is a generalised Schwartz space, then $B$ is precompact in $E$ 
and relatively compact in the completion $\widehat{E}$ by \cite[3.5.1 Theorem, p.\ 64]{Jarchow}. 
Hence the closure $\overline{B}$ is compact in $\widehat{E}$. 
The topology induced by $\sigma(\widehat{E},G)$ on $\overline{B}$ is Hausdorff 
and weaker than the initial topology induced by $\widehat{E}$, implying
that the two topologies coincide on $\overline{B}$ as in part (i).
This yields that $\sigma(E,G)$ and the initial topology of $E$ coincide on $B$ because 
$\sigma(E,G)=\sigma(\widehat{E},G)$ on $B$ and the initial topologies of $E$ and $\widehat{E}$ 
coincide on $B$ as well.
\end{proof}

We apply this observation to the space $\mathcal{A}(\overline{\Omega},E)$.

\begin{rem}
Let $E$ be an lcHs over $\C$ and $\Omega\subset\C$ open and bounded. If
\begin{enumerate}
\item[(i)] $E$ is a semi-Montel space and $G\subset E'$ determines boundedness, or if 
\item[(ii)] $E$ is a generalised Schwartz space and $G\subset\widehat{E}'$ a separating subspace which 
determines boundedness in $E$, 
\end{enumerate}
then 
\[
 \mathcal{A}(\overline{\Omega},E)
=\{f\colon\overline{\Omega}\to E\;|\;\forall\;e'\in G:\;e'\circ f\in\mathcal{A}(\overline{\Omega})\}.
\]
Indeed, let us denote the right-hand side by $\mathcal{A}(\overline{\Omega},E)_{\sigma}$ 
and set $E_{\sigma}:=(E,\sigma(E,G))$. 
Then $\mathcal{A}(\overline{\Omega},E)_{\sigma}=\mathcal{A}(\overline{\Omega},E_{\sigma})$ 
and $f(\overline{\Omega})$ is bounded for every $f\in\mathcal{A}(\overline{\Omega},E)_{\sigma}$ 
as $G$ determines boundedness in $E$. 
The initial topology of $E$ and $\sigma(E,G)$ coincide on the bounded range $f(\overline{\Omega})$ 
of $f\in\mathcal{A}(\overline{\Omega},E)_{\sigma}$ by \prettyref{prop:general_Schwartz}. Hence we deduce that
\[
\mathcal{A}(\overline{\Omega},E)_{\sigma}=\mathcal{A}(\overline{\Omega},E_{\sigma})=\mathcal{A}(\overline{\Omega},E).
\]
\end{rem}

In this way Bierstedt proves his weak-strong principles for weighted continuous functions in \cite[2.10 Lemma, p.\ 140]{B2} with $G=E'=\widehat{E}'$.

\section*{\texorpdfstring{$\FV$}{FV(Omega)} a Fr\'{e}chet-Schwartz space and \texorpdfstring{$E$}{E} locally complete}

\begin{defn}[{\cite[Definition 12, p.\ 8]{B/F/J}}]\label{def:fix_top_2}
 Let $Y$ be a Fr\'{e}chet space. An increasing sequence $(B_{n})_{n\in\N}$ of bounded subsets of $Y_{b}'$ \emph{fixes 
 the topology} in $Y$ if $(B_{n}^{\circ})_{n\in\N}$ is a fundamental system of zero neighbourhoods of $Y$.
\end{defn}

\begin{rem}\label{rem:ex_almost_norming}
Let $Y$ be a Banach space. If $B\subset Y_{b}'$ is bounded, i.e.\ bounded w.r.t.\ the operator norm, 
such that $B$ fixes the topology in $Y$, i.e.\ $B^{\circ}$ is bounded in $Y$, then $B$ is called an almost norming subset.
Examples of almost norming subspaces are given in \cite[Remark 1.2, p.\ 780-781]{Arendt2000}. 
For instance, the set of point evaluations $B:=\{\delta_{1/n}\;|\; n\in\N\}$ is almost norming for
the Hardy space $Y:=H^{\infty}(\D):=\mathcal{C}^{\infty}_{\overline{\partial},b}(\D)$.
\end{rem}

\begin{defn}[{chain-structured}]\label{def:chain_structured} 
Let $\FV$ be a $\dom$-space. We say that $U\subset\bigcup_{m\in \N}\{m\}\times\omega_{m}$ is \emph{chain-structured} if 
\begin{enumerate}
\item[(i)] $(k,x)\in U\;\;\Rightarrow\;\;\forall\;m\in \N,\,m\geq k:\;(m,x)\in U$,
\item[(ii)] $\forall\;(k,x)\in U,\,m\geq k,\,f\in\FV:\;T^{\K}_{k}(f)(x)=T^{\K}_{m}(f)(x)$.
\end{enumerate}
\end{defn}

\begin{rem} 
Let $\Omega\subset\R^{d}$ be open and $\mathcal{V}^{\infty}$ be a directed family of weights. 
Concerning the operators $(T^{\K}_{m})_{m\in\N_{0}}$ 
of $\mathcal{CV}^{\infty}(\Omega)$ from \prettyref{ex:standard_example} 
where $\omega_{m}=\{\beta\in\N_{0}^{d}\;|\;|\beta|\leq m\}\times\Omega$ resp.\ 
$\omega_{m}=\N_{0}^{d}\times\Omega$, we have for all $k\in\N_{0}$ and $f\in\mathcal{CV}^{\infty}(\Omega)$ that
\[
T^{\K}_{k}(f)(\beta,x)=\partial^{\beta}f(x)=T^{\K}_{m}(f)(\beta,x),\quad \beta\in\N_{0}^{d},\;|\beta|\leq k,\;x\in\Omega,
\]
for all $m\in\N_{0}$, $m\geq k$. Hence condition (ii) of \prettyref{def:chain_structured} is fulfilled for any 
$U\subset\bigcup_{m\in \N_{0}}\{m\}\times\omega_{m}$ in this case. 
Condition (i) says that once a {\glq link\grq} $(k,\beta,x)$ belongs to $U$ for some 
order $k$, then the {\glq link\grq} $(m,\beta,x)$ belongs to $U$ for any higher order $m$ as well. 
\end{rem}

\begin{defn}[{diagonally dominated, increasing}] 
We say that a family $\mathcal{V}:=(\nu_{j,m})_{j,m\in\N}$ of weights on $\Omega$ 
is \emph{diagonally dominated} and \emph{increasing} if $\omega_{m}\subset \omega_{m+1}$ for all $m\in \N$ and 
$\nu_{j,m}\leq \nu_{\max(j,m),\max(j,m)}$ on $\omega_{\min(j,m)}$ for all $j,m\in\N$ as well as 
$\nu_{j,j}\leq \nu_{j+1,j+1}$ on $\omega_{j}$ for all $j\in\N$. 
\end{defn}

\begin{rem}\label{rem:fix_top_1=2}
Let $\FV$ be a $\dom$-space, $U\subset\bigcup_{m\in \N}\{m\}\times\omega_{m}$ chain-structured, $G\subset E'$ a separating subspace
and $\mathcal{V}$ diagonally dominated and increasing.  
\begin{enumerate}
 \item [a)] If $U$ fixes the topology in $\FV$, then 
 \[
 \mathcal{FV}_{G}(U,E)_{lb}=\{f\in\mathcal{FV}_{G}(U,E)\;|\;\forall\;i\in \N:\;N_{U,i}(f)\;\text{bounded in}\; E\}
 \]
 with $N_{U,i}(f):=N_{U,i,i}(f)$.
 \item [b)] Let $\FV$ be a Fr\'{e}chet space. We set $U_{m}:= \{(m,x)\in U\;|\; x\in \omega_{m}\}$ and 
 $B_{j}:=\bigcup_{m=1}^{j}\{T^{\K}_{m,x}(\cdot)\nu_{m,m}(x)\;|\; (m,x)\in U_{m}\}\subset \FV'$ for $j\in\N$. 
 Then $U$ fixes the topology in $\FV$ in the sense of \prettyref{def:fix_top_1} if and only if 
 the sequence $(B_{j})_{j\in\N}$ fixes the topology in $\FV$ in the sense of \prettyref{def:fix_top_2}. 
\end{enumerate}
\end{rem}
\begin{proof}
Let us begin with a). We only need to show that the inclusion '$\supset$` holds. Let $f$ be an element of the right-hand side and 
$i,k\in\N$. We set $m:=\max(i,k)$ and observe that for $(k,x)\in U$ we have $(m,x)\in U$ by (i) and
\[
(e'\circ f)(k,x)=T^{\K}_{k}(f_{e'})(x)\underset{\text{(ii)}}{=}T^{\K}_{m}(f_{e'})(x)=(e'\circ f)(m,x)
\]
for each $e'\in G$ with (i) and (ii) from the definition of $U$ being chain-structured. Since $G$ is separating, it follows that 
$f(k,x)=f(m,x)$. Hence we get for all $\alpha\in\mathfrak{A}$
\begin{align*}
\sup_{y\in N_{U,i,k}(f)}p_{\alpha}(y)&=\sup_{\substack{x\in\omega_{k}\\(k,x)\in U}}p_{\alpha}(f(k,x))\nu_{i,k}(x)
\underset{\text{(i)}}{\leq} \sup_{\substack{x\in\omega_{m}\\(m,x)\in U}}p_{\alpha}(f(k,x))\nu_{m,m}(x)\\
&= \sup_{\substack{x\in\omega_{m}\\(m,x)\in U}}p_{\alpha}(f(m,x))\nu_{m,m}(x)<\infty
\end{align*}
using that $\omega_{k}\subset\omega_{m}$ and $\mathcal{V}$ is diagonally dominated.

Let us turn to part b). '$\Rightarrow$`: Let $j\in\N$ and $A\subset \FV$ be bounded. Then 
\[
\sup_{y\in B_{j}}\sup_{f\in A}|y(f)|=\sup_{\substack{1\leq m\leq j\\(m,x)\in U_{m}}}\sup_{f\in A}|T^{\K}_{m}(f)(x)|\nu_{m,m}(x)
\leq\sup_{f\in A}\sup_{1\leq m\leq j}|f|_{m,m}<\infty
\]
since $A$ is bounded, implying that $B_{j}$ is bounded in $\FV_{b}'$. Further, $(B_{j})$ is increasing by definition. 
Additionally, for all $j\in\N$
\begin{align*}
B_{j}^{\circ}&=\bigcap_{m=1}^{j}\{f\in\FV\;|\;\sup_{\substack{x\in\omega_{m}\\(m,x)\in U}}|T^{\K}_{m}(f)(x)|\nu_{m,m}(x)\leq 1\}\\
&=\{f\in\FV\;|\;\sup_{\substack{x\in\omega_{j}\\(j,x)\in U}}|T^{\K}_{j}(f)(x)|\nu_{j,j}(x)\leq 1\}
\end{align*}
because $U$ is chain-structured and $\mathcal{V}$ increasing. 
Thus $(B_{j}^{\circ})$ is a fundamental system of zero neighbourhoods of $\FV$ if $U$ fixes the topology.\\
'$\Leftarrow$`: Let $j,m\in\N$. Then there are $i\in\N$ and $\varepsilon>0$ such that 
\[
\varepsilon B_{i}^{\circ}\subset \{f\in\FV\;|\;|f|_{j,m}\leq 1\}=:D_{j,m},
\]
which follows from fixing the topology in the sense of \prettyref{def:fix_top_2}.
Let $f\in D_{j,m}$ and set
\[
|f|_{U_{i}}:=\sup_{(i,x)\in U_{i}}|T^{\K}_{i}(f)(x)|\nu_{i,i}(x).
\]
If $|f|_{U_{i}}= 0$, then $tf\in\varepsilon B_{i}^{\circ}$ for all $t>0$ and 
hence $t|f|_{j,m}=|tf|_{j,m}\leq 1$ for all $t>0$, which yields 
$|f|_{j,m}=0=|f|_{U_{i}}$.
If $|f|_{U_{i}}\neq 0$, then $\tfrac{f}{|f|_{U_{i}}}\in B_{i}^{\circ}$ and thus $\varepsilon\tfrac{f}{|f|_{U_{i}}}\in D_{j,m}$, implying
\[
|f|_{j,m}=\frac{1}{\varepsilon}|f|_{U_{i}}\bigl|\varepsilon\frac{f}{|f|_{U_{i}}}\bigr|_{j,m}\leq \frac{1}{\varepsilon}|f|_{U_{i}}.
\]
The inequality $|f|_{j,m}\leq\tfrac{1}{\varepsilon}|f|_{U_{i}}$ still holds if $|f|_{U_{i}}= 0$.
\end{proof}

\begin{thm}[{\cite[Theorem 16, p.\ 236]{B/F/J}}]\label{thm:fix_top}
 Let $Y$ be a Fr\'{e}chet-Schwartz space, $(B_{j})_{j\in\N}$ fix the topology in $Y$ and 
 $\mathsf{A}\colon X:=\operatorname{span}(\bigcup_{j\in\N} B_{j})\to E$ be a linear map which is bounded on each $B_{j}$. If
 \begin{enumerate}
  \item [a)] $(\mathsf{A}^{t})^{-1}(Y)$ is dense in $E_{b}'$ and $E$ locally complete, or if 
  \item [b)] $(\mathsf{A}^{t})^{-1}(Y)$ is dense in $E_{\sigma}'$ and $E$ is $B_{r}$-complete,
 \end{enumerate}
 then $\mathsf{A}$ has a (unique) extension $\widehat{\mathsf{A}}\in Y\varepsilon E$.
\end{thm}

Now, we generalise \cite[Theorem 17, p.\ 237]{B/F/J}.

\begin{thm}\label{thm:ext_FS_fix_top}
 Let $E$ be an lcHs, $G\subset E'$ a separating subspace, 
 $(T^{E}_{m},T^{\K}_{m})_{m\in M}$ be a strong, consistent generator for $(\mathcal{FV},E)$, 
 $\FV$ a Fr\'{e}chet-Schwartz space, $\mathcal{V}$ diagonally dominated and increasing and
 $U$ be chain-structured and fix the topology in $\FV$.
 If
 \begin{enumerate}
  \item [a)] $G$ is dense in $E_{b}'$ and $E$ locally complete, or if 
  \item [b)] $E$ is $B_{r}$-complete,
 \end{enumerate}
 then the restriction map $R_{U,G}\colon S(\FV\varepsilon E)\to \mathcal{FV}_{G}(U,E)_{lb}$ is surjective.
\end{thm}
\begin{proof}
Let $f\in \mathcal{FV}_{G}(U,E)_{lb}$. We set $X:=\operatorname{span}(\bigcup_{j\in\N} B_{j})$ 
with $B_{j}$ from \prettyref{rem:fix_top_1=2} b) and $Y:=\FV$. 
Let $\mathsf{A}\colon X\to E$ be the linear map determined by 
\[
\mathsf{A}(T^{\K}_{m,x}(\cdot)\nu_{m,m}(x)):=f(m,x)\nu_{m,m}(x), 
\]
for $1\leq m\leq j$ and $(m,x)\in U_{m}$ with $U_{m}$ from \prettyref{rem:fix_top_1=2} b).
The map $\mathsf{A}$ is well-defined since $G$ is $\sigma(E',E)$-dense, and 
bounded on each $B_{j}$ because $\mathsf{A}(B_{j})=\bigcup_{m=1}^{j}N_{U,m}(f)$.
Let $e'\in G$ and $f_{e'}$ be the unique element in $\FV$ such that 
$T^{\K}_{m}(f_{e'})(x)=(e'\circ f)(m,x)$ for all $(m,x)\in U$, which implies 
$T^{\K}_{m}(f_{e'})(x)\nu_{m,m}(x)=(e'\circ \mathsf{A})(T^{\K}_{m,x}(\cdot)\nu_{m,m}(x))$ for all $(m,x)\in U_{m}$.
This equation allows us to consider $f_{e'}$ as a linear form on $X$ 
(by $f_{e'}(T^{\K}_{m,x}(\cdot)\nu_{m,m}(x)):=(e'\circ \mathsf{A})(T^{\K}_{m,x}(\cdot)\nu_{m,m}(x))$), 
which yields $e'\circ \mathsf{A}\in\FV$ for all $e'\in G$. It follows that
$G\subset(\mathsf{A}^{t})^{-1}(Y)$. Noting that $G$ is $\sigma(E',E)$-dense, we 
apply \prettyref{thm:fix_top} and obtain an extension $\widehat{\mathsf{A}}\in\FV\varepsilon E$ of 
$\mathsf{A}$. We set $F:=S(\widehat{\mathsf{A}})$
and observe that for all $(m,x)\in U$ there is $j\in\N$, $j\geq m$, such that $(j,x)\in U_{j}$ and 
$\nu_{j,j}(x)>0$ by \eqref{eq:loc3} and because $U$ is chain-structured and $\mathcal{V}$ diagonally dominated and increasing.
Due to the proof of \prettyref{rem:fix_top_1=2} a) we have $f(j,x)=f(m,x)$ and thus
\begin{align*}
T^{E}_{m}(F)(x)&=T^{E}_{m}S(\widehat{\mathsf{A}})(x)=\widehat{\mathsf{A}}(T^{\K}_{m,x})
=\frac{1}{\nu_{j,j}(x)}\widehat{\mathsf{A}}(T^{\K}_{m,x}(\cdot)\nu_{j,j}(x))\\
&=\frac{1}{\nu_{j,j}(x)}\widehat{\mathsf{A}}(T^{\K}_{j,x}(\cdot)\nu_{j,j}(x))
=f(j,x)=f(m,x)
\end{align*}
by consistency, implying $R_{U,G}(F)=f$.
\end{proof}

In particular, condition a) is fulfilled if $E$ is semireflexive. Indeed, if $E$ is semireflexive, then $E$ is quasi-complete 
by \cite[Chap.\ IV, 5.5, Corollary 1, p.\ 144]{schaefer}
and $\overline{G}^{b(E',E)}=\overline{G}^{\mu(E',E)}=E'$ by \cite[11.4.1 Proposition, p.\ 227]{Jarchow} and the bipolar theorem. 
For instance, condition b) is satisfied if $E$ is a Fr\'echet space or $E=(\mathcal{C}^{\infty}_{P(\partial),b}(\Omega),\beta)$ 
with a hypoelliptic linear partial differential operator $P(\partial)^{\K}$ and open $\Omega\subset\R^{d}$,
which is a $B_{r}$-complete space by \prettyref{prop:strict_top_B_r} 
and may not be a Fr\'echet space by \prettyref{rem:strict_top_B_r}.

As stated, our preceding theorem generalises \cite[Theorem 17, p.\ 237]{B/F/J} 
where $\FV$ is a closed subspace of $\mathcal{CW}^{\infty}(\Omega)$ 
for open, connected $\Omega\subset\R^{d}$. A characterisation of sets 
that fix the topology in the space $\mathcal{CW}^{\infty}_{\overline{\partial}}(\Omega)$ 
of holomorphic functions on an open, connected set $\Omega\subset\C$ is given in \cite[Remark 14, p.\ 235]{B/F/J}. 
The characterisation given in \cite[Remark 14 (b), p.\ 235]{B/F/J} is still valid and applied in \cite[Corollary 18, p.\ 238]{B/F/J} 
for closed subspaces of $\mathcal{CW}^{\infty}_{P(\partial)}(\Omega)$ 
where $P(\partial)^{\K}$ is a hypoelliptic linear partial differential operator which satisfies the maximum principle, 
namely, that $U\subset\Omega$ fixes the topology 
if and only if there is a sequence $(\Omega_{n})_{n\in\N}$ of relatively compact, open subsets of $\Omega$ 
with $\bigcup_{n\in\N}\Omega_{n}=\Omega$ such that $\partial\Omega_{n}\subset\overline{U\cap\Omega_{n+1}}$ for all $n\in\N$. 
Among the hypoelliptic operators $P(\partial)^{\K}$ satisfying the maximum principle 
are the Cauchy-Riemann operator $\overline{\partial}$ 
and the Laplacian $\Delta$. Further examples can be found in \cite[Corollary 3.2, p.\ 33]{gilbarg_trudinger2001}.
The statement of \cite[Corollary 18, p.\ 238]{B/F/J} for the space of holomorphic functions is itself a generalisation of 
\cite[Theorem 2, p.\ 401]{grosse-erdmann2004} with \cite[Remark 2 (a), p.\ 406]{grosse-erdmann2004} where $E$ is $B_{r}$-complete
and of \cite[Theorem 6, p.\ 10]{jorda2005} where $E$ is semireflexive. 
The case that $G$ is dense in $E_{b}'$ and $E$ is sequentially complete 
is covered by \cite[3.3 Satz, p.\ 228-229]{Gramsch1977}, not only for spaces of holomorphic functions, 
but for several classes of function spaces. 

Let us turn to other families of weights than $\mathcal{W}^{\infty}$. 
Due to \prettyref{prop:AV_CV_coincide_hol_har} we already know that $U:=\{0\}\times\C$ fixes the topology in 
$\mathcal{CV}^{\infty,\ast}_{\overline{\partial}}(\C)=\mathcal{AV}^{\infty}_{\overline{\partial}}(\C)$ and $U:=\{0\}\times\R^{d}$ in 
$\mathcal{CV}^{\infty,\ast}_{\Delta}(\R^{d})=\mathcal{AV}^{\infty}_{\Delta}(\R^{d})$ 
if $\mathcal{V}^{\infty}$ fulfils \prettyref{cond:weights}.
Next, we concentrate on the first case since smaller sets that fix the topology are known.

\begin{cor}
Let $E$ be an lcHs over $\C$, $G\subset E'$ a separating subspace, $\mathcal{V}^{\infty}$ fulfil \prettyref{cond:weights} and 
$U\subset\C$ fix the topology of $\mathcal{AV}^{\infty}_{\overline{\partial}}(\C)$. If 
 \begin{enumerate}
  \item [a)] $G$ is dense in $E_{b}'$ and $E$ locally complete, or if 
  \item [b)] $E$ is $B_{r}$-complete,
 \end{enumerate}
and $f\colon U\to E$ is a function in $\ell\mathcal{V}^{\infty}(U)$ such that $e'\circ f$ admits an extension 
$f_{e'}\in\mathcal{AV}^{\infty}_{\overline{\partial}}(\C)$ for each $e'\in G$, 
then there is a unique extension $F\in\mathcal{AV}^{\infty}_{\overline{\partial}}(\C,E)$ of $f$.
\end{cor}
\begin{proof}
The existence of $F$ follows from the proof of \prettyref{prop:AV_CV_coincide_hol_har} and \prettyref{thm:ext_FS_fix_top} 
with $(T^{E}_{m},T^{\C}_{m})_{m\in M}:=(\id_{E^{\C}},\id_{\C^{\C}})$. The uniqueness of $F$ is a result of \prettyref{prop:injectivity}.
\end{proof}

We have the following sufficient conditions on a family of weights $\mathcal{V}^{\infty}$ 
which guarantee the existence of a countable set 
$U\subset\C$ that fixes the topology of $\mathcal{AV}^{\infty}_{\overline{\partial}}(\C)$ due to Abanin and Varziev.

\begin{prop}\label{prop:AV_fix_top_Abanin}
Let $\mathcal{V}^{\infty}:=(\nu_{j})_{j\in\N}$ where $\nu_{j}(z):=\exp(a_{j}\mu(z)-\varphi(z))$, $z\in\C$, 
with some continuous, subharmonic function 
$\mu\colon\C\to[0,\infty)$, a continuous function $\varphi\colon\C\to\R$ and a strictly increasing, positive sequence $(a_{j})_{j\in\N}$ 
with $a:=\lim_{j\to\infty} a_{j}\in (0,\infty]$. Let there be
\begin{enumerate}
\item[(i)] $s\geq 0$ and $C>0$ such that $|\varphi(z)-\varphi(\zeta)|\leq C$ and $|\mu(z)-\mu(\zeta)|\leq C$ for all $z,\zeta\in\C$ with 
$|z-\zeta|\leq (1+|z|)^{-s}$,
\item[(ii)] $\max(\varphi(z),\mu(z))\leq |z|^{q}+C_{0}$ for some $q,C_{0}>0$ and
\item[(iii)] $\ln(|z|)=O(\mu(z))$ as $|z|\to\infty$ if $a=\infty$, or $\ln(|z|)=o(\mu(z))$ as $|z|\to\infty$ if $0<a<\infty$. 
\end{enumerate}
Let $(\lambda_{k})_{k\in\N}$ be the sequence of simple zeros of a function 
$L\in\mathcal{AV}^{\infty,1}_{\overline{\partial}}(\C)$ having no other zeros 
where $\mathcal{V}^{\infty,1}:=(\nu_{j}^{2}/\nu_{m_{j}})_{j\in\N}$ for some sequence $(m_{j})_{j\in\N}$ in $\N$. 
Suppose that there are $j_{0}\in\N$ and 
a sequence of circles $\{z\in\C\;|\;|z|=R_{m}\}$ with $R_{m}\nearrow\infty$ such that 
\[
|L(z)|\nu_{j_{0}}(z)\geq C_{m},\quad m\in\N,\,z\in\C,\,|z|=R_{m},
\]
for some $C_{m}\nearrow\infty$ and 
\[
|L'(\lambda_{k})|\nu_{j_{0}}(\lambda_{k})\geq 1\quad\text{for all sufficiently large}\;k\in\N.
\]
Then $\mathcal{V}^{\infty}$ fulfils \prettyref{cond:weights} for all $a\in(0,\infty]$ and $U:=(\lambda_{k})_{k\in\N}$ fixes the topology 
of $\mathcal{AV}^{\infty}_{\overline{\partial}}(\C)$ if $a=\infty$. 
If $\mu$ is a radial function, i.e.\ $\mu(z)=\mu(|z|)$, $z\in\C$, with $\mu(2z)\sim\mu(z)$ as $|z|\to\infty$, then 
$U$ fixes the topology of $\mathcal{AV}^{\infty}_{\overline{\partial}}(\C)$ for all $a\in(0,\infty]$.
\end{prop}
\begin{proof}
First, we check that \prettyref{cond:weights} is satisfied. We set $k:=\max(s,2)$ 
and observe that (i) is also fulfilled with $k$ instead of $s$. 
Let $z\in\C$ and $\|\zeta\|_{\infty},\|\eta\|_{\infty}\leq (1/\sqrt{2})(1+|z|)^{-k}=:r(z)$. 
From $|\cdot|\leq \sqrt{2}\|\cdot\|_{\infty}$ and (i) it follows
\[
|\mu(z+\zeta)-\mu(z+\eta)|\leq |\mu(z+\zeta)-\mu(z)|+|\mu(z)-\mu(z+\eta)|\leq C
\] 
and thus $\mu(z+\zeta)\leq C+\mu(z+\eta)$. In the same way we obtain $-\varphi(z+\zeta)\leq C-\varphi(z+\eta)$. 
Hence we have
\[
a_{j}\mu(z+\zeta)-\varphi(z+\zeta)\leq C(a_{j}+1)+a_{j}\mu(z+\eta)-\varphi(z+\eta)
\]
for $j\in\N$, implying	
\[
\nu_{j}(z+\zeta)\leq e^{C(a_{j}+1)}\nu_{j}(z+\eta),
\]
which means that $(\alpha.1)$ holds. By (iii) there are $\varepsilon>0$ and $R>0$ such that $\ln(|z|)\leq \varepsilon \mu(z)$ for all 
$z\in\C$ with $|z|\geq R$ if $a=\infty$. This yields for all $|z|\geq \max(2,R)$ that
\[
a_{j}\mu(z)+k\ln(1+|z|)\leq a_{j}\mu(z)+k\ln(|z|^2) = a_{j}\mu(z)+2k\ln(|z|)\leq a_{j}\mu(z)+2k\varepsilon\mu(z).
\] 
Since $a=\infty$, there is $n\in\N$ such that $a_{n}\geq a_{j}+2k\varepsilon$, resulting in 
\[
a_{j}\mu(z)+k\ln(1+|z|)\leq a_{n}\mu(z)
\] 
for all $|z|\geq \max(2,R)$. Therefore we derive 
\begin{equation}\label{eq:AV_fix_top}
a_{j}\mu(z)+k\ln(1+|z|)\leq a_{n}\mu(z)+k\ln(1+\max(2,R))
\end{equation}
for all $z\in\C$, which means that $(\alpha.2)$ and $(\alpha.3)$ hold with $\psi_{j}(z):=r(z)$. 
If $0<a<\infty$, for every $\varepsilon>0$ there is $R>0$ such that $\ln(|z|)\leq \varepsilon \mu(z)$ for all 
$z\in\C$ with $|z|\geq R$ by (iii). Thus we may choose $\varepsilon>0$ such that $a_{j+1}-a_{j}\geq 2k\varepsilon>0$ 
because $(a_{j})$ is strictly increasing. We deduce that \eqref{eq:AV_fix_top} with $n:=j+1$ holds in this case as well and 
$(\alpha.2)$ and $(\alpha.3)$, too.

Observing that the condition that $U=(\lambda_{k})_{k\in\N}$ is the sequence of simple zeros of a function 
$L\in\mathcal{AV}^{\infty,1}_{\overline{\partial}}(\C)$ means that $L\in\mathscr{L}(\Phi^{a}_{\varphi,\mu};U)$ 
and (i) that $\varphi$ and $\mu$ vary slowly w.r.t.\ $r(z):=(1+|z|)^{-s}$ in the notation of 
\cite[Definition, p.\ 579, 584]{abanin2013} and \cite[p.\ 585]{abanin2013}, respectively, 
the statement that $U$ fixes the topology is a consequence of \cite[Theorem 2, p.\ 585-586]{abanin2013}.
\end{proof}

\begin{rem}\label{rem:ex_fix_top_AP}
\begin{enumerate}
\item[a)] Let $D\subset\C$ be convex, bounded and open with $0\in D$. 
Let $\varphi(z):=H_{D}(z):=\sup_{\zeta\in D}\re(z\zeta)$, $z\in\C$, be the supporting 
function of $D$, $\mu(z):=\ln(1+|z|)$, $z\in\C$, and $a_{j}:=j$, $j\in\N$. 
Then $\varphi$ and $\mu$ fulfil the conditions of \prettyref{prop:AV_fix_top_Abanin} with $a=\infty$ by \cite[p.\ 586]{abanin2013} 
and the existence of an entire function $L$ which fulfils the conditions of \prettyref{prop:AV_fix_top_Abanin} is guaranteed by 
\cite[Theorem 1.6, p.\ 1537]{abanin2011}. Thus there is a countable set $U:=(\lambda_{k})_{k\in\N}\subset\C$ which fixes the topology 
in $A^{-\infty}_{D}:=\mathcal{AV}^{\infty}_{\overline{\partial}}(\C)$ with $\mathcal{V}^{\infty}:=(\exp(a_{j}\mu-\varphi))_{j\in\N}$.
\item[b)] An explicit construction of a set $U:=(\lambda_{k})_{k\in\N}\subset\C$ which fixes the topology in $A^{-\infty}_{D}$ is given 
in \cite[Algorithm 3.2, p.\ 3629]{abanin2010}. This construction does not rely on the entire function $L$. 
In particular (see \cite[p.\ 15]{bonet2017}), for $D:=\D$ we have $\varphi(z)=|z|$, 
for each $k\in\N$ we may take $l_{k}\in\N$, $l_{k}>2\pi k^2$, 
and set $\lambda_{k,j}:=kr_{k,j}$, $1\leq j\leq l_{k}$, where $r_{k,j}$ denote the $l_{k}$-roots of unity. 
Ordering $\lambda_{k,j}$ in a sequence of one index appropriately, we obtain a sequence which fixes the topology of $A^{-\infty}_{\D}$.
\item[c)] Let $\mu\colon\C\to[0,\infty)$ be a continuous, subharmonic, radial function which increases with $|z|$ and satisfies 
\begin{enumerate}
\item[(i)] $\sup_{\zeta\in\C,\|\zeta\|_{\infty}\leq r(z)}\mu(z+\zeta)\leq \inf_{\zeta\in\C,\|\zeta\|_{\infty}\leq r(z)}\mu(z+\zeta)+C$ 
for some continuous function $r\colon\C\to (0,1]$ and $C>0$,
\item[(ii)] $\ln(1+|z|^2)=o(\mu(|z|))$ as $|z|\to\infty$,
\item[(iii)] $\mu(2|z|)=O(\mu(|z|))$ as $|z|\to\infty$. 
\end{enumerate}
Then $\mathcal{V}^{\infty}:=(\exp(-(1/j)\mu))_{j\in\N}$ fulfils \prettyref{cond:weights} where $(\alpha.1)$ follows from (i) and 
$(\alpha.2)$, $(\alpha.3)$ as in the proof of \prettyref{prop:AV_fix_top_Abanin}. 
If $\mu(|z|)=o(|z|^2)$ as $|z|\to\infty$ or $\mu(|z|)=|z|^2$, $z\in\C$, then $U:=\{\alpha n+i\beta m\;|\;n,m\in\Z\}$ fixes 
the topology in $A_{\mu}^{0}:=\mathcal{AV}^{\infty}_{\overline{\partial}}(\C)$ 
for any $\alpha,\beta>0$ by \cite[Corollary 4.6, p.\ 20]{bonet2017} 
and \cite[Proposition 4.7, p.\ 20]{bonet2017}, respectively. 
\item[d)] For instance, the conditions on $\mu$ in c) are  fulfilled for $\mu(z):=|z|^{\gamma}$, $z\in\C$, with $0<\gamma\leq 2$ 
by \cite[1.5 Examples (3), p.\ 205]{meise1987}. If $\gamma=1$, then $A_{\mu}^{0}=A^{0}_{\overline{\partial}}(\C)$ is the space 
of entire functions of exponential type zero.
\item[e)] More general characterisations of countable sets that fix the topology of $\mathcal{AV}^{\infty}_{\overline{\partial}}(\C)$ 
are given in \cite[Theorem 1, p.\ 580]{abanin2013} and \cite[Theorem 4.5, p.\ 17]{bonet2017}.
\end{enumerate}
\end{rem}

The spaces $A_{\mu}^{0}$ from c) are known as H\"ormander algebras and the space $A^{-\infty}_{D}(\C)$ considered in a) is isomorphic 
to the strong dual of the Korenblum space $A^{-\infty}(D)$ via Laplace transform by \cite[Proposition 4, p.\ 580]{melikhov2004}.
\section{Extension of sequentially bounded functions}
In this section we restrict to the case that $E$ is a Fr\'echet space and $G\subset E'$ is generated by a sequence that 
fixes the topology in $E$.

\begin{defn}[{$sb$-restriction space}]
Let $E$ be a Fr\'{e}chet space, $(B_{n})$ fix the topology in $E$ and $G:=\operatorname{span}(\bigcup_{n\in\N} B_{n})$. 
Let $\FV$ be a $\dom$-space, $U$ a set of uniqueness for $(T^{\K}_{m},\mathcal{FV})_{m\in M}$ and
set
\[
\mathcal{FV}_{G}(U,E)_{sb}:=\{f\in\mathcal{FV}_{G}(U,E)\;|\;\forall\;n\in\N:\; \{f_{e'}\;|\;e'\in B_{n}\}\;
\text{is bounded in}\;\FV\}.
\]
\end{defn}

Let $E$ be a Fr\'{e}chet space, $(B_{n})$ fix the topology in $E$, $G:=\operatorname{span}(\bigcup_{n\in\N} B_{n})$,
$(T^{E}_{m},T^{\K}_{m})_{m\in M}$ be a strong, consistent generator for $(\mathcal{FV},E)$ 
and $U$ a set of uniqueness for $(T^{\K}_{m},\mathcal{FV})_{m\in M}$. 
For $u\in \FV\varepsilon E$ we have $R_{U,G}(f)\in\mathcal{FV}_{G}(U,E)$ with $f:=S(u)$
by \prettyref{rem:R_well-defined} and for $j\in J$ and $m\in M$ 
\[
\sup_{e'\in B_{n}}|f_{e'}|_{j,m}=\sup_{e'\in B_{n}}\sup_{x\in\omega_{m}}|e'(T^{E}_{m}(f)(x)\nu_{j,m}(x))|
=\sup_{e'\in B_{n}}\sup_{y\in N_{j,m}(f)}|e'(y)|
\]
with $N_{j,m}(f):=\{T^{E}_{m}(f)(x)\nu_{j,m}(x)\;|\;x\in\omega_{m}\}$. This set is bounded in $E$ since 
\[
\sup_{y\in N_{j,m}(f)}p_{\alpha}(f)=|f|_{j,m,\alpha}<\infty
\]
for all $\alpha\in\mathfrak{A}$, implying $\sup_{e'\in B_{n}}|f_{e'}|_{j,m}<\infty$ and 
$R_{U,G}(f)\in\mathcal{FV}_{G}(U,E)_{sb}$. 
Hence the injective linear map
\[
R_{U,G}\colon S(\FV\varepsilon E)\to \mathcal{FV}_{G}(U,E)_{sb},\;f\mapsto (T^{E}_{m}(f)(x))_{(m,x)\in U},
\]
is well-defined.

\begin{que}\label{que:surjective_sb_rest_space}
Let $E$ be a Fr\'{e}chet space, $(B_{n})$ fix the topology in $E$ and $G:=\operatorname{span}(\bigcup_{n\in\N} B_{n})$. 
Let $(T^{E}_{m},T^{\K}_{m})_{m\in M}$ be a strong, consistent generator for $(\mathcal{FV},E)$
and $U$ a set of uniqueness for $(T^{\K}_{m},\mathcal{FV})_{m\in M}$.
When is the injective restriction map 
\[
R_{U,G}\colon S(\FV\varepsilon E)\to \mathcal{FV}_{G}(U,E)_{sb},\;f\mapsto (T^{E}_{m}(f)(x))_{(m,x)\in U},
\]
surjective?
\end{que}

\begin{rem}
Let $E$ be a Fr\'{e}chet space with increasing system of seminorms $(p_{\alpha_{n}})_{n\in\N}$, $B_{n}:=B_{\alpha_{n}}^{\circ}$ where 
$B_{\alpha_{n}}:=\{x\in E\;|\;p_{\alpha_{n}}(x)<1\}$, 
$(T^{E}_{m},T^{\K}_{m})_{m\in M}$ a strong, consistent generator for $(\mathcal{FV},E)$
and $U$ a set of uniqueness for $(T^{\K}_{m},\mathcal{FV})_{m\in M}$. If 
\begin{enumerate}
\item[(i)] $\FV$ is a BC-space, or if
\item[(ii)] $U$ fixes the topology of $\FV$,
\end{enumerate}
then $\mathcal{FV}_{E'}(U,E)_{sb}=\mathcal{FV}_{E'}(U,E)$ by \prettyref{prop:ext_F_semi_M} in (i) 
resp.\ \prettyref{rem:rest_spaces_coincide} and \prettyref{prop:ext_E_semi_M} in (ii). Hence 
\prettyref{thm:ext_F_semi_M} (i) resp.\ \prettyref{thm:fix_topo_E_semi_M} answers \prettyref{que:surjective_sb_rest_space} in these cases. 
\end{rem}

Let us turn to the case where $G$ need not coincide with $E'$.

\begin{prop}[{\cite[Lemma 9, p.\ 504]{F/J}}]\label{prop:seq_bound}
 Let $E$ be a Fr\'{e}chet space, $(B_{n})$ fix the topology in $E$, $Y$ a Fr\'{e}chet-Schwartz space and 
 $X\subset Y_{b}'(=Y_{\kappa}')$ a dense subspace. Set $G:=\operatorname{span}(\bigcup_{n\in\N} B_{n})$ and let
 $\mathsf{A}\colon X\to E$ be a linear map which is 
 $\sigma(X,Y)$-$\sigma(E,G)$-continuous and satisfies that $\mathsf{A}^{t}(B_{n})$ is bounded in $Y$ 
 for each $n\in\N$. 
 Then $\mathsf{A}$ has a (unique) extension $\widehat{\mathsf{A}}\in Y\varepsilon E$.
\end{prop}

Next, we improve \cite[Theorem 1 ii), p.\ 501]{F/J}.

\begin{thm}\label{thm:ext_FS_set_uni_seq_bounded}
Let $E$ be a Fr\'{e}chet space, $(B_{n})$ fix the topology in $E$ and $G:=\operatorname{span}(\bigcup_{n\in\N} B_{n})$,
$(T^{E}_{m},T^{\K}_{m})_{m\in M}$ a strong, consistent generator for $(\mathcal{FV},E)$, 
$\FV$ a Fr\'{e}chet-Schwartz space and $U$ a set of uniqueness for $(T^{\K}_{m},\mathcal{FV})_{m\in M}$. 
Then the restriction map $R_{U,G}\colon S(\FV\varepsilon E)\to \mathcal{FV}_{G}(U,E)_{sb}$ is surjective.
\end{thm}
\begin{proof}
Let $f\in \mathcal{FV}_{G}(U,E)_{sb}$. We set $X:=\operatorname{span}\{T^{\K}_{m,x}\;|\;(m,x)\in U\}$ and 
$Y:=\FV$. Let $\mathsf{A}\colon X\to E$ be the linear map determined by
$\mathsf{A}(T^{\K}_{m,x}):=f(m,x)$ which is well-defined since $G$ is $\sigma(E',E)$-dense. 
From 
\[
e'(\mathsf{A}(T^{\K}_{m,x}))=(e'\circ f)(m,x)=T^{\K}_{m,x}(f_{e'})
\]
for every $e'\in G$ and $(m,x)\in U$ it follows that $\mathsf{A}$ is $\sigma(X,Y)$-$\sigma(E,G)$-continuous and
\[
\sup_{e'\in B_{n}}|\mathsf{A}^{t}(e')|_{j,k}=\sup_{e'\in B_{n}}|f_{e'}|_{j,k}<\infty
\]
for all $j\in J$, $k\in M$ and $n\in\N$. Due to \prettyref{prop:seq_bound} there is 
an extension $\widehat{\mathsf{A}}\in\FV\varepsilon E$ of $\mathsf{A}$. 
We set $F:=S(\widehat{\mathsf{A}})$ and get for all $(m,x)\in U$ that
\[
T^{E}_{m}(F)(x)=T^{E}_{m}S(\widehat{\mathsf{A}})(x)=\widehat{\mathsf{A}}(T^{\K}_{m,x})=f(m,x)
\]
by consistency, which means $R_{U,G}(F)=f$.
\end{proof}

\begin{cor}
Let $E$ be a Fr\'{e}chet space, $(B_{n})$ fix the topology in $E$ and $G:=\operatorname{span}(\bigcup_{n\in\N} B_{n})$. 
Let $\mathcal{V}^{\infty}$ fulfil \prettyref{cond:weights} and 
$U\subset\R^{d}$ be a set of uniqueness for $(\id_{\K^{\R^{d}}},\mathcal{AV}^{\infty}_{P(\partial)})$ 
where $P(\partial)=\overline{\partial}$ or $\Delta$. 
If $f\colon U\to E$ is a function such that $e'\circ f$ admits an extension 
$f_{e'}\in\mathcal{AV}^{\infty}_{P(\partial)}(\R^{d})$ for each $e'\in G$ 
and $\{f_{e'}\;|\;e'\in B_{n}\}$ is bounded in $\mathcal{AV}^{\infty}_{P(\partial)}(\R^{d})$ for each $n\in\N$, 
then there is a unique extension $F\in\mathcal{AV}^{\infty}_{P(\partial)}(\R^{d},E)$ of $f$.
\end{cor}
\begin{proof}
$\mathcal{AV}^{\infty}_{P(\partial)}(\R^{d})$ is a Fr\'echet-Schwartz space and $(\id_{E^{\R^{d}}},\id_{\K^{\R^{d}}})$ a 
strong, consistent generator for $(\mathcal{AV}^{\infty}_{P(\partial)},E)$ by \prettyref{prop:AV_CV_coincide_hol_har} 
and the proof of \prettyref{cor:weak_strong_CV}. Now, \prettyref{thm:ext_FS_set_uni_seq_bounded} and \prettyref{prop:injectivity} 
prove our statement.
\end{proof}

We already mentioned examples of families of weights $\mathcal{V}^{\infty}$ 
that fulfil \prettyref{cond:weights} and sets of uniqueness for $(\id_{\K^{\R^{d}}},\mathcal{AV}^{\infty}_{P(\partial)})$ 
in \prettyref{rem:ex_NF_cond_weights}, \prettyref{rem:ex_set_uni_AP} and \prettyref{rem:ex_fix_top_AP}. 
If $E$ is a Banach space, then an almost norming set fixes the topology and examples can be found via \prettyref{rem:ex_almost_norming}.
\section{Representation by sequence spaces}
Our last section is dedicated to the representation of weighted spaces of $E$-valued functions by weighted
spaces of $E$-valued sequences if there is a counterpart of this representation in the scalar-valued case 
involving the coefficient functionals associated to a Schauder basis (see \prettyref{rem:Schauder_coeff_set_uni} b)).

\begin{thm}\label{thm:Schauder_coeff_space}
Let $E$ be locally complete, $G\subset E'$ determine boundedness and $\F$ and $\FE$ resp.\ $\ell(\N)$ and $\ell(\N,E)$ 
be $\varepsilon$-into-compatible with $e'\circ g\in\ell(\N)$ for all $e'\in E'$ and $g\in\ell(\N,E)$.
Let $\F$ have an equicontinuous Schauder basis with associated coefficient functionals $(T^{\K}_{n})_{n\in \N}$ such that 
\[
T^{\K}\colon \F\to\ell(\N),\;T^{\K}(f):=(T^{\K}_{n}(f))_{n\in \N},
\]
is an isomorphism and let there be $T^{E}\colon \FE\to E^{\N}$ such that $(T^{E},T^{\K})$ is 
a strong, consistent family for $(\mathcal{F},E)$.
If 
\begin{enumerate}
\item[(i)] $\F$ is a Fr\'{e}chet-Schwartz space, or if
\item[(ii)] $E$ is sequentially complete, $G=E'$ and $\F$ is a semi-Montel BC-space,
\end{enumerate}
then the following holds.
\begin{enumerate}
\item[a)] $\mathcal{F}_{G}(\N,E)=\ell(\N,E)$.
\item[b)] $\ell(\N)$ and $\ell(\N,E)$ are $\varepsilon$-compatible, in particular, $\ell(\N)\varepsilon E\cong\ell(\N,E)$.
\item[c)] The map 
\[
T^{E}\colon \FE\to\ell(\N,E),\;T^{E}(f):=(T^{E}_{n}(f))_{n\in \N},
\]
is a well-defined isomorphism, $\F$ and $\FE$ are $\varepsilon$-compatible, in particular, 
$\F\varepsilon E\cong\FE$, and $T^{E}=S_{\ell(\N)}\circ(T^{\K}\varepsilon\id_{E})\circ S_{\F}^{-1}$.
\end{enumerate}
\end{thm}
\begin{proof}
a)(1) First, we remark that $\N$ is a set of uniqueness for $(T^{\K},\mathcal{F})$. 
Let $u\in\F\varepsilon E$ and $n\in\N$. Then 
\begin{align}\label{eq:Schauder_coeff_space}
 R_{\N,G}(S_{\F}(u))(n)&=(T^{E}\circ S_{\F})(u)(n)=T^{E}_{n}(S_{\F}(u))=u(T^{\K}_{n})=u(\delta_{n}\circ T^{\K})\notag\\
&=(u\circ (T^{\K})^{t})(\delta_{n})
 =(T^{\K}\varepsilon\id_{E})(u)(\delta_{n})\notag\\
&=\bigl(S_{\ell(\N)}\circ (T^{\K}\varepsilon\id_{E})\bigr)(u)(n)
\end{align}
by consistency and the $\varepsilon$-into-compatibility, yielding $\mathcal{F}_{G}(\N,E)\subset \ell(\N,E)$ once we have shown that $R_{\N,G}$ is surjective, which we postpone to part b). 

a)(2) Let $g\in\ell(\N,E)$. Then $e'\circ g\in\ell(\N)$ for all $e'\in E'$ and $g_{e'}:=(T^{\K})^{-1}(e'\circ g)\in\F$. 
We note that $T^{\K}_{n}(g_{e'})=(e'\circ g)(n)$ for all $n\in\N$, which implies $\ell(\N,E)\subset \mathcal{F}_{G}(\N,E)$. 

b) We only need to show that $S_{\ell(\N)}$ is surjective. Let $g\in\ell(\N,E)$, which implies 
$g\in\mathcal{F}_{G}(\N,E)$ by part a)(2). 

We claim that $R_{\N,G}$ is surjective. In case (i) this follows directly from \prettyref{thm:ext_FS_set_uni}. 
Let us turn to case (ii) and denote by $(f_{n})_{n\in\N}$ the equicontinuous Schauder basis of $\F$ 
associated to $(T^{\K}_{n})_{n\in \N}$. We check that condition (ii) of \prettyref{thm:ext_F_semi_M} 
is fulfilled. Let $f'\in\F'$ and set 
\[
f_{k}'\colon \F\to\K,\;f_{k}'(f):=\sum_{n=1}^{k}T^{\K}_{n}(f)f'(f_{n}),
\]
for $k\in\N$. Then $f_{k}'\in\F'$ for every $k\in\N$ and $(f_{k}')$ converges to $f'$ in $\F_{\sigma}'$ 
since $(\sum_{n=1}^{k}T^{\K}_{n}(f)f_{n})$ converges to $f$ in $\F$. From the equicontinuity of the 
Schauder basis we deduce that $(f_{k}')$ converges to $f'$ in $\F_{\kappa}'$ by \cite[8.5.1 Theorem (b), p.\ 156]{Jarchow}. 
Let $f\in\mathcal{F}_{E'}(\N,E)$. For each $e'\in E'$ and $k\in\N$ we have 
\[
\mathscr{R}^{t}_{f}(f_{k}')(e')=f_{k}'(f_{e'})=\sum_{n=1}^{k}T^{\K}_{n}(f_{e'})f'(f_{n})
=e'(\sum_{n=1}^{k}f(n)f'(f_{n}))
\]
since $f\in\mathcal{F}_{E'}(\N,E)$, implying $\mathscr{R}^{t}_{f}(f_{k}')\in\mathcal{J}(E)$. 
Hence we can apply \prettyref{thm:ext_F_semi_M} (ii) and obtain that $R_{\N,E'}$ is surjective, finishing the proof of part a)(1).

Thus there is $u\in\F\varepsilon E$ such that $R_{\N,E'}(S_{\F}(u))=g$ in both cases.
Then $(T^{\K}\varepsilon\id_{E})(u)\in\ell(\N)\varepsilon E$ and from \eqref{eq:Schauder_coeff_space} we derive 
\[
S_{\ell(\N)}((T^{\K}\varepsilon\id_{E})(u))=R_{\N,G}(S_{\F}(u))=g,
\]
proving the surjectivity of $S_{\ell(\N)}$. 

c) First, we note that map $T^{E}$ is well-defined. Indeed, we have $(e'\circ T^{E})(f)=T^{\K}(e'\circ f)\in\ell(\N)$ 
for all $f\in\FE$ and $e'\in E'$ by the strength of the family. Part a) implies that $T^{E}(f)\in\mathcal{F}_{G}(\N,E)=\ell(\N,E)$ 
and thus the map $T^{E}$ is well-defined and its linearity follows from the linearity of the $T^{E}_{n}$ for $n\in\N$.
Next, we prove that $T^{E}$ is surjective. 
Let $g\in\ell(\N,E)$. Since $T^{\K}\varepsilon\id_{E}$ is an isomorphism 
and $S_{\ell(\N)}$ by part b) as well, we obtain that 
$u:=((T^{\K}\varepsilon\id_{E})^{-1}\circ S_{\ell(\N)}^{-1})(g)\in \F\varepsilon E$. 
Therefore $S_{\F}(u)\in\FE$ and from \eqref{eq:Schauder_coeff_space} we get
\[
T^{E}(S_{\F}(u))=(T^{E}\circ S_{\F})(u)=\bigl(S_{\ell(\N)}\circ (T^{\K}\varepsilon\id_{E})\bigr)(u)=g,
\]
which means that $T^{E}$ is surjective. The injectivity of $T^{E}$ by \prettyref{prop:injectivity}, implies that 
\[
S_{\F}=(T^{E})^{-1}\circ\bigl(S_{\ell(\N)}\circ (T^{\K}\varepsilon\id_{E})\bigr),
\]
yielding the surjectivity of $S_{\F}$ and thus the $\varepsilon$-compatibility of $\F$ and $\FE$. 
Furthermore, we have $T^{E}=S_{\ell(\N)}\circ (T^{\K}\varepsilon\id_{E})\circ S_{\F}^{-1}$, resulting in 
$T^{E}$ being an isomorphism. 
\end{proof}

We note that a Schauder basis of $\F$ is already equicontinuous by the uniform boundedness principle 
if $\F$ is barrelled. Further, the index set of the equicontinuous Schauder basis of $\F$ 
in \prettyref{thm:Schauder_coeff_space} need not be $\N$ but may be any other countable index set 
as long as the equicontinuous Schauder basis is unconditional which is always fulfilled 
if $\F$ is nuclear by \cite[21.10.1 Dynin-Mitiagin Theorem, p.\ 510]{Jarchow}.

Let us demonstrate an application of the preceding theorem to Fourier expansions of vector-valued $2\pi$-periodic smooth functions 
and the multiplier space of the Schwartz space. 
We equip the space $\mathcal{C}^{\infty}(\R^{d},E)$ for an lcHs $E$ 
with the system of seminorms generated by 
\[
|f|_{K,m,\alpha}:=\sup_{\substack{x\in \R^{d}\\ \beta\in\N_{0}^{d},|\beta|\leq m}}
p_{\alpha}((\partial^{\beta})^{E}f(x))\chi_{K}(x),\quad f\in\mathcal{C}^{\infty}(\R^{d},E),
\]
for $K\subset\R^{d}$ compact, $m\in\N_{0}$ and $\alpha\in\mathfrak{A}$, i.e.\ we consider $\mathcal{CW}^{\infty}(\R^{d},E)$.
By $\mathcal{C}^{\infty}_{2\pi}(\R^{d},E)$ we denote its topological subspace consisting of the functions 
which are $2\pi$-periodic in each variable. 
If $E$ is a locally complete lcHs over $\C$, then the function given by $x\mapsto f(x)e^{-i\langle n,x\rangle_{\R^d}}$ 
is Pettis-integrable on $[-\pi,\pi]^{d}$ 
for every $f\in\mathcal{C}^{\infty}_{2\pi}(\R^{d},E)$ and $n\in\Z^{d}$ by \cite[Lemma 4.7, p.\ 369]{kruse2018_1} 
where $\langle\cdot,\cdot\rangle_{\R^d}$ is the usual scalar product on $\R^{d}$. 
Hence we are able to define the $n$-th Fourier coefficient of $f\in\mathcal{C}^{\infty}_{2\pi}(\R^{d},E)$ by the Pettis-integral
\[
\widehat{f}(n):=(2\pi)^{-d}\int_{[-\pi,\pi]^{d}}f(x)e^{-i\langle n,x\rangle_{\R^d}}\d x, \quad n\in\Z^{d},
\]
if $E$ is locally complete. Our aim is to prove that the map $f\mapsto (\widehat{f}(n))_{n\in\Z^{d}}$ is an isomorphism 
from $\mathcal{C}^{\infty}_{2\pi}(\R^{d},E)$ to the space $s(\Z^{d},E)$ of rapidely decreasing $E$-valued sequences 
given by
\[
s(\Omega,E):=\{x=(x_{n})\in E^{\Omega}\;|\;\forall\;j\in\N,\,\alpha\in\mathfrak{A}:\; 
|x|_{j,\alpha}:=\sup_{n\in\Omega}p_{\alpha}(x_{n})(1+|n|^{2})^{j/2}<\infty\},
\]
where $\Omega=\Z^{d}$ or $\N_{0}^{d}$.

\begin{cor}\label{cor:Schauder_coeff_space_2pi_per}
If $E$ is a locally complete lcHs over $\C$, then 
$\mathcal{C}^{\infty}_{2\pi}(\R^{d},E)\cong \mathcal{C}^{\infty}_{2\pi}(\R^{d})\varepsilon E$ and
$s(\Z^{d},E)\cong s(\Z^d)\varepsilon E$
and the map 
 \[
  \mathfrak{F}^{E}\colon \mathcal{C}^{\infty}_{2\pi}(\R^{d},E)\to s(\Z^d,E),\;\mathfrak{F}^{E}(f):=(\widehat{f}_{n})_{n\in\Z^{d}},
 \]
is an isomorphism with
$\mathfrak{F}^{E}=S_{s(\Z^{d})}\circ(\mathfrak{F}^{\C}\varepsilon\id_{E})\circ S_{\mathcal{C}^{\infty}_{2\pi}(\R^{d})}^{-1}$.
\end{cor}
\begin{proof}
By the proof of \cite[Theorem 4.10, p.\ 373]{kruse2018_1} the spaces $\mathcal{C}^{\infty}_{2\pi}(\R^{d})$ and
$\mathcal{C}^{\infty}_{2\pi}(\R^{d},E)$ are $\varepsilon$-compatible. Moreover, the spaces $s(\Z^{d})$ and $s(\Z^d,E)$ 
are $\varepsilon$-into-compatible by \prettyref{thm:S_iso_into} and it is obvious that 
$e'\circ x\in s(\Z^{d})$ for every $e'\in E'$ and $x\in s(\Z^d,E)$. 
The space $\mathcal{C}^{\infty}_{2\pi}(\R^{d})$ is a nuclear Fr\'echet space, in particular, barrelled and 
thus its Schauder basis $(e^{-i\langle n,\cdot\rangle_{\R^d}})$ is unconditional and equicontinuous. 
The corresponding coefficient 
functionals are given by $(\delta_{n}\circ \mathfrak{F}^{\C})$ and the map $\mathfrak{F}^{\C}$ is an isomorphism 
(see e.g. \cite[Satz 1.7, p.\ 18]{Kaballo}). Again, we derive from the proof of \cite[Theorem 4.10, p.\ 373]{kruse2018_1} 
that the family $(\mathfrak{F}^{E},\mathfrak{F}^{\C})$ is strong and consistent 
for $(\mathcal{C}^{\infty}_{2\pi},E)$. Now, we can apply \prettyref{thm:Schauder_coeff_space}, yielding our statement. 
\end{proof}

The preceding corollary improves a special case of \cite[Theorem 4.2, p.\ 364]{kruse2018_1} and 
\cite[Theorem 4.11, p.\ 375]{kruse2018_1} from sequentially complete $E$ to locally complete $E$. 
In the same way we can prove the corresponding result for the Schwartz space $\mathcal{S}(\R^{d},E)$ and $s(\N_{0}^{d},E)$ 
with sequentially complete $E$ which is given in \cite[Theorem 4.9 a), p.\ 371]{kruse2018_1} by a different proof. 
For the space $\mathcal{CW}^{\infty}_{\overline{\partial}}(\D_{R}(0),E)$, $0<R\leq\infty$, 
of holomorphic functions and the K\"othe space $\lambda^{\infty}(A_{R},E)$ with K\"othe matrix $A_{R}:=(r_{k}^{j})_{j\in\N_{0},k\in\N}$ 
for some strictly increasing sequence $(r_{k})_{k\in\N}$ in $(0,R)$ converging to $R$ and locally complete $E$ 
a corresponding statement may be proved using \cite[Example 27.27, p.\ 341-342]{meisevogt1997} 
where the map $T^{E}$ assigns to each holomorphic function on $\D_{R}(0)$ its sequence of Taylor coefficients.

Let us turn to the space of multipliers for the Schwartz space defined by 
\[
\qquad\quad \mathcal{O}_{M}(\R^{d},E)
:=\{f\in \mathcal{C}^{\infty}(\R^{d},E)\;|\;\forall\;g\in\mathcal{S}(\R^{d}),\,
m\in\N_{0},\,\alpha\in \mathfrak{A}:\;\|f\|_{g,m,\alpha}<\infty\}
\]
where
\[
 \|f\|_{g,m,\alpha}:=\sup_{\substack{x\in\mathbb{R}^{d}\\
  \beta\in\N_{0}^{d},|\beta|\leq m}}
 p_{\alpha}\bigl((\partial^{\beta})^{E}f(x)\bigr)|g(x)|
\]
(see \cite[$4^{0}$), p.\ 97]{Schwartz1955}). For simplicity we restrict to the case $d=1$.
Fix a compactly supported test function $\varphi\in\mathcal{C}^{\infty}_{c}(\R)$ with $\varphi(x)=1$ for $x\in[0,\frac{1}{4}]$ 
and $\varphi(x)=0$ for $x\geq\frac{1}{2}$.   
For $f\in\mathcal{C}^{\infty}(\R,E)$ we set 
\[
f_{j}(x):=f(x+j)-\sum_{k=0}^{\infty}a_{k}\varphi(-2^k(x-1))f(-2^k(x-1+j)+1) ,\;x\in[0,1],\,j\in\Z,
\]
where 
\[
 a_{k}:=\prod_{j=0,j\neq k}^{\infty}\frac{1+2^j}{2^j-2^k},\;k\in\N_{0}.
\]
Fixing $x\in[0,1)$, we observe that $f_{j}(x)$ is well-defined for each $j\in\Z$ since there are only finitely many summands 
due to the compact support of $\varphi$ and $-2^k(x-1)\to\infty$ for $k\to\infty$. For $x=1$ we have $f_{j}(1)=0$ for each $j$ 
and the convergence of the series in $E$ follows from the uniform continuity of $f$ on $[0,1]$, $f(0)=0$ 
and $\sum_{k=0}^{\infty}a_{k}=1$ by the case $n=0$ in \cite[Lemma (iii), p.\ 625]{seeley1964}. 
For each $e'\in E'$ we note that 
\[
e'(f_{j}(x))=(e'\circ f)(x+j)-\sum_{k=0}^{\infty}a_{k}\varphi(-2^k(x-1))(e'\circ f)(-2^k(x-1+j)+1),\,x\in[0,1],j\in\Z,
\]
which implies that $e'\circ f_{j}\in \mathcal{E}_{0}$ by \cite[Proposition 3.2, p.\ 15]{bargetz2014}. 
Using the weak-strong principle \prettyref{cor:weak_strong_E_0}, we obtain 
that $f_{j}\in\mathcal{E}_{0}(E)$ for all $j\in\Z$ if $E$ is locally complete. Setting 
\[
 \rho\colon\R\to[0,1],\;\rho(x):=1-\cos(\arctan(x))=1-\frac{1}{\sqrt{1+x^2}},
\]
we deduce from the proof and with the notation of \cite[Proposition 2.2, p.\ 1494]{bargetz2015} that 
$e'\circ f_{j}\circ\rho=(\Phi_{2}^{-1}\circ\Phi_{1})(e'\circ f_{j})$ is an element of the Schwartz space $\mathcal{S}(\R)$ 
for each $e'\in E'$. The weak-strong principle \prettyref{cor:weak_strong_CV} c) yields that $f_{j}\circ\rho\in\mathcal{S}(\R,E)$ 
if $E$ is locally complete. Hence $(f_{j}\circ\rho)\cdot h_{2n}$ is Pettis-integrable on $\R$  
for every $j\in\Z$ and $n\in\N_{0}$ by \cite[Proposition 4.8, p.\ 370]{kruse2018_1}  
if $E$ is sequentially complete where 
\[
h_{n}\colon\R\to\R,\;
h_{n}(x):=(2^{n}n!\sqrt{\pi})^{-1/2}(x-\frac{d}{dx})^{n}e^{-x^{2}/2},
\]
is the $n$-th Hermite function. Therefore the Pettis-integral
\[
 b_{n,j}(f):=\langle f_{j}\circ\rho, h_{2n}\rangle_{L^{2}}:=\int_{\R}f_{j}(\rho(x))h_{2n}(x)\d x,\;j\in\Z,\,n\in\N_{0},
\]
is a well-defined element of $E$ if $E$ is sequentially complete. 
By \cite[Theorem 2.1, p.\ 1496-1497]{bargetz2015} (cf.\ \cite[Theorem 3, p.\ 478]{valdivia1981}) the map 
 \[
  \Theta^{\K}\colon \mathcal{O}_{M}(\R)\to s(\N)_{b}'\widehat{\otimes}_{\pi}s(\N),\;\Theta^{\K}(f):=(b_{\sigma(n,j)}(f))_{(n,j)\in\N^2},
 \]
is an isomorphism where $\sigma\colon\N^{2}\to\N_{0}\times\Z$ is the enumeration given 
by $\sigma(n,j):=(n-1,(j-1)/2)$ if $j$ is odd, and $\sigma(n,j):=(n-1,-j/2)$ if $j$ is even. 
Here, we have to interpret $\Theta^{\K}(f)$ as an element of $s(\N)_{b}'\widehat{\otimes}_{\pi}s(\N)$ by identification 
of isomorphic spaces. Namely, 
\[
s(\N)_{b}'\widehat{\otimes}_{\pi}s(\N)\cong s(\N)\widehat{\otimes}_{\pi}s(\N)_{b}'
\cong s(\N) \varepsilon s(\N)_{b}'\cong s(\N,s(\N)_{b}')
\]
holds where the first isomorphy is due to the commutativity of $\widehat{\otimes}_{\pi}$, the second due to 
the nuclearity of $s(\N)$ and the last due to \cite[Theorem 4.2, p.\ 364]{kruse2018_1} via $S_{s(\N)}$. 
Then we intrepret $\Theta^{\K}(f)$ as an element of $s(\N,s(\N)_{b}')$ by means of
\[
j\in\N \longmapsto [a\in s(\N)\mapsto \sum_{n\in\N}a_{n}b_{\sigma(n,j)}]
\]
(see also \eqref{eq:Schauder_coeff_space_multiplier_1} below).

\begin{cor}\label{cor:Schauder_coeff_space_multiplier}
If $E$ is a sequentially complete lcHs, then $\mathcal{O}_{M}(\R,E)\cong \mathcal{O}_{M}(\R)\varepsilon E$ and the map 
 \[
 \Theta^{E}\colon \mathcal{O}_{M}(\R,E)\to s(\N,L_{b}(s(\N),E)),\;\Theta^{E}(f):=(b_{\sigma(n,j)}(f))_{(n,j)\in\N^2},
 \]
is an isomorphism where we interpret $\Theta^{E}(f)$ as an element of $s(\N,L_{b}(s(\N),E))$.
\end{cor}
\begin{proof}
The spaces $\mathcal{O}_{M}(\R)$ and $\mathcal{O}_{M}(\R,E)$ are $\varepsilon$-into-compatible 
by \prettyref{prop:S_iso_into_standard_example} e) as $\mathcal{O}_{M}(\R)$ is a complete barrelled nuclear space, 
in particular a Montel space, by \cite[Chap.\ II, \S4, n$^\circ$4, Th\'{e}or\`{e}me 16, p.\ 131]{Gro}. 

Next, we show that $S_{\mathcal{O}_{M}(\R)}$ is surjective. We only need to prove that condition (ii) of 
\prettyref{thm:ext_F_semi_M} is fulfilled with $\mathcal{F}_{E'}(U,E)$ replaced by $\mathcal{O}_{M}(\R,E)$,
which is \cite[3.16 Condition c)]{kruse2017a}. Then 
we may apply \cite[3.17 Theorem, p.\ 12]{kruse2017a} to obtain the surjectivity. 
Let $f'\in\mathcal{O}_{M}(\R)'$. Using the equicontinuous unconditional Schauder basis 
$(\psi_{\sigma(i,j)})_{(i,j)\in\N^{2}}$ with associated coefficients functionals 
$\delta_{i,j}\circ\Theta^{\K}=b_{\sigma(i,j)}$ given in \cite[Proposition 3.2, p.\ 1499]{bargetz2015}, we set for $n\in\N$
\[
f_{n}'\colon \mathcal{O}_{M}(\R)\to\K,\;f_{n}'(f):=\sum_{(i,j)\in\N^2,\,|(i,j)|\leq n}b_{\sigma(i,j)}(f)f'(\psi_{\sigma(i,j)}).
\]
Along the lines of the proof of \prettyref{thm:Schauder_coeff_space} b)(ii)
we derive that $(f_{n}')$ converges to $f'$ in $\mathcal{O}_{M}(\R)_{\kappa}'$. 
Let $f\in\mathcal{O}_{M}(\R,E)$. For each $e'\in E'$ and $(i,j)\in\N^2$ we have
\begin{align}\label{eq:Schauder_coeff_space_multiplier}
  \delta_{i,j}\circ\Theta^{\K}(e'\circ f)
&=b_{\sigma(i,j)}(e'\circ f)
 =\int_{\R}(e'\circ f)_{(j-1)/2}(\rho(x))h_{2(i-1)}(x)\d x \notag\\
&=\langle e',\int_{\R}f_{(j-1)/2}(\rho(x))h_{2(i-1)}(x)\d x\rangle
 =\langle e',\delta_{i,j}\circ\Theta^{E}(f)\rangle\notag\\
&=e'(b_{\sigma(i,j)}(f))
\end{align}
if $j$ is odd since $(f_{(j-1)/2}\circ\rho)\cdot h_{2(i-1)}$ is Pettis-integrable on $\R$. 
The analogous result holds for even $j$ as well. This implies
\begin{align*}
\mathscr{R}^{t}_{f}(f_{n}')(e')&=f_{n}'(e'\circ f)=\sum_{(i,j)\in\N^2,\, |(i,j)|\leq n}b_{\sigma(i,j)}(e'\circ f)f'(\psi_{\sigma(i,j)})\\
&=e'\bigl(\sum_{(i,j)\in\N^2,\, |(i,j)|\leq n}b_{\sigma(i,j)}(f)f'(\psi_{\sigma(i,j)})\bigr),
\end{align*}
yielding $\mathscr{R}^{t}_{f}(f_{n}')\in\mathcal{J}(E)$. 
This shows that \cite[3.16 Condition c)]{kruse2017a} is fulfilled and the surjectivity of $S_{\mathcal{O}_{M}(\R)}$ 
is a consequence of \cite[3.17 Theorem, p.\ 12]{kruse2017a} in combination with \cite[Lemma 13 a), p.\ 1523]{kruse2017} 
as $\mathcal{O}_{M}(\R)$ is a Montel space.  

Further, we deduce from \eqref{eq:Schauder_coeff_space_multiplier} that 
$(\Theta^{E},\Theta^{\K})$ is a strong family for $(\mathcal{O}_{M},E)$. 
By \cite[3.17 Theorem, p.\ 12]{kruse2017a} the inverse of $S_{\mathcal{O}_{M}(\R)}$ is given by the map 
$f\mapsto \mathcal{J}^{-1}\circ \mathscr{R}^{t}_{f}$. Let $u\in\mathcal{O}_{M}(\R)\varepsilon E$. Then 
$f:=S_{\mathcal{O}_{M}(\R)}(u)\in\mathcal{O}_{M}(\R,E)$ and for each $(i,j)\in\N^{2}$
we get
\begin{align*}
u(\delta_{i,j}\circ\Theta^{\K})&=S_{\mathcal{O}_{M}(\R)}^{-1}(f)(\delta_{i,j}\circ\Theta^{\K})
= \mathcal{J}^{-1}\bigl(\mathscr{R}^{t}_{f}(\delta_{i,j}\circ\Theta^{\K})\bigr)
 \underset{\eqref{eq:Schauder_coeff_space_multiplier}}{=}b_{\sigma(i,j)}(f)\\
&=(\delta_{i,j}\circ\Theta^{E})(S_{\mathcal{O}_{M}(\R)}(u)),
\end{align*}
implying the consistency of our family. 

In order to apply \prettyref{thm:Schauder_coeff_space} we need spaces $\ell\mathcal{V}(\N^2)$ and $\ell\mathcal{V}(\N^2,E)$ of 
sequences with values in $\K$ and $E$, respectively. In addition, the space $\ell\mathcal{V}(\N^2)$ has to be isomorphic 
to $s(\N,s(\N)_{b}')$ so that $\Theta^{\K}\colon\mathcal{O}_{M}(\R)\to s(\N,s(\N)_{b}')\cong\ell\mathcal{V}(\N^2)$ 
becomes the isomorphism we need for \prettyref{thm:Schauder_coeff_space}. We set
\[
\ell\mathcal{V}(\N^2,E):=\{x=(x_{n,j})\in E^{\N^2}\;|\;\forall\;k\in\N,\,B\subset s(\N)\;\text{bounded},\,\alpha\in\mathfrak{A}:\;
\|x\|_{k,B,\alpha}<\infty\}
\]
where 
\[
\|x\|_{k,B,\alpha}:=\sup_{(j,a)\in\omega_{B}}p_{\alpha}(T^{E}(x)(j,a))\nu_{k,B}(j,a)
\]
with $\omega_{B}:=\N\times B$ and $\nu_{k,B}\colon \omega_{B}\to [0,\infty)$, $\nu_{k,B}(j,a):=(1+j^2)^{k/2}$, plus 
\[
T^{E}(x)(j,a):=\sum_{n\in\N}a_{n}x_{n,j}.
\]
We claim that the map 
\begin{equation}\label{eq:Schauder_coeff_space_multiplier_1}
T^{E}\colon\ell\mathcal{V}(\N^2,E)\to s(\N,L_{b}(s(\N),E)),\; x\mapsto (T^{E}(x)(j,\cdot))_{j\in\N},
\end{equation}
is an isomorphism. We remark for each $k\in\N$, bounded $B\subset s(\N)$ and $\alpha\in\mathfrak{A}$ that
\[
|T^{E}(x)|_{k,(B,\alpha)}=\sup_{j\in\N}\sup_{a\in B}p_{\alpha}(T^{E}(x)(j,a))(1+j^2)^{k/2}=\|x\|_{k,B,\alpha}
\]
for all $x\in\ell\mathcal{V}(\N^2,E)$, implying that $T^{E}$ is an isomorphism into.  
Let $y:=(y_{j})\in s(\N,L_{b}(s(\N),E))$. Then $y_{j}\in L_{b}(s(\N),E)$ for $j\in\N$ and 
we set $x_{n,j}:=y_{j}(e_{n})$ for $n\in\N$ where $e_{n}$ is the $n$-th unit sequence in $s(\N)$. We note 
that with $x:=(x_{n,j})_{(n,j)\in\N^2}$ 
\[
T^{E}(x)(j,a)=\sum_{n\in\N}a_{n}x_{n,j}=\sum_{n\in\N}a_{n}y_{j}(e_{n})=y_{j}(\sum_{n\in\N}a_{n}e_{n})=y_{j}(a)
\]
holds for all $j\in\N$ and $a:=(a_{n})\in s(\N)$ since $(e_{n})$ is a Schauder basis of $s(\N)$ 
with associated coefficient functionals $a\mapsto a_{n}$. It follows that $x\in\ell\mathcal{V}(\N^2,E)$ 
and the surjectivity of $T^{E}$. 

The next step is to prove that $\ell\mathcal{V}(\N^2)$ and $\ell\mathcal{V}(\N^2,E)$ are $\varepsilon$-into-compatible. 
Due to \prettyref{thm:S_iso_into} we only need to show that $(T^{E},T^{\K})$ is a consistent generator 
for $(\ell\mathcal{V},E)$. Let $u\in\ell\mathcal{V}(\N^2)\varepsilon E$. Then 
\begin{equation}\label{eq:Schauder_coeff_space_multiplier_2}
\sum_{n=1}^{m}a_{n}S_{\ell\mathcal{V}(\N^{2})}(u)(j,n)=\sum_{n=1}^{m}a_{n}u(\delta_{j,n})=u(\sum_{n=1}^{m}a_{n}\delta_{j,n})
\end{equation}
for all $m\in\N$ and $a:=(a_{n})\in s(\N)$. Since 
\[
(\sum_{n=1}^{m}a_{n}\delta_{j,n})(x)=\sum_{n=1}^{m}a_{n}x_{j,n}\to T^{\K}(x)(j,a)=T^{\K}_{(j,a)}(x),\quad m\to\infty,
\]
for all $x\in\ell\mathcal{V}(\N^2)$, we deduce that $(\sum_{n=1}^{m}a_{n}\delta_{j,n})_{m}$ converges 
to $T^{\K}_{(j,a)}(x)$ in $\ell\mathcal{V}(\N^2)_{\kappa}'$ by the Banach-Steinhaus theorem, 
which is applicable as $\ell\mathcal{V}(\N^2)\cong s(\N,s(\N)_{b}')\cong\mathcal{O}_{M}(\R)$ is barrelled. 
We conclude that 
\[
 u(T^{\K}_{(j,a)})
=\lim_{m\to\infty}u(\sum_{n=1}^{m}a_{n}\delta_{j,n})
\underset{\eqref{eq:Schauder_coeff_space_multiplier_2}}{=}\sum_{n=1}^{\infty}a_{n}S_{\ell\mathcal{V}(\N^{2})}(u)(j,n)
=T^{E}S_{\ell\mathcal{V}(\N^{2})}(u)(j,a)
\]
and thus the consistency of $(T^{E},T^{\K})$. 

Furthermore, we clearly have $e'\circ x\in\ell\mathcal{V}(\N^2)$ for all $x\in\ell\mathcal{V}(\N^2,E)$ and the map 
$\Theta\colon \mathcal{O}_{M}(\R)\to s(\N)_{b}'\widehat{\otimes}_{\pi}s(\N)\cong\ell\mathcal{V}(\N^2)$ 
is an isomorphism by \cite[Theorem 2.1, p.\ 1496-1497]{bargetz2015} and \eqref{eq:Schauder_coeff_space_multiplier_1}. 
Due to \cite[Chap.\ II, \S4, n$^\circ$4, Th\'{e}or\`{e}me 16, p.\ 131]{Gro} the dual $\mathcal{O}_{M}(\R)_{b}'$ is an 
LF-space and thus $\mathcal{O}_{M}(\R)\cong(\mathcal{O}_{M}(\R)_{b}')_{b}'$ is the strong dual of an LF-space by reflexivity
and therefore webbed by \cite[Satz 7.25, p.\ 165]{Kaballo}. Finally, we can apply \prettyref{thm:Schauder_coeff_space} (ii), 
yielding our statement. 
\end{proof}

\begin{rem}
The actual isomorphism in \prettyref{cor:Schauder_coeff_space_multiplier} (without the interpretation) is given by
$\widetilde{\Theta}^{E}:=T^{E}\circ\Theta^{E}$ with $T^{E}$ from \eqref{eq:Schauder_coeff_space_multiplier_1} and we have
\[
 \widetilde{\Theta}^{E}=T^{E}\circ\Theta^{E}
=T^{E}\circ S_{\ell\mathcal{V}(\N^{2})}\circ (\Theta^{\K}\varepsilon\id_{E})\circ S^{-1}_{\mathcal{O}_{M}(\R)}.
\]
\end{rem}

For quasi-complete $E$ the $\varepsilon$-compatibility $\mathcal{O}_{M}(\R^{d},E)\cong \mathcal{O}_{M}(\R^{d})\varepsilon E$ 
is already contained in \cite[Proposition 9, p.\ 108, Th\'{e}or\`{e}me 1, p.\ 111]{Schwartz1955}.
\bibliography{biblio}
\bibliographystyle{plainnat}
\end{document}